\documentclass[11pt]{article}
\usepackage[utf8]{inputenc}
\usepackage{tikz}
\usepackage{amsthm}
\usepackage{amsmath,stmaryrd, mathrsfs}
\usepackage{amsfonts}
\usepackage{amssymb}
\usepackage{dsfont}
\usepackage{tikz-cd}
\usepackage{url}
\usepackage{hyperref}
\usepackage{booktabs}
\usepackage{import}
\usepackage{enumitem}
\usepackage[normalem]{ulem}
\usepackage{soul}

\usepackage{upgreek}

\usepackage{geometry}

\usetikzlibrary{shapes.misc}
\usetikzlibrary{shapes.symbols}
\usetikzlibrary{snakes}
\usetikzlibrary{decorations}
\usetikzlibrary{decorations.markings}

\newcommand{\RR}{\mathbf{R}}
\newcommand{\NN}{\mathbf{N}}
\newcommand{\ZZ}{\mathbf{Z}}

\newcommand{\EE}{\mathbb{E}}

\newcommand{\PP}{\mathbb{P}}

\newcommand{\mL}{\mathcal{L}}

\newcommand{\mC}{\mathcal{C}}

\newcommand{\mT}{\mathcal{T}}
\newcommand{\mI}{\mathcal{I}}

\newcommand{\mY}{\mathcal{Y}}

\newcommand{\mE}{\mathcal{E}}

\newcommand{\mU}{\mathcal{U}}
\newcommand{\mV}{\mathcal{V}}

\newcommand{\mK}{\mathcal{K}}

\newcommand{\mf}[1]{\mathfrak{#1}}

\newcommand{\ve}{\varepsilon}

\newcommand*{\ud}{\mathrm{\,d}}
\newcommand{\red}[1]{{ \color{red} #1 }}

\def\dash{\leavevmode\unskip\kern0.18em--\penalty\exhyphenpenalty\kern0.18em}

\def\drawx{\draw[-,solid, line width = 1.0mm] (-3pt,-3pt) -- (3pt,3pt);\draw[-,solid, line width = 1.0mm] (-3pt,3pt) -- (3pt,-3pt);}
\tikzset{
	dot/.style={circle,fill=black,draw=black, solid,inner sep=0pt,minimum size=0.95mm},
	root/.style={circle, thick, fill=white,draw=purple, solid,inner sep=0pt,minimum size=0.95mm},
	idot/.style={circle, thick, fill=white,draw=purple, solid,inner sep=0pt,minimum size=0.95mm},
	square/.style={rectangle,fill=black,draw=black, solid,inner sep=0pt,minimum size=1mm},
	empty/.style={circle,fill=white,draw=white, solid,inner sep=0pt,minimum size=0.5mm},
	var/.style={circle,fill=black!10,draw=black,inner sep=0pt, minimum size=
	2mm},
	symb/.style={circle,fill=symbols,draw=symbols, solid,inner sep=0pt,minimum size=0.5mm},
	yy/.style={circle,fill=gray!20,draw=black,inner sep=0pt,minimum size=0.8mm},
	>=stealth,
	dotred/.style={circle,fill=black!50,inner sep=0pt, minimum size=2mm},
	generic/.style={semithick,shorten >=1pt,shorten <=1pt},
	dist/.style={ultra thick,draw=testcolor,shorten >=1pt,shorten <=1pt},
	testfcn/.style={ultra thick,testcolor,shorten >=1pt,shorten <=1pt,<-},
	testfcnx/.style={ultra thick,testcolor,shorten >=1pt,shorten <=1pt,<-,
		postaction={decorate,decoration={markings,mark=at position 0.6 with {\drawx}}}},
	kprime/.style={semithick,shorten >=1pt,shorten <=1pt,densely dashed,->},
	kprimex/.style={semithick,shorten >=1pt,shorten <=1pt,densely dashed,->,
		postaction={decorate,decoration={markings,mark=at position 0.4 with {\drawx}}}},
	kernel/.style={semithick,shorten >=1pt,shorten <=1pt,->},
	multx/.style={shorten >=1pt,shorten <=1pt,
		postaction={decorate,decoration={markings,mark=at position 0.5 with {\drawx}}}},
	kernelx/.style={semithick,shorten >=1pt,shorten <=1pt,->,
		postaction={decorate,decoration={markings,mark=at position 0.4 with {\drawx}}}},
	kernel1/.style={->,semithick,shorten >=1pt,shorten <=1pt,postaction={decorate,decoration={markings,mark=at position 0.45 with {\draw[-] (0,-0.1) -- (0,0.1);}}}},
	kernel2/.style={->,semithick,shorten >=1pt,shorten <=1pt,postaction={decorate,decoration={markings,mark=at position 0.45 with {\draw[-] (0.05,-0.1) -- (0.05,0.1);\draw[-] (-0.05,-0.1) -- (-0.05,0.1);}}}},
	kernelBig/.style={semithick,shorten >=1pt,shorten <=1pt,decorate, decoration={zigzag,amplitude=1.5pt,segment length = 3pt,pre length=2pt,post length=2pt}},
	rho/.style={dotted,semithick,shorten >=1pt,shorten <=1pt},
	renorm/.style={shape=circle,fill=white,inner sep=1pt},
	labl/.style={shape=rectangle,fill=white,inner sep=1pt},
	xi/.style={circle,fill=symbols!10,draw=symbols,inner sep=0pt,minimum size=1.2mm},
	xix/.style={crosscircle,fill=symbols!10,draw=symbols,inner sep=0pt,minimum size=1.2mm},
	xib/.style={circle,fill=symbols!10,draw=symbols,inner sep=0pt,minimum size=1.6mm},
	xibx/.style={crosscircle,fill=symbols!10,draw=symbols,inner sep=0pt,minimum size=1.6mm},
	not/.style={circle,fill=symbols,draw=symbols,inner sep=0pt,minimum size=0.5mm},
	>=stealth,
	}

\makeatletter
\def\DeclareSymbol#1#2#3{\expandafter\gdef\csname MH@symb@#1\endcsname{\tikz[baseline=#2,scale=0.3]{#3}}
\expandafter\gdef\csname
MH@symb@#1s\endcsname{\scalebox{0.6}{\tikz[baseline=#2,scale=0.3]{#3}}}}
\def\<#1>{\csname MH@symb@#1\endcsname}
\makeatother

\DeclareSymbol{0}{0}{ \draw  (0,0.34) node[dot] {};}
\DeclareSymbol{r0}{0}{ \draw  (0,0.34) node[dot, red] {};}
\DeclareSymbol{1}{0}{\draw[white] (-.4,0) -- (.4,0); \draw (0,0)  -- (0,1) node[dot] {};}
\DeclareSymbol{10}{0}{
\node[root] (r) at (0,0) {};
\draw (r)  -- (0,1) node[dot] {}  ;}
\DeclareSymbol{11}{0}{\draw[white] (-.4,0) -- (.4,0); \draw (0,0) node[dot] {}  -- (0,1) node[dot] {} -- (0,2) node[dot] {};}
\DeclareSymbol{2}{0}{
\node[root] (r) at (0,0) {};
\draw  (r)   -- (-.7,1) node[dot] {}   ;
\draw (r) -- (.7,1) node[dot] {}  ;}

\DeclareSymbol{3}{0}{
\node[root] (r) at (0, 0) {};
\draw (-.7,1) node[dot] {}   --  (r);
\draw (r) -- (.7,1) node[dot] {}  ;
\draw (r) -- (0,1) node[dot] {}  ;}

\DeclareSymbol{1r}{0}{
\node[idot] (n1) at (0,0) {};
\draw (n1) -- (0,1) node[dot] {};}

\DeclareSymbol{2r}{0}{
\node[idot] (n1) at (0,0) {};
\draw (n1) -- (-.7,1) node[dot] {}   ;
\draw (n1) -- (.7,1) node[dot] {}  ;}

\DeclareSymbol{3r}{0}{
\node[idot] (n1) at (0,0) {};
\draw (-.7,1) node[dot] {}   -- (n1);
\draw (n1) -- (.7,1) node[dot] {}  ;
\draw (n1) -- (0,1) node[dot] {}  ;}

\DeclareSymbol{10pair10}{0}{
\node[root] (r) at (0,0) {};
\draw (r)  -- (0,1) node[dot] {}  ;
\node[root] (r2) at (2,0) {};
\draw (r2)  -- (2,1) node[dot] {}  ;
\draw[purple, thick] (0,1) to[out=90,in=180] (1,1.5);
\draw[purple, thick] (1,1.5) to[out=0,in=90] (2,1);
}

\DeclareSymbol{cycle2}{0}{
\draw  (0,0) node[idot] {}   -- (-.7,1) node[dot] {}  ;
\draw (0,0) -- (.7,1) node[dot] {};
\draw  (3,0) node[idot] {}   -- (2.3,1) node[dot] {}  ;
\draw (3,0) -- (3.7,1) node[dot] {};
\draw[purple, thick] (-.7,1) to[out=90,in=180] (1.5,2.5);
\draw[purple, thick] (1.5,2.5) to[out=0,in=90] (3.7,1);
\draw[purple, thick] (0.7,1) to[out=90,in=180] (1.45,1.5);
\draw[purple, thick] (1.45,1.5) to[out=0,in=90] (2.3,1);
}

\DeclareSymbol{cycle3}{0}{
\draw  (0,0) node[idot] {}   -- (-.7,1) node[dot] {}  ;
\draw (0,0) -- (.7,1) node[dot] {};
\draw  (3,0) node[idot] {}   -- (2.3,1) node[dot] {}  ;
\draw (3,0) -- (3.7,1) node[dot] {};
\draw  (6,0) node[idot] {}   -- (5.3,1) node[dot] {}  ;
\draw (6,0) -- (6.7,1) node[dot] {};
\draw[purple, thick] (-.7,1) to[out=90,in=180] (3,2.5);
\draw[purple, thick] (3,2.5) to[out=0,in=90] (6.7,1);
\draw[purple, thick] (0.7,1) to[out=90,in=180] (1.45,1.5);
\draw[purple, thick] (1.45,1.5) to[out=0,in=90] (2.3,1);
\draw[purple, thick] (3.7,1) to[out=90,in=180] (4.5,1.5);
\draw[purple, thick] (4.5,1.5) to[out=0,in=90] (5.3,1);
}

\DeclareSymbol{cycle1}{0}{
\draw  (0,0)   -- (-.7,1) node[dot] {}  ;
\draw (0,0) node[idot] {} -- (.7,1) node[dot] {};
\draw[purple, thick] (-.7,1) to[out=90,in=180] (0,1.5);
\draw[purple, thick] (0,1.5) to[out=0,in=90] (0.7,1);
}

\DeclareSymbol{1red}{0}{\draw[white] (-.4,0) -- (.4,0); \draw[red] (0,0)  -- (0,1) node[dot,red] {};}
\DeclareSymbol{11red}{0}{\draw[white] (-.4,0) -- (.4,0); \draw[red] (0,0)  -- (0,1) node[dot,red] {};
\draw[red] (0,1)  -- (0,2) node[dot,red] {};}
\DeclareSymbol{111red}{0}{\draw[white] (-.4,0) -- (.4,0); \draw[red] (0,0)  -- (0,1) node[dot,red] {};
\draw[red] (0,1)  -- (0,2) node[dot,red] {};
\draw[red] (0,2)  -- (0,3) node[dot,red] {};}

\DeclareSymbol{1but}{0}{\draw[white] (-.4,0) -- (.4,0); \draw[red] (0,0.5) node[dot,red] {};}
\DeclareSymbol{11but}{0}{\draw[white] (-.4,0) -- (.4,0); \draw[red] (0,0) node[dot,red] {}  -- (0,1) node[dot,red] {};}
\DeclareSymbol{111but}{0}{\draw[white] (-.4,0) -- (.4,0); \draw[red] (0,0) node[dot,red] {}  -- (0,1) node[dot,red] {} -- (0,2) node[dot,red] {};}

\DeclareSymbol{2but}{0}{\draw[white] (-.4,0) -- (.4,0);
\draw[red] (0,0) node[dot,red] {}  -- (0.7,1) node[dot,red] {};
\draw[red] (0,0) node[dot,red] {}  -- (-0.7,1) node[dot,red] {};
}

\DeclareSymbol{21but}{0}{\draw[white] (-.4,0) -- (.4,0);
\draw[red] (0,0) node[dot,red] {}  -- (0.7,1) node[dot,red] {};
\draw[red] (0,0) node[dot,red] {}  -- (-0.7,1) node[dot,red] {};
\draw[red] (-.7,1)  -- (-0.7,2) node[dot,red] {};
}

\DeclareSymbol{211but}{0}{\draw[white] (-.4,0) -- (.4,0);
\draw[red] (0,0) node[dot,red] {}  -- (0.7,1) node[dot,red] {};
\draw[red] (0,0) node[dot,red] {}  -- (0,1) node[dot,red] {};
\draw[red] (0,0) node[dot,red] {}  -- (-0.7,1) node[dot,red] {};
\draw[red] (-.7,1)  -- (-0.7,2) node[dot,red] {};
}

\DeclareSymbol{2110but}{0}{\draw[white] (-.4,0) -- (.4,0);
\draw[red] (0,0)   -- (0,-1) node[dot,red] {};
\draw[red] (0,0) node[dot,red] {}  -- (0.7,1) node[dot,red] {};
\draw[red] (0,0) node[dot,red] {}  -- (0,1) node[dot,red] {};
\draw[red] (0,0) node[dot,red] {}  -- (-0.7,1) node[dot,red] {};
\draw[red] (-.7,1)  -- (-0.7,2) node[dot,red] {};
}

\DeclareSymbol{21red}{0}{\draw[white] (-.4,0) -- (.4,0);
\draw[red] (0,0) node[dot,red] {}  -- (0.7,1) node[dot,red] {};
\draw[red] (0,0) node[dot,red] {}  -- (-0.7,1) node[dot,red] {};
\draw[red] (-.7,1)  -- (-0.7,2) node[dot,red] {};
\draw[red] (0,0)   -- (0,-1) ;
}

\DeclareSymbol{210}{0}{\draw[white] (-.4,0) -- (.4,0);
\draw (0,1)  -- (0.7,2);
\draw (0,1)  -- (-0.7,2);
\draw (-.7,2)  -- (-0.7,3);
\draw (0,1)   -- (0,0) ;
}

\DeclareSymbol{210_noise}{0}{
\node[root] (r) at (0,0) {};
\node[idot] (n1) at (0, 1) {};
\node[idot] (n2) at (-.7,2) {};
\draw (n1)  -- (0.7,2) node[dot] {};
\draw (n1)  -- (n2);
\draw (n2)  -- (-0.7,3) node[dot] {};
\draw (n1)   -- (r) ;
}

\DeclareSymbol{21}{0}{
\node[root] (r) at (0,0) {};
\node[idot] (n1) at (-.7,1) {};
\draw (r)  -- (0.7,1) node[dot] {} ;
\draw (r)  -- (n1);
\draw (n1)  -- (-0.7,2) node[dot] {} ;
}

\DeclareSymbol{20}{0}{\draw[white] (-.4,0) -- (.4,0);
\draw (0,0)  -- (-0.7,1) ;
\draw (-.7,1)  -- (-0.7,2) node[dot] {} ;
}

\DeclareSymbol{3red}{0}{\draw[white] (-.4,0) -- (.4,0);
\draw (-.7,1) node[dot,red] {} -- (0,0);
\draw (0,0) -- (.7,1) node[dot,red] {};
\draw (0,0) -- (0,1) node[dot,red] {};}

\DeclareSymbol{0inv}{0}{ \draw  (0,0.5) node[dot,olive] {};}
\DeclareSymbol{1inv}{0}{\draw[white] (-.4,0) -- (.4,0); \draw (0,0) node[dot,olive] {}  -- (0,1) node[dot] {};}
\DeclareSymbol{2inv}{0}{
\draw  (0,1) node[dot] {}   -- (-.7,0) node[dot,olive] {}  ;
\draw (0,1) -- (.7,0) node[dot,olive] {};}
\DeclareSymbol{3inv}{0}{\draw[white] (-.4,0) -- (.4,0);
\draw (-.7,0) node[dot,olive] {} -- (0,1) node[dot] {};
\draw (0,1)  -- (.7,0) node[dot,olive] {};
\draw (0,1) -- (0,0) node[dot,olive] {};}

\DeclareSymbol{30_3}{0}{
\node[idot] (r) at (0,0) {};
\node[idot] (n1) at (0,1) {};
\draw (-.7,2) node[dot] {} -- (n1);
\draw (n1) -- (.7,2) node[dot] {};
\draw (n1) -- (0,2) node[dot] {};
\draw (n1) -- (r);}

\DeclareSymbol{1130}{-3}{\draw (0,0) -- (0,-1); \draw (0,0) -- (0,1) ;
\draw (-.9,1) -- (0,0) -- (.9,1) ;
\draw (-1.6,2)  -- (-.9,1) -- (-.2,2) ;
\draw (-.9,1) -- (-.9,2) ;
}

\DeclareSymbol{1130_noise}{0}{
\node[root] (r) at (0,0) {};
\node[idot] (n1) at (0,1) {};
\node[idot] (n2) at (-.9,2) {};
\draw (n1) -- (r); \draw (n1) -- (0,2) node[dot] {};
\draw (n2) -- (n1) -- (.9,2) node[dot] {};
\draw (-1.6,3) node[dot] {} -- (n2) -- (-.2,3) node[dot] {};
\draw (n2) -- (-.9,3) node[dot] {};
}

\DeclareSymbol{113}{0}{ \draw (0,0) -- (0,1) ;
\draw (-.9,1) -- (0,0) -- (.9,1) ;
\draw (-1.6,2) -- (-.9,1) -- (-.2,2) ;
\draw (-.9,1) -- (-.9,2) ;
}

\DeclareSymbol{113_noise}{0}{
\node[root] (r) at (0,0) {};
\node[idot] (nl) at (-.9, 1) {};
\draw (r) -- (0,1) node[dot] {};
\draw (nl) -- (r) -- (.9,1) node[dot] {};
\draw (-1.6,2) node[dot] {} -- (nl) -- (-.2,2) node[dot] {};
\draw (nl) -- (-.9,2) node[dot] {};
}
\DeclareSymbol{311_noise}{0}{
\node[root] (r) at (0,0) {};
\node[idot] (n1) at (.9,1) {};
\draw (r) -- (0,1) node[dot] {};
\draw (n1) -- (r) -- (-.9,1) node[dot] {};
\draw (1.6,2) node[dot] {} -- (n1) -- (.2,2) node[dot] {};
\draw (n1) -- (.9,2) node[dot] {};
}
\DeclareSymbol{131_noise}{0}{
\node[root] (r) at (0,0) {};
\node[idot] (n1) at (0,1) {};
\draw (r) -- (n1) ;
\draw (.9,1) node[dot] {}-- (r) -- (-.9,1) node[dot] {}; \draw (n1) -- (0,2) node[dot] {};
\draw (-.7,2) node[dot] {} -- (n1) -- (.7,2) node[dot] {};
}

\DeclareSymbol{113_noise_0}{0}{
\node[root] (r) at (0,0) {};
\node[idot] (n1) at (0,1) {};
\node[idot] (n2) at (-.9,2) {};
\draw (r) -- (n1);
\draw (n1) -- (0,2) node[dot] {};
\draw (n2) -- (n1) -- (.9,2) node[dot] {};
\draw (-1.6,3) node[dot] {} -- (n2) -- (-.2,3) node[dot] {};
\draw (n2) -- (-.9,3) node[dot] {};
}
\DeclareSymbol{131_noise_0}{0}
{
\node[root] (r) at (0,0) {};
\node[idot] (n1) at (0,1) {};
\node[idot] (n2) at (0,2) {};
\draw (n1) -- (n2) ;
\draw (n2) -- (0,3) node[dot] {};
\draw (.9,2) node[dot] {}-- (n1) -- (-.9,2) node[dot] {}; \draw (n2) -- (n1) ;
\draw (-.7,3) node[dot] {} -- (n2) -- (.7,3) node[dot] {};
\draw (r) -- (n1);
}
\DeclareSymbol{311_noise_0}{0}{
\node[root] (r) at (0,0) {};
\node[idot] (n1) at (0,1) {};
\node[idot] (n2) at (.9,2) {};
\draw (n1) -- (0,2) node[dot] {};
\draw (n2) -- (.9,3) node[dot] {};
\draw (n2) -- (n1) -- (-.9,2) node[dot] {};
\draw (1.6,3) node[dot] {} -- (n2) -- (.2,3) node[dot] {};
\draw (n2) -- (.9,2) ;
\draw (r) -- (n1);
}

\DeclareSymbol{331_noise_0}{0}{
\node[root] (r) at (0,0) {};
\node[idot] (n1) at (0,1) {};
\node[idot] (n2) at (-.9,2) {};
\node[idot] (n3) at (.9,2) {};
\draw (r) -- (n1);
\draw (n1) -- (0,2) node[dot] {};
\draw (n2) -- (n1) -- (n3) {};
\draw (-1.6,3) node[dot] {} -- (n2) -- (-.2,3) node[dot] {};
\draw (n2) -- (-.9,3) node[dot] {};
\draw (1.6,3) node[dot] {} -- (n3) -- (.2,3) node[dot] {};
\draw (n3) -- (.9,3) node[dot] {};
}

\DeclareSymbol{333_noise_0}{0}{
\draw (-2.6,3) node[dot] {} -- (-1.9,2) -- (-1.2,3) node[dot] {};
\draw (-1.9,2) -- (-1.9,3) node[dot] {};
\draw (2.6,3) node[dot] {} -- (1.9,2) -- (1.2,3) node[dot] {};
\draw (1.9,2) --(1.9,3) node[dot]{};
\draw (-.8,3) node[dot] {}-- (0,2) --(.8,3) node[dot]{};
\draw (0,1)  -- (0,3)  node[dot] {};
\draw (0,1)  -- (0,2) node[idot] {} ;
\draw (-1.9,2) node[idot] {} -- (0,1) -- (1.9,2) node[idot] {};
\draw (0,0) node[idot] {} -- (0,1) node[idot] {};
}

\DeclareSymbol{113red}{0}{ \draw[red] (0,0) -- (0,1) node[dot,red] {};
\draw[red] (-.9,1) -- (0,0) -- (.9,1) node[dot,red] {};
\draw[red,red] (-1.6,2) node[dot,red] {} -- (-.9,1) -- (-.2,2) node[dot,red] {};
\draw[red] (-.9,1) -- (-.9,2) node[dot,red] {};
}

\DeclareSymbol{1130_3i}{0}{\draw (0,1) -- (0,0); \draw (0,1) -- (0,2) node[dot] {};
\draw (-.9,2) -- (0,1) -- (.9,2) node[dot] {};
\draw (-1.6,3) node[dot] {} -- (-.9,2) -- (-.2,3) node[dot] {};
\draw (-.9,2) -- (-.9,3) node[dot] {};
\draw[purple, thick] (0,2) to[out=90,in=180] (.45,2.5);
\draw[purple, thick] (.45,2.5) to[out=0,in=90] (.9,2);
\node[root] (r) at (0,0) {};
\node[idot] (n1) at (0,1) {};
\node[idot] (n2) at (-.9,2) {};
}

\DeclareSymbol{1130_3ii}{0}{\draw (0,1) -- (0,0); \draw (0,1) -- (0,2) node[dot] {};
\draw (-.9,2) -- (0,1) -- (.9,2) node[dot] {};
\draw (-1.6,3) node[dot] {} -- (-.9,2) -- (-.2,3) node[dot] {};
\draw (-.9,2) -- (-.9,3) node[dot] {};
\draw[purple, thick] (-.9,3) to[out=90,in=180] (-.55,3.5);
\draw[purple, thick] (-.55,3.5) to[out=0,in=90] (-0.2,3);
\node[root] (r) at (0,0) {};
\node[idot] (n1) at (0,1) {};
\node[idot] (n2) at (-.9,2) {};
}

\DeclareSymbol{1130_3iii}{0}{
\draw (0,1) -- (0,0); \draw (0,1) -- (0,2) node[dot] {};
\draw (-.9,2) -- (0,1) -- (.9,2) node[dot] {};
\draw (-1.6,3) node[dot] {} -- (-.9,2) -- (-.2,3) node[dot] {};
\draw (-.9,2) -- (-.9,3) node[dot] {};
\draw[purple, thick] (-.9,3) to[out=90,in=180] (0,3.5);
\draw[purple, thick] (0,3.5) to[out=0,in=90] (.9,2);
\node[root] (r) at (0,0) {};
\node[idot] (n1) at (0,1) {};
\node[idot] (n2) at (-.9,2) {};
}

\DeclareSymbol{1130_1i}{0}{
\draw (0,1) -- (0,0); \draw (0,1) -- (0,2) node[dot] {};
\draw (-.9,2) -- (0,1) -- (.9,2) node[dot] {};
\draw (-1.6,3) node[dot] {} -- (-.9,2) -- (-.2,3) node[dot] {};
\draw (-.9,2) -- (-.9,3) node[dot] {};
\draw[purple, thick] (-.9,3) to[out=90,in=180] (-.55,3.5);
\draw[purple, thick] (-.55,3.5) to[out=0,in=90] (-0.2,3);
\draw[purple, thick] (0,2) to[out=90,in=180] (.45,2.5);
\draw[purple, thick] (.45,2.5) to[out=0,in=90] (.9,2);
\node[root] (r) at (0,0) {};
\node[idot] (n1) at (0,1) {};
\node[idot] (n2) at (-.9,2) {};
}

\DeclareSymbol{1130_1ia}{0}{
\draw (0,1) -- (0,0); \draw (0,1) -- (0,2) node[dot] {};
\draw (-.9,2) -- (0,1) -- (.9,2) node[dot] {};
\draw (-1.6,3) node[dot] {} -- (-.9,2) -- (-.2,3) node[dot] {};
\draw (-.9,2) -- (-.9,3) node[dot] {};
\draw[purple, thick] (-1.6,3) to[out=90,in=180] (-.9,3.5);
\draw[purple, thick] (-.9,3.5) to[out=0,in=90] (-0.2,3);
\draw[purple, thick] (0,2) to[out=90,in=180] (.45,2.5);
\draw[purple, thick] (.45,2.5) to[out=0,in=90] (.9,2);
\node[root] (r) at (0,0) {};
\node[idot] (n1) at (0,1) {};
\node[idot] (n2) at (-.9,2) {};
}

\DeclareSymbol{1130_1ib}{0}{
\draw (0,1) -- (0,0); \draw (0,1) -- (0,2) node[dot] {};
\draw (-.9,2) -- (0,1) -- (.9,2) node[dot] {};
\draw (-1.6,3) node[dot] {} -- (-.9,2) -- (-.2,3) node[dot] {};
\draw (-.9,2) -- (-.9,3) node[dot] {};
\draw[purple, thick] (-1.6,3) to[out=90,in=180] (-1.25,3.5);
\draw[purple, thick] (-1.25,3.5) to[out=0,in=90] (-0.9,3);
\draw[purple, thick] (0,2) to[out=90,in=180] (.45,2.5);
\draw[purple, thick] (.45,2.5) to[out=0,in=90] (.9,2);
\node[root] (r) at (0,0) {};
\node[idot] (n1) at (0,1) {};
\node[idot] (n2) at (-.9,2) {};
}

\DeclareSymbol{1130_1ii}{0}{
\draw (0,1) -- (0,0); \draw (0,1) -- (0,2) node[dot] {};
\draw (-.9,2) -- (0,1) -- (.9,2) node[dot] {};
\draw (-1.6,3) node[dot] {} -- (-.9,2) -- (-.2,3) node[dot] {};
\draw (-.9,2) -- (-.9,3) node[dot] {};
\draw[purple, thick] (-.9,3) to[out=90,in=180] (0,3.5);
\draw[purple, thick] (-0.2,3) to[out=90,in=180] (-0.1,3.2);
\draw[purple, thick] (-0.1,3.2) to[out=0,in=90] (0,2);
\draw[purple, thick] (0,3.5) to[out=0,in=90] (.9,2);
\node[root] (r) at (0,0) {};
\node[idot] (n1) at (0,1) {};
\node[idot] (n2) at (-.9,2) {};
}

\DeclareSymbol{1130_1iii}{0}{
\draw (0,1) -- (0,0); \draw (0,1) -- (0,2) node[dot] {};
\draw (-.9,2) -- (0,1) -- (.9,2) node[dot] {};
\draw (-1.6,3) node[dot] {} -- (-.9,2) -- (-.2,3) node[dot] {};
\draw (-.9,2) -- (-.9,3) node[dot] {};
\draw[purple, thick] (-1.6,3) to[out=90,in=180] (-1.25,3.5);
\draw[purple, thick] (-1.25,3.5) to[out=0,in=90] (-0.9,3);
\draw[purple, thick] (-0.2,3) to[out=90,in=180] (-0.1,3.2);
\draw[purple, thick] (-0.1,3.2) to[out=0,in=90] (0,2);
\node[root] (r) at (0,0) {};
\node[idot] (n1) at (0,1) {};
\node[idot] (n2) at (-.9,2) {};
}

\DeclareSymbol{yn_cycle_exmpl}{0}{\draw (0,0) -- (0,-1); \draw (0,0) -- (0,1) node[dot] {};
\draw (-1.5,1) -- (0,0) -- (1.5,1) ;
\draw (-2.1,2)  -- (-1.5,1) -- (-.6,2) node[dot] {};
\draw (-1.5,1) -- (-1.3,2) node[dot] {};
\draw (2.1,2) node[dot] {} -- (1.5,1) -- (.9,2) node[dot] {};
\draw (1.5,1) -- (1.5,2) node[dot] {};
\draw (-2.7,3) node[dot] {}  -- (-2.1,2) -- (-1.5,3) node[dot] {};
\draw (-2.1,2) -- (-2.1,3) node[dot] {};
\draw[purple, thick] (-2.7,3) to[out=90,in=180] (-2.4,3.5);
\draw[purple, thick] (-2.4,3.5) to[out=0,in=90] (-2.1,3);
\draw[purple, thick] (-1.5,3) to[out=90,in=180] (-.3,3.5);
\draw[purple, thick] (-.3,3.5) to[out=0,in=90] (1.5,2);
\draw[purple, thick] (-1.3,2) to[out=90,in=180] (-.3,2.7);
\draw[purple, thick] (-.3,2.7) to[out=0,in=90] (0.9,2);
\draw[purple, thick] (-.6,2) to[out=90,in=180] (-.3,2.2);
\draw[purple, thick] (-.3,2.2) to[out=0,in=90] (0,1);
}

\DeclareSymbol{30_3red}{0}{\draw[white] (-.4,0) -- (.4,0);
\draw[red] (-.7,2) node[dot,red] {} -- (0,1);
\draw[red] (0,1) -- (.7,2) node[dot,red] {};
\draw[red] (0,1) -- (0,2) node[dot,red] {};
\draw[red] (0,1) -- (0,0);}

\DeclareSymbol{30_1}{0}{
\draw (-.7,2) node[dot] {} -- (0,1);
\draw (0,1) -- (.7,2) node[dot] {};
\draw (0,1) -- (0,2) node[dot] {};
\draw[purple, thick] (0,2) to[out=90,in=180] (.35,2.5);
\draw[purple, thick] (.35,2.5) to[out=0,in=90] (.7,2);
\draw (0,1) node[idot] {} -- (0,0) node[idot] {};
}

\DeclareSymbol{30_2}{0}{
\draw[white] (-.4,0) -- (.4,0);
\draw (-.7,2) node[dot] {} -- (0,1);
\draw (0,1) -- (.7,2) node[dot] {};
\draw (0,1) -- (0,2) node[dot] {};
\draw (0,1) -- (0,0);
\draw[purple, thick] (0,2) to[out=90,in=0] (-.35,2.5);
\draw[purple, thick] (-.35,2.5) to[out=180,in=90] (-.7,2);}

\DeclareSymbol{2cycle}{0}{
\draw[purple, thick] (-.7,1) -- (-1.7,1);
\draw[purple, thick] (.7,1) -- (1.7,1);
\draw  (0,0)   -- (-.7,1) node[dot] {}  ;
\draw (0,0) node[idot] {} -- (.7,1) node[dot] {};
\node at (0.7,0) {\scalebox{0.8}{$\ve$}};
}

\DeclareSymbol{lab1130_1iii}{-3}{
\draw (0,0) -- (0,-1); \draw (0,0) -- (0,1) node[dot] {};
\draw (-.9,1) -- (0,0) -- (.9,1) node[dot] {};
\draw (-1.6,2) node[dot] {} -- (-.9,1) -- (-.2,2) node[dot] {};
\draw (-.9,1) -- (-.9,2) node[dot] {};
\draw[purple, thick] (-1.6,2) to[out=90,in=180] (-1.25,2.5);
\draw[purple, thick] (-1.25,2.5) to[out=0,in=90] (-0.9,2);
\draw[purple, thick] (-0.2,2) to[out=90,in=180] (-0.1,2.2);
\draw[purple, thick] (-0.1,2.2) to[out=0,in=90] (0,1);
\node (c) [label = left:{\tiny $1$}] at (.5, 0.0) {};
\node (a) [label = left:{\tiny $2$}] at (-.3, 0.8) {};
\node (b) [label = right:{\tiny $3$}] at (0.4, 1.2) {};
\node at (0.7,-1) {\scalebox{0.8}{$\ve$}};
\node[root] (r) at (0,-1) {};
\node[idot] (n1) at (0,0) {};
\node[idot] (n2) at (-.9,1) {};
}

\makeatletter
\def\DeclareLargeSymbol#1#2#3{\expandafter\gdef\csname MH@symb@#1\endcsname{\tikz[baseline=#2,scale=0.6]{#3}}
\expandafter\gdef\csname
MH@symb@#1s\endcsname{\scalebox{0.6}{\tikz[baseline=#2,scale=0.15]{#3}}}}
\def\<#1>{\csname MH@symb@#1\endcsname}
\makeatother

\DeclareLargeSymbol{contain1cycle}{-17}{
\draw  (0,0)   -- (-.7,1) ;
\node at (-0.7, 1.3) {\scalebox{0.9}{$\tau_1$}};
\node at (0.8, 0.3) {\scalebox{0.9}{$\tau_2$}};
\node at (1.7, 0.3) {\scalebox{0.9}{$\tau_3$}};
\draw  (0,0)   -- (0,1) node[dot] {}  ;
\draw  (0,0) -- (.7,1) node[dot] {};
\node at (-1,-0.2) {\scalebox{0.8}{$(s_1,y_1)$} };
\node at (-0.3,-1.2) {\scalebox{0.8}{$(s_0,y_0)$} };
\draw (.7,-1) -- (1.4, 0); \draw (0.7, -1) -- (0.7, 0);
\node at (-1.7,0.8) {\scalebox{0.8}{$(s_2,y_2)$} };
\node at (0.9, -1.2) {$\cdot$}; \node at (1.05, -1.4) {$\cdot$}; \node at (1.23, -1.6) {$\cdot$};
\draw[purple, thick] (0,1) to[out=90,in=180] (0.35,1.5);
\draw[purple, thick] (.35,1.5) to[out=0,in=90] (0.7,1);
\draw  (0.7,-1) node[idot] {}  --(0,0) node[idot] {} ;
}
\DeclareLargeSymbol{containno1cycle}{-17}{
\node at (0, 0.3) {\scalebox{0.9}{$\tau_1$}};
\node at (0.8, 0.3) {\scalebox{0.9}{$\tau_2$}};
\node at (1.6, 0.3) {\scalebox{0.9}{$\tau_3$}};
\node at (-1,-0.2) {\scalebox{0.8}{$(s_2,y_2)$} };
\node at (-0.3,-1.2) {\scalebox{0.8}{$(s_0,y_0)$} };
\draw (.7,-1) -- (1.4, 0); \draw (0.7, -1) -- (0.7, 0);
\node at (0.9, -1.2) {$\cdot$}; \node at (1.05, -1.4) {$\cdot$}; \node at (1.23, -1.6) {$\cdot$};
\draw  (0.7,-1)  node[idot] {}  --(0,0);
}

\makeatletter
\def\DeclareSmallSymbol#1#2#3{\expandafter\gdef\csname MH@symb@#1\endcsname{\tikz[baseline=#2,scale=0.2]{#3}}
\expandafter\gdef\csname
MH@symb@#1s\endcsname{\scalebox{0.6}{\tikz[baseline=#2,scale=0.15]{#3}}}}
\def\<#1>{\csname MH@symb@#1\endcsname}
\makeatother

\DeclareSmallSymbol{s3m}{0}{\draw[white] (-.4,0) -- (.4,0);
\draw (-.7,1) node[dot] {}   -- (0,0);
\draw (0,0) -- (.7,1) node[dot] {}  ;
\draw (0,0) -- (0,1) node[dot] {}  ;
\draw[purple, thick] (-.7,1) to[out=90,in=180] (-0.35,1.5);
\draw[purple, thick] (-.35,1.5) to[out=0,in=90] (0,1);
\node at (-0.7,0) {\scalebox{0.5}{$1$}};
\node at (1.4,1) {\scalebox{0.5}{$3$}};
\draw (0,-1) node[idot] {} --(0,0) node[idot] {};
}

\DeclareSmallSymbol{s3}{0}{\draw[white] (-.4,0) -- (.4,0);
\draw (-.7,1) node[dot] {}   -- (0,0);
\draw (0,0) -- (.7,1) node[dot] {}  ;
\draw (0,0) -- (0,1) node[dot] {}  ;
}

\DeclareSmallSymbol{s1}{0}{\draw[white] (-.4,0) -- (.4,0);
\draw (0,0) -- (0,1) node[dot] {}  ;
}

\DeclareSmallSymbol{s2}{0}{
\draw  (0,0)   -- (-.7,1) node[dot] {}   ;
\draw (0,0) -- (.7,1) node[dot] {}  ;
}

\newcommand\T{\mathscr{T}}
\newcommand{\E}{\mathbf{E}}

\newtheorem{theorem}{Theorem}[section]
\newtheorem{lemma}[theorem]{Lemma}
\newtheorem{proposition}[theorem]{Proposition}
\newtheorem{corollary}[theorem]{Corollary}
\newtheorem{remark}[theorem]{Remark}

\newtheorem{definition}[theorem]{Definition}

\newtheorem{example}{Example}

\title{The Allen--Cahn equation with weakly critical random initial datum}

\numberwithin{equation}{section}

\begin{document}
\author{Simon Gabriel\footnote{Email:
\href{mailto:simon.gabriel@uni-muenster.de}{simon.gabriel@uni-muenster.de}, University of
M\"unster, DE and   University of Warwick, UK}, \ Tommaso Rosati\footnote{Email:
\href{mailto:t.rosati@warwick.ac.uk}{t.rosati@warwick.ac.uk}, University of
Warwick, UK} \ and Nikos
Zygouras\footnote{Email:
\href{mailto:n.zygouras@warwick.ac.uk}{n.zygouras@warwick.ac.uk},  University
of Warwick, UK}}
\date{}

\maketitle

\begin{abstract}
\noindent
This work considers the two-dimensional Allen--Cahn equation
\begin{equation*}
    \partial_t u = \frac{1}{2}\Delta u + \mathfrak{m}\,  u
-u^3\;, \quad u(0,x)= \eta (x)\;, \qquad \forall (t,x) \in [0, \infty) \times
\RR^{2} \;,
\end{equation*}
where the initial condition $ \eta $ is a two-dimensional white noise, which lies
in the scaling critical space of initial data to the
equation. In a weak coupling scaling, 
we establish a Gaussian limit
with nontrivial size of fluctuations, thus casting the nonlinearity as
marginally relevant. The result builds on a precise analysis of
the Wild expansion of the solution  and an understanding of the underlying stochastic and
combinatorial structure. This gives rise to a representation for the limiting variance 
in terms of Butcher series associated to the solution of an ordinary differential equation.
\end{abstract}

\setcounter{tocdepth}{2}
\begingroup
  \hypersetup{hidelinks}
  \tableofcontents
\endgroup

\section{Introduction} 
We consider the two-dimensional Allen--Cahn equation
\begin{equation}\label{e:main}
    \partial_t u = \frac{1}{2}\Delta u + \mathfrak{m}\,   u
-u^3\;, \quad u(0,\cdot)= \eta ( \cdot)\;, \quad t \geqslant 0 \;, 
\end{equation}
with $\mathfrak{m} \in \RR $ and initial condition a
spatial white noise $ \eta$ on $ \RR^{2} $, namely a homogeneous Gaussian random field with
correlation function
\begin{equation*}
\begin{aligned}
\EE [\eta (x) \eta (y)] = \delta (x-y) \;, \qquad \forall x, y \in
\RR^{2} \;,
\end{aligned}
\end{equation*}
where $ \delta $ is the Dirac delta at zero.
Under these assumptions, there exists no
solution theory for \eqref{e:main}, since the initial condition is scaling
critical for the equation. By this we mean that formally (assuming the solution
$ u $ exists and setting for simplicity $ \mf{m} =0 $), under the re-scaling $
u^{\delta}(t, x) := \delta u ( \delta^{2} t, \delta x) $, which leaves the
initial condition invariant in law, $ u^{\delta} $ solves again
\begin{equation}\label{renorm}
\begin{aligned}
\partial_{t} u^{\delta} = \frac{1}{2} \Delta u^{\delta} - (u^{\delta})^{3} \;.
\end{aligned}
\end{equation}
The invariance under this re-scaling indicates that the nonlinearity cannot be treated 
perturbatively and so for instance modern theories on singular SPDEs do
not apply, since they
work under the assumption that on small scales the equation is governed by the
effect of the Laplacian.

Obtaining a solution theory to \eqref{e:main} is therefore 
challenging but also of great interest.
Indeed, if $ \mf{m} > 0 $, then
\eqref{e:main} is a model for the formation of phase fields, since $
\mf{m} u - u^{3} $ corresponds to the gradient of a double-well potential and
the solution tends (for large times) to take the value of one of the two minima of
the potential, leading to the evolution of two competing phases. Here, starting the
equation at a generic initial condition should lead to many conjectured
long-time properties, for instance regarding the speed of coarsening of sets in
the evolution of mean curvature flow (see \cite{OhtaKawasaki,bray} from the physics
literature and \cite{otto2002,HLR22} for some mathematical results and a discussion of
the problem). In this context, space
white noise plays the role of a canonical ``totally mixed'' initial condition, which gets
instantaneously smoothened by the heat flow, leading to random level sets which
then evolve under mean curvature flow.

On the other hand, when $ \mf{m} \leqslant 0 $, and if instead of a random initial
condition one chooses an additive space-time white noise, then equation~\eqref{e:main} is a
fundamental model in stochastic quantisation. In this case the
invariant measure of the equation is a celebrated model in
quantum field theory, and a recent proof of triviality in the critical
dimension $ d=4 $ has been a breakthrough in the mathematical understanding of
such measure \cite{MR4276286}. At the same time, there is no result for
the dynamics of the equation in $ d =4 $. 

We will study \eqref{e:main} in a \emph{weak coupling} regime.
Namely, for $ \ve \in (0, \tfrac{1}{2}) $\footnote{The choice $ \ve< 1/2 $ is arbitrary to avoid issues
with the blow-up at $ \ve =1 $ of the logarithm.} and 
$p_{t}( x) = \tfrac{1}{2 \pi t} \exp \left( -\tfrac{|x|^2}{ 2 t} \right)$,
we study the limiting behaviour of the solution $ \mU_{\ve} $ to
\begin{equation}\label{e:acn}
    \partial_t \mU_\varepsilon = \frac{1}{2}\Delta \mU_\varepsilon
+ \mathfrak{m}\, \mU_{\ve} - \frac{ 1}{\log{ \tfrac{1}{\ve}}} \mU_\varepsilon^3\;, \qquad
\mU_\varepsilon(0,\cdot)= \hat{\lambda}\,  p_{\ve^{2}} \star \eta (\cdot)\;,
\end{equation}
where $\star$ denotes spatial convolution. The parameter
$\hat{\lambda} > 0 $ will be referred to as the \emph{coupling constant}. 
By scaling, \eqref{e:acn} is equivalent to the problem
\begin{equation}\label{e:ac2}
    \partial_t u_\varepsilon = \frac{1}{2}\Delta u_\varepsilon
+ \mathfrak{m}\, u_{\ve} - u_\varepsilon^3\;, \qquad
u_\varepsilon(0,\cdot)=  \eta_\varepsilon(\cdot):= \frac{ \hat{\lambda}}{\sqrt{\log{ \tfrac{1}{\ve}}}}  p_{\ve^{2}} \star \eta (\cdot) \;,
\end{equation}
and the solutions to \eqref{e:acn} and \eqref{e:ac2} are related via
 $\mU_\varepsilon(t,x)=\sqrt{\log\tfrac{1}{\varepsilon}} \,\,
u_\varepsilon(t,x)$.
We will also use the notation $\lambda^2_\varepsilon:= \hat{\lambda}^{2} \big(\log{ \tfrac{1}{\ve}}\big)^{-1}$.
We remark that even though the nonlinearity is attenuated to zero, it still has
a non-trivial effect in the limit.
To make this point more precise let us, first, state our main result:

\begin{theorem}\label{thm:main}
There exists a $ \hat{\lambda}_{\mathrm{fin}} \in (0, \infty) $ such that if $
\mf{m} \in \RR $ and  $T, \hat{\lambda} \in (0, \infty)$
satisfy
\begin{equation}\label{e_main_ass}
\begin{aligned}
\overline{\mf{m}} \, T \leqslant  \log{ (\hat{\lambda}_{\mathrm{fin}} /
\hat{\lambda})}\,, \qquad \text{with} \qquad \overline{\mf{m}} = \max \{ \mf{m}\;, 0 \} \;,
\end{aligned}
\end{equation}
and if $\sigma_{\Hat{\lambda}} := \sigma( \hat{\lambda}^{2}) :=
( 1+\frac{3}{\pi}\Hat{\lambda}^{2})^{- \frac{1}{2}} $, 
where $ \zeta \mapsto
\sigma (\zeta) $ is the solution to
\begin{align}\label{ODE}
\frac{ \ud}{\ud \zeta} 
 \sigma = - \frac{3}{2 \pi}
\sigma^{3} \qquad \text{with} \qquad  \sigma (0) = 1 \;,
\end{align}
 then we have:
\begin{equation*}
\begin{aligned}
\lim_{\ve \to 0} \EE \left[ \big| \mU_{\ve}(t, x) - \hat{\lambda}\,  \sigma_{ \hat{\lambda}}
\,e^{\mf{m} t} \, p_{t} \star \eta (x)
\big|^{2}\right] = 0 \;, \qquad \forall (t, x) \in (0, T] \times
\RR^{2} \;.
\end{aligned}
\end{equation*}
\end{theorem}
Our result shows that the non-linearity is {\it marginally}\footnote{{ ``marginally'' 
refers to the fact that the nonlinearity has neither a dominant nor a negligible effect
as seen via the renormalisation scaling $u^{\delta}(t, x) = \delta u ( \delta^{2} t, \delta x) $ 
and the invariance in \eqref{renorm}}} 
{\it relevant} as it affects the size of the limiting
fluctuations $ \hat{\lambda}\, \sigma_{ \hat{\lambda}}$, which are strictly weaker than the
fluctuations of the limit when simply dropping the non-linearity in
\eqref{e:acn}
(the solution to this linear problem reads as $ \hat{\lambda} e^{\mf{m} t} \, p_{t} \star \eta (x)$).
The latter describes the limiting fluctuations of a \emph{sub-critical}
scaling of the initial condition, considered in the previous work by
Hairer--L\^e--Rosati \cite{HLR22}, which studied the long-time
behaviour of \eqref{e:ac2}.
Let us note that our requirement $
\lambda< \hat{\lambda}_{\mathrm{fin}} < \infty$ emerges as a technical constraint and
we do not expect any actual phase transition. The restriction comes 
from the necessity to control certain series expansions (see in particular 
Propositions \ref{prop_max_principal} and \ref{prop_exp_close_to_but}
as well as Remark~\ref{rem:needWild}).
We conjecture that the result should extend to all $ \hat{\lambda}, t > 0 $.
Such extension would be especially interesting in relation to the study
of the metastable behaviour of the Allen--Cahn equation
(with $ \mf{m} > 0 $) at large scales.

The understanding of stochastic PDEs (and related statistical mechanics models) 
at the critical dimension is only now starting to take shape.
The only studied examples, so far, are two-dimensional stochastic
heat equations (SHE) with multiplicative space-time white noise, the two-dimensional
isotropic and anisotropic KPZ equation and the two-dimensional Burgers
equation.
 For the linear SHE, a weak coupling regime was noted by Bertini--Cancrini \cite{BC98} and explored 
 by Caravenna--Sun--Zygouras \cite{CSZ17}. 
 Here the coupling constant which appears in the
weak scaling plays a crucial role as a phase transition takes place at a precisely defined,
critical value. Below this value, Gaussian fluctuations,
 similar in spirit to Theorem~\ref{thm:main}
emerge \cite{CSZ17}, while a limiting field, which is not Gaussian or an
exponential of Gaussian,
emerges at the critical value \cite{caravenna2023critical, CSZ24}.
A similar phase transition takes place for the two-dimensional isotropic KPZ equation \cite{DC20, CSZ_kpz, Gu19}
but, so far, only Gaussian fluctuations below the critical coupling value have been established \cite{CSZ_kpz, Gu19}.
This result has been obtained via the use of the Cole-Hopf transformation, which relates the solution of the linear
SHE to that of the KPZ equation. 
For the two-dimensional anisotropic KPZ and Burgers equation, 
the weak coupling limit has been studied in \cite{CES_aKPZ, CET_aKPZ, CGT_burg} and Gaussian fluctuations 
have been established, building crucially on the explicit Gaussian
invariant measure available for that equation. We note that no phase transition takes place in this setting.
Finally, Dunlap--Gu \cite{DG22} and Dunlap--Graham \cite{dunlap20232d} have proposed another approach to
weak coupling limits of the (nonlinear) stochastic heat equations through the study of 
forwards-backwards SDEs, while the linear SHE on a critical hierarchical lattice has
been explored by Clark \cite{clark}.

Our approach in this work attempts to make a first step towards analysing 
singular SPDEs at the {\it critical} dimension via a systematic study 
of its Wild expansion, i.e.\ an expansion of the solution (in the spirit of a Picard iteration) 
in terms of iterated space-time stochastic integrals that are indexed by trees and live in
certain Wiener chaoses. 
An approach of this sort has been successful in the 
study of {\it subcritical} SPDEs via the theory of regularity structures \cite{Hairer2014}, thanks to the
fact that one can restrict attention to a finite number of
terms in the expansion, before exploiting sophisticated analytic solution theories. 
On the contrary, at the critical case {\it all} terms in the expansion contribute and their contribution 
needs to be accounted for. 

In the
case of \eqref{e:acn} and \eqref{e:ac2}, 
analysing the terms in the Wild expansion, we discover that the main contribution comes from certain projections of the Wild 
terms on the first Wiener chaos, see Proposition~\ref{prop_single_tree}.
Hence, the Gaussian limiting behaviour. 
An interesting structure emerges that clarifies the role of 
the ODE \eqref{ODE} as determining the limiting order of the fluctuations of
the Gaussian field. 
Roughly speaking, the terms in the Wild expansion, which have the dominant contribution, 
appear in the limit as a Gaussian variable, multiplied by an iterated integral in time variables
indexed by another tree structure. The new iterated integrals are recognised as terms in the
celebrated Butcher series (or B-series) expansion of the ODE \eqref{ODE}.
Our approach has both an analytical and a combinatorial component. 
On the analytical side, we take advantage of the
contractive properties of \eqref{e:ac2} to control the error of the truncated Wild series 
(see Proposition \ref{prop_max_principal} and Remark~\ref{rem:needWild}). This analysis is inspired by the study of \emph{finite}
Wild expansions performed in \cite{HLR22}. However, in order to treat a diverging
number of terms, it requires a new and deeper understanding of the
interplay between the graphical properties of the trees, that index terms in the
expansion, and their analytic contribution.
Thus, we introduce a combinatorial component to perform a detailed analysis of
Wiener chaos decompositions in terms of graph theoretical trees.
Here we find that specific cycles appearing in contracted trees (dubbed $
\mathrm{v} $-cycles) play a fundamental r\^ole, and interesting links to
permutation cycles and their statistics appear: see Sections \ref{sec:exmpl}
and \ref{sec:estimates-single-tree}.

Let us close this introduction with a curious link of our main result to a mean-field equation
of McKean-Vlasov type. More precisely, the limiting fluctuations of
\eqref{e:acn} appear to agree
with the limiting fluctuations of the McKean-Vlasov equation
\begin{equation}\label{e_homog_ac}
\begin{aligned}
\partial_t \mV_{ \ve} = \frac{1}{2} \Delta \mV_{ \ve}
+\mf{m}\,
\mV_{ \ve} - \frac{3 }{\log{\tfrac{1}{\ve}}} 
\EE \left[ \mV^{2}_{\ve}\right] \cdot
\mV_{ \ve}\,, \qquad \mV_{ \ve}(0,\cdot) = \hat{\lambda}\, p_{\ve^{2}} \star \eta ( \cdot)\,.
\end{aligned}
\end{equation}
In particular, we have that 
\begin{proposition}\label{p_homog_ac}
For any $ \hat{\lambda} > 0 $ there exists a unique solution $ \mV_{ \ve}$
to \eqref{e_homog_ac}, in the sense of Definition~\ref{def:solmf}, which satisfies:
\begin{equation*}
\begin{aligned}
\lim_{\ve \to 0} \EE \left[ | \mV_{\ve}(t, x) - \hat{\lambda}\, \sigma_{ \hat{\lambda}}
 e^{\mf{m} t}  p_{t} \star \eta (x) |^{2}\right] = 0\;, \qquad
\forall (t, x) \in (0, \infty) \times
\RR^{2} \,,
\end{aligned}
\end{equation*}
with $\sigma_{\hat \lambda}$ as in \eqref{ODE}.
\end{proposition}
An emergence of an equation like \eqref{e_homog_ac} might appear plausible 
if one considers an ansatz where the
leading order terms of the solution to \eqref{e:acn} is Gaussian. In such a setting and if
 $ \mU_{\ve} $ is Gaussian, then the projection on the
first homogeneous Wiener chaos of $ \mU_{\ve}^{3} $ is given by $ 3 \EE
[\mU_{\ve}^{2} ] \mU_{\ve} $, which agrees with the nonlinearity in \eqref{e_homog_ac}. 
However, $\mU_\varepsilon$, itself is far from Gaussian and
it is also far from obvious, a priori, that its limit \eqref{e:acn} is Gaussian.
A deeper understanding of the relations between \eqref{e:acn} and \eqref{e_homog_ac} is desirable.

\subsection*{Outline of the paper}

The remainder of the paper is structured as follows. In Section~\ref{sec:wild}
we give an introduction to rooted trees and their use in the analysis of ODEs
and SPDEs.
In Section~\ref{sec:prf-main}, we  present the main steps of the proof of Theorem~\ref{thm:main} while assuming the paper's
key ingredient,
Proposition~\ref{prop_single_tree}.
Section~\ref{sec:exmpl} introduces notation and estimates required, before we
 provide a proof of
Proposition~\ref{prop_single_tree} in  Section~\ref{sec:estimates-single-tree}.
In Section \ref{sec:MV} we prove Proposition \ref{p_homog_ac}.
Finally, we establish some technical results in the Appendix.

\subsection*{Notation}
Let $ \NN= \{0,1,2,\ldots \}$.
We denote with
$P_t= \exp \left(\frac{t}{2}\Delta\right)$ the heat semigroup on $\RR^2$:
\begin{equation*}
\begin{aligned}
P_t \varphi (x)  = p_t \star \varphi (x) \;, \quad  p_t (x)  = \frac{1}{2 \pi t} \exp \left( -\frac{|x|^2}{2 t} \right) \mathds{1}_{[0,
\infty)}(t) \;, \quad  \forall (t, x) \in \RR \times \RR^{2}\;,
\end{aligned}
\end{equation*}
where $\star$ denotes spatial convolution.
Similarly, we define the heat semigroup with mass $ \mf{m} \in \RR $ by $
P^{(\mathfrak{m})}_{t} = e^{\mathfrak{m}\, t} P_{t}$ and associated kernel $
p^{(\mathfrak{m})}_{t}(x) = e^{\mathfrak{m}\, t}p_{t}(x)
$. Here we allow the semigroup to be defined for any Schwartz distribution
$ \varphi $ on $ \RR^{2} $: if such $ \varphi $ is not locally integrable the integral should be
interpreted in the usual generalised sense.
We will abuse notation and sometimes write singletons of the form $\{a\}$
simply as $a$. Thus, $ A \setminus a := A \setminus \{a\}$, for some set $A$.

\subsection*{Acknowledgments}
The authors thank Felix Otto for pointing out the possibility 
of connecting our main result to 
the McKean-Vlasov problem \eqref{e_homog_ac}.
We also thank Giuseppe Cannizzaro, Khoa L\^e, Hao Shen, Daniel Ueltschi, Lorenzo Zambotti, Rongchan Zhu and Xiangchan  Zhu for helpful discussions.
SG was supported by the Warwick Mathematics Institute Centre for Doctoral
Training, and acknowledges funding from the University of Warwick and EPSRC through grant EP/R513374/1.
NZ was supported by the EPSRC grant EP/R024456/1

\section{Trees, Wild expansion and Butcher series}\label{sec:wild}

The linchpin of our argument is a precise control on the Wild expansion of
the solution $ u_{\ve} $ to the equation \eqref{e:ac2}.
Wild expansions were popularised in the
context of stochastic PDEs by Hairer's seminal work \cite{Hairer2013SolvingKPZ},
and are originally attributed
to the work of Wild \cite{wild_1951}.
A Wild expansion is an expansion of a solution to a parabolic PDE or an ODE in
terms of iterated integrals. The terms in such an expansion are naturally
indexed by rooted trees, which in our setting are associated -- similarly to Feynman
diagrams -- to integrals involving the heat kernel and the correlation
function of the noise $ \eta_{\ve} $. Wild expansions are also naturally linked
to Butcher series' \dash see for example \cite{Gub_ram} and the many references
therein \dash which allow for a tidy bookkeeping of the coefficients appearing in the Wild
expansion of a solution to an ODE. This section is devoted to establish all such
connections rigorously, and in a manner that will be useful to our analysis.

\subsection{Finite rooted trees}\label{sec:trees}
We start by introducing basic concepts concerning trees.
We will work with finite, rooted trees. A tree $ \tau $ is a
connected, undirected planar graph
that contains no cycle.
We  denote by $\mV(\tau)$ the set of vertices of a tree $\tau$ and by
$\mE(\tau)$ the set of its edges.
A finite rooted tree is a tree with a finite number of vertices, and
one particular vertex $ \mf{o} \in \mV (\tau) $ singled out as the
\textbf{root}. Rooted trees induce a partial order on the set of vertices $
\mV (\tau) $, by writing $ v \prec w $ if the unique path from $ w $ to the
root $ \mf{o} $ passes through $ v $, for any pair $ v, w \in \mV(\tau) $
with $ v \neq w $. In particular, if $ v \prec w $ we say that $ w $ is a \textbf{descendant} of
$ v $ or that $v$ is an \textbf{ancestor} of $w$, and we say that $ v $ is a \textbf{leaf} if it has no descendants.
The closest ancestor $v \in \mV ( \tau)$ of $w \in \mV ( \tau) \setminus \mf{o}$ is called \textbf{parent} of $w$,  we write
$\mathfrak{p}(w)=v$.
The set of leaves of a tree $\tau$ will be denoted by $\mL(\tau)$ and we define
$\ell(\tau):=|\mL(\tau)|$ the total number of leaves in a tree.
The collection of vertices of a tree which are not leaves will be
called {\bf inner vertices}, it will be denoted by $\mI(\tau)$ and
its cardinality is defined to be $i(\tau)$.
An exception to this convention will be made when the root is
the only vertex of the tree, in which case
it will considered as a leaf and thus {\it not} an inner vertex.
The cardinality of all the vertices of a tree
$\tau$ will be denoted by $|\tau|$ and we have that
$|\tau|=\ell(\tau)+i(\tau)$.
Finally, we call the {\bf degree} of a vertex the number of incident edges.

As a convention, we draw trees (from now on always
finite and rooted) growing upwards, out of the root of the tree, which is placed at its
bottom. For example, the following are trees with inner nodes coloured
white (and circled red) and leaves coloured black:
\begin{align}\label{treeexample}
   \mathbf{1}  \,,\ \<0>\,,\  \<3>\,,\  \<113_noise>\, ,\ \<10>\,,\  \<2>\,,\  \<210_noise>\,,\  \<1130_noise>\,.
\end{align}
Here $ \mathbf{1} $ is the empty tree, which is different from the single root $ \<0>
$\,.
Next, we will be working with {\it unordered} trees, which means, for
example, that the three trees below will be considered to be identical:
\begin{equation*}
\<113_noise> \qquad \<131_noise> \qquad \<311_noise> \,\,.
\end{equation*}
We will denote by $ \mT $ the set of all finite, unordered, rooted
trees. By convention $ \mT $ contains also the empty tree $ \mathbf{1} $.

In our setting, large trees arise naturally from smaller ones by
combining them, as we will now explain.
Suppose that
$\tau_1,...,\tau_n\in \mT$ are given.
Then we define the grafted tree
$\tau=[\tau_1 \cdots \tau_n]$ which is
built by connecting
the roots of the trees $\tau_1,...,\tau_n\in \mT$, by means of
new edges, to a new common vertex, which acts as the root of the new tree $\tau$.
The following graphical representation is perhaps the best explanation
of this construction:
\begin{equation}\label{e_plant_def}
\begin{aligned}
[\tau_1 \cdots \tau_n] =\begin{tikzpicture}[baseline=2]
			\node[root] (0) at (0,0) {};
			\draw (0) to (-.5,0.4) node at (-.6, 0.6) {\scalebox{0.8}{$\tau_1$}} ;
			\draw (0) to (-.2,0.4) node at (-.2,.6) {\scalebox{0.8}{$\tau_2$}} ;
			\draw (0) to (0,0.4) node at (0.2,.5) {\tiny$\cdots$} ;
			\draw (0) to (0.2,0.4) ;
			\draw (0) to (0.5,0.4) node at (0.6,.6) {\scalebox{0.8}{$\tau_n$}} ;
		\end{tikzpicture}.
\end{aligned}
\end{equation}
Here are some examples of graftings of trees:
    \begin{align*}
	[\, \mathbf{1}\,] =\<0>\,, \qquad
        [ \, \<0> \, ] = \<10>\,, \qquad
        [ \; \<10> \  \<0> \;]= \<21>
        \qquad \text{or}\qquad
        [\; \<3>\  \<0> \  \<0> \;]=\<113_noise> \;.
    \end{align*}
We observe that the empty
tree $\mathbf{1}$ can be ignored in a grafting (unless it is the only
tree): $[\mathbf{1}\ \tau_{1} \cdots \tau_{n} ]= [\tau_{1} \cdots \tau_{n}]$.
Next, one obtains the space of finite,
unordered rooted trees from the space of
finite rooted trees by quotienting through the equivalence relation that identifies
$$[\tau_1 \cdots \tau_n] \equiv [\tau_{\upsigma(1)} \cdots \tau_{\upsigma(n)}]$$
for every $\upsigma \in S_n$ (the group of permutations on $n$ indices)
and any choice of $ n \in \NN $ and trees $ \tau_{1}, \cdots, \tau_{n}
$ and the same for any subtree of a given tree.

In particular, due to the cubic nonlinearity that characterises the Allen--Cahn
equation, we will be dealing with {\bf sub-ternary} trees:
\begin{align*}
    \mT_{\leqslant 3}
    :=
    \{
    \tau \in \mT\,:\, \text{any inner node in $\tau$ has at most $3$
descendants}
    \}\,,
\end{align*}
and its subset $\mT_3$ of {\bf ternary} trees:
\begin{align*}
    \mT_{3}
    :=
    \{
    \tau \in \mT \setminus \{ \mathbf{1} \}\,:\, \text{any inner node in $\tau$
has exactly $3$ descendants} \}\,.
\end{align*}
For example, $\mT_{\leqslant 3}$ contains all the trees appearing in
\eqref{treeexample},
and the second, third and fourth trees in the above list
additionally lie in $\mT_3$.
We work with the conventions that
the empty tree belongs to $ \mT$ and $\mT_{\leqslant 3} $, but
\emph{not} $ \mT_{3} $
and that the ``single node'' tree $ \tau = \<0> $ belongs to $ \mT_{3} $
(here by convention the root counts as a leaf and not as an inner node).

It will be convenient to introduce some terminology for subtrees of sub-ternary trees. We will
call the  tree $\<3r>$ a {\bf trident},
$\<2r>$ a {\bf cherry} and $\<1r>$ a {\bf
lollipop}. Further, when these trees are embedded into
larger trees, we will call the root vertices of these components (marked in red
white here) respectively the {\bf basis of the trident}, {\bf basis of the cherry}
or {\bf basis of the lollipop}.
Finally, we note that for ternary trees $\tau \in \mT_3$, the number of leaves $\ell(\tau)$ and the number of inner vertices
$i(\tau)$ satisfy the relation $\ell(\tau)=2i (\tau)
+1$.

We close this subsection by introducing several important
quantities:
\begin{itemize}
\item[{\bf 1.}] {\bf Symmetry factors.} For any $\tau\in \mT$ we write
$s(\tau) \in \NN $ for the {\it symmetry factor} associated to the tree.
This amounts to the cardinality of the symmetry
group  associated with ${\tau}$.
More precisely, if we assign a label to each vertex of the tree $ \tau
$, then $ s(\tau) $ is the number of permutations of the labels that leave the
structure of rooted unordered tree invariant.
It is given by the following recursive formula. First, set $s(\<0>)=s(\mathbf{1})=1$. Then,
if $ \tau $ is of the form ${\tau} = [(\tau_1)^{k_1} \cdots
(\tau_n)^{k_n}]$ for pairwise distinct ${\tau_i}$'s, each one appearing
$k_i$ times and with $ \tau_{i} \neq \mathbf{1} $, for $i=1,...,n$, then
\begin{align}\label{eq_def_symfactor}
    s({\tau}) = \prod_{i=1}^n k_i!\; s({\tau_i})^{k_i}\,.
\end{align}
Since any rooted tree can be constructed by grafting together strictly
smaller trees, the above defines the symmetry factor for all rooted
trees.

\item[{\bf 2.}] {\bf Tree factorials.} Similarly we define the tree
factorial $ \tau ! \in \NN $ for any $ \tau \in \mT $.
For the empty tree $\mathbf{1}$ we define $\mathbf{1}!=1$ and, iteratively,
for a tree  ${\tau }= [\tau_1 \cdots \tau_n]$ we define $\tau !$ by
\begin{align*}
    \tau ! = | \tau | \, \tau_1 ! \cdots  \tau_n ! \, ,
\end{align*}
where  $|{\tau}|=1+|\tau_1| + \cdots |\tau_n|$
is the total number of vertices of the tree $\tau$.
We observe that the tree factorial of a linear
tree over $n$ vertices
(that is, the tree over $ n $ vertices in which every inner node has exactly one
descendant), is equal to $n!$. Thus, the notion of tree factorial generalises the
usual notion of factorial.

\item[{\bf 3.}]{\bf Tree differentials.} For an analytic function $h
\colon \RR \to \RR$, we define recursively its
{\it tree differential} (sometimes called {\it elementary
differential}) $h^{(\tau)}: \RR \to \RR$ as follows:
For all $y\in\RR$, we define $h^{(
\mathbf{1})}(y) = y$ for the empty tree $\mathbf{1}$ and  we set
$h^{( \<0>)}(y) =h(y) $ for the single node tree $\<0>$. Then, for an arbitrary tree $ \tau \in
\mT\setminus \{\mathbf{1},\<0>\}$ such that ${\tau} = [\tau_1 \cdots
\tau_n]$, $ \tau_{i} \neq \mathbf{1}$, we define
inductively
\begin{equation}\label{e:def_elDiff}
\begin{aligned}
  h^{ ( \tau)} (y) =  h^{(n)} (y)\, \prod_{j=1}^n h^{(
\tau_{j})}  (y)\;, \quad  \text{for all} \quad y \in \RR \;,
\end{aligned}
\end{equation}
where $h^{(n)}$ is the usual $n^{th}$ derivative of $h$.
\end{itemize}

\subsection{Butcher series} \label{sec:butcher}
In this section we will review how trees are used to index series
expansions of solutions to ordinary differential
equations. In particular, the kind of expansion that we are
interested in goes under the name of {\it Butcher series} (or {\it B-series} for short) \cite{But63,HaWa74}.
As we have already mentioned, the structure of
such series expansions in
combination with the structure of our Wild expansion
plays an important role in our analysis. In particular, this
structure lies behind the identification of the limiting
fluctuation strength $ \sigma_{ \hat{\lambda}} $ in Theorem~\ref{thm:main}.
To start our brief discussion on Butcher series, consider an analytic function $h\colon \RR \to \RR$ and
the differential equation
\begin{align*}
  \frac{ \ud y}{\ud \zeta}=h(y)\,, \quad \forall \zeta>0
   \quad \text{and} \quad y(0)=y_0 \in \RR\,.
\end{align*}
Then  the solution $y(\zeta)$ can be expressed, locally around
$\zeta =0$, through the following series:
\begin{equation}\label{e:bs}
    y(\zeta) =
    \sum_{{\tau} \in \mT} \frac{h^{( \tau)} (y_0)}{{\tau}!\, s({\tau})}
\zeta^{|{\tau}|} =: B_h(\zeta, y_{0}) \,, \qquad \forall \zeta \in [0, \zeta_{\star} )  \;,
\end{equation}
where $ \zeta_{\star} > 0 $ depends on $ h $ and $ y_{0} $. The sum
runs over all rooted, unordered trees (including the empty tree $\mathbf{1}$)
and $s (\tau),\tau !$ and $ h^{( \tau)}$ have been defined in Section
\ref{sec:trees} above.
For details on this derivation we refer to \cite{But63,HaWa74,butcherbook}, see also
\cite[Theorem 5.1]{Gub_ram}.
In any case, at the heart of \eqref{e:bs} lies the identity (see
for example \cite[Theorem 311C]{butcherbook})
\begin{equation}\label{e_tay_as_b}
\begin{aligned}
\frac{y^{(n)}(0)}{n!}
=
\sum_{ \substack{ \tau \in \mT \\ | \tau| =n}} \frac{h^{( \tau)}(y_{0})}{
\tau!\, s ( \tau)}\,,
\end{aligned}
\end{equation}
which allows to express the solution of the ODE in terms of the Butcher series,
whenever its Taylor series centered at zero converges
absolutely.
In this work, the following ODEs are of particular relevance:
\begin{equation}\label{e_bseries_examp}
\begin{aligned}
\begin{cases}
\dot{y}&=-y^3\\
y (0)&=1
 \end{cases}
\qquad \text{ and} \qquad
\begin{cases}
\dot{\overline{y}}&=\overline{y}^3\\
\overline{y}(0)&=1
 \end{cases} \,,
\end{aligned}
\end{equation}
both of which admit explicit solutions
$y(\zeta)=(1+2\zeta)^{-1/2}$ and $\overline{y}(\zeta)=(1-2\zeta)^{-1/2}$, respectively.
Observe that both solutions are real analytic at $0$ with radius
of convergence
$\tfrac{1}{2}$, so that
\eqref{e:bs}
holds with $\zeta_{\star}= \tfrac{1}{2}$.
Strikingly $ \overline{y} $
explodes at $ \zeta = 1/2 $, whereas $ y $ is defined for all times.
Our approach to treating the coefficients of the Butcher series does not
allow to distinguish the behaviour of the solution $ y $ from that of $
\overline{y} $ and  leads ultimately
 to one of the requirements in
Theorem~\ref{thm:main} for the coupling constant $ \hat{\lambda} $ to be sufficiently small.
Overcoming this issue would
require to take into account the sign of the terms in the Butcher series
(or avoid the series entirely), and this lies beyond the reach of our current proofs.

\subsection{Wild expansion}
In this subsection we establish some basic properties concerning the
Wild expansion of the solution $ u_{\ve} $ to \eqref{e:ac2}.
By convolving with the heat kernel, we can explain the heuristics that lead formally to the derivation of
the Wild expansion to \eqref{e:ac2}.
We write the solution of
equation \eqref{e:ac2} in its mild formulation
\begin{align}\label{Pic1}
u_\ve(t,x)
&=  P_t^{(\mathfrak{m})} \eta_\ve(x) - \int_0^t \Big(P^{(\mathfrak{m})}_{t-s}
u_\ve^3(s,\cdot) \Big) (x)\,\ud s  \\
&= \int_{\RR^2} p^{(\mathfrak{m})}_{t}(x-y) \eta_\ve (y) \ud y
- \int_0^t\int_{\RR^2} p^{(\mathfrak{m})}_{t-s}(x-y) u_\ve^3(s,y) \ud s \, \ud y \notag \\
&= \int_0^t\int_{\RR^2} p^{(\mathfrak{m})}_{t-s}(x-y) \eta_\ve (y) \delta_0(s) \ud s  \ud y
- \int_0^t\int_{\RR^2} p^{(\mathfrak{m})}_{t-s}(x-y) u_\ve^3(s,y) \ud s  \ud y  \notag\,,
\end{align}
where in the last line we  rewrote the first integral
using a delta function at time $ s=0 $.
We now introduce our first diagrammatic notation:
\begin{align}\label{Feyn-tree1}
\<10>_{\, \ve}\,(t,x) := \int_0^t\int_{\RR^2} p^{(\mathfrak{m})}_{t-s}(x-y) \, \delta_0(s) \,
\eta_\ve (y) \, \ud s \ud y =  P^{(\mathfrak{m})}_t \eta_\ve(x) \;,
  \end{align}
with the right-hand side interpreted in the It\^o sense.
Here we associate to the lollipop a random function in the
following way.
We assign the time-space variable $(t,x)$ to the root,
the time-space variable $(s,y)$ as well as the weight $\delta_0(s) \, \eta_\ve (y)$ to the leaf and the kernel
$p_{t-s}^{(\mathfrak{m})}(x-y)$ to the connecting edge. Finally, we
integrate over the variables associated to all nodes except the root.
In other words, the edge represents a time-space convolution between the heat kernel and the weight of the
leaf, evaluated at the time-space variables assigned to the root.
Therefore, we can rewrite \eqref{Pic1} as follows:
\begin{align}\label{Pic2}
u_\ve(t,x)
&=\<10>_{\, \ve}\,(t,x)  - \int_0^t \Big(P_{t-s}^{(\mathfrak{m})} u_\ve^3(s,\cdot) \Big) (x)\,\ud s .
\end{align}
Writing explicitly the arguments of the functions appearing in \eqref{Pic2} can be
cumbersome, so to shorten the notation we will equivalently write
\begin{equation*}
\begin{aligned}
u_{\ve} = \<10>_{\, \ve}  - P^{(\mathfrak{m})} \ast  u^3_{\ve} \;,
\end{aligned}
\end{equation*}
where $ P^{(\mathfrak{m})} \ast \varphi = \int_{\RR \times \RR^{2}} p^{(\mathfrak{m})}_{t-s}(x-y)
\varphi_{s} (y)  \ud y \ud s $ denotes space-time convolution with the kernel
$ p^{(\mathfrak{m})} $.
Now, we can iterate this
description of $ u_{\ve} $ by inserting the identity for
$u_\ve$ in \eqref{Pic2} into the right-hand
side of the expression itself:
\begin{align}\label{Pic3}
u_\ve
&=\<10>_{\, \ve}  - P^{(\mathfrak{m})} \ast \Big( \<10>_{\, \ve}
- P^{(\mathfrak{m})} \ast u_\ve^3 \Big)^3  \;.
\end{align}
Expanding the cube on the right-hand side
above, we find the expression
\begin{equation} \label{e:Pic4}
\begin{aligned}
u_\ve
= \<10>_{\, \ve}  - P^{(\mathfrak{m})} \ast \<10>_{\, \ve}^{\, 3}
&+3 P^{(\mathfrak{m})} \ast \left( \<10>_{\, \ve}^2 \cdot
P^{(\mathfrak{m})} * u_\ve^3 \right) \\
&-3 P^{(\mathfrak{m})} \ast \left(  \<10>_{\, \ve} \cdot
\Big(P^{(\mathfrak{m})} \ast u_\ve^3\Big)^2 \right) + P^{(\mathfrak{m})} \ast \Big( P^{(\mathfrak{m})}\ast u_\ve^3\Big)^3 \;.
\end{aligned}
\end{equation}
In order to motivate the representations we are after, let us just focus on the
second term of the above expression.
We use the following representation of the cubic power
\begin{align*}
\<3>_{\! \ve}\,(t,x) := \Big( \<10>_{\, \ve}\,(t,x)\Big)^3,
\end{align*}
where we have glued together three lollipops at a common root, thus forming a
trident whose basis is associated to time-space coordinates
$(t,x)$. We then introduce the planted trident:
\begin{align*}
\<30_3>_{\! \ve}\, (t,x)
& := \int_0^t P_{t-s_1}^{(\mathfrak{m})} \Big( \<10>_{\, \ve}\,(s_1,\cdot)\Big)^3 (x) \ud s_1 \\
& =\int_0^t \Big( P_{t-s_1}^{(\mathfrak{m})} \<3>_{\! \ve}\,(s_1,\cdot)\Big)(x) \ud s_1 =
P^{(\mathfrak{m})} \ast  \<3>_{\! \ve} (t, x)\;,
\end{align*}
which allows us to rewrite \eqref{e:Pic4} in terms of
\begin{equation} \label{e:Pic5}
\begin{aligned}
u_\ve
= \<10>_{\, \ve}  - \<30_3>_{\! \ve}
&+3 P^{(\mathfrak{m})} \ast \left( \<10>_{\, \ve}^2 \cdot
P^{(\mathfrak{m})} * u_\ve^3 \right) \\
&-3 P^{(\mathfrak{m})} \ast \left(  \<10>_{\, \ve} \cdot
\Big(P^{(\mathfrak{m})} \ast u_\ve^3\Big)^2 \right) + P^{(\mathfrak{m})} \ast \Big( P^{(\mathfrak{m})}\ast u_\ve^3\Big)^3 \;.
\end{aligned}
\end{equation}
We observe that, crucially, the first two terms are rather explicit: they live
in a finite inhomogeneous Wiener chaos and we are able to control
them with tools from stochastic analysis. Of course, we can continue this
expansion at will and iterating the procedure above we obtain formally an
expression for $ u_{\ve} $ that we call Wild expansion. The elements of this
expansion are given in the next
definition.

\begin{definition}\label{def:Wild}
To any tree $ \tau \in \mT_{3} $ and $ \ve \in (0, \tfrac{1}{2})$ we associate a random function
 $(t, x) \mapsto X_\ve^{ \tau} (t, x) $ for $ (t, x) \in (0, \infty) \times
\RR^{2}$ as follows. For $ \tau
= \<0>$ we set
\begin{align*}
    X_{\ve}^{\, \<0>} (t, x) &:= P_{t}^{(\mathfrak{m})} \eta_\varepsilon (x) \,.
\end{align*}
Then, iteratively, assuming we have defined $X_{\ve}^{\tau_1}, X_{\ve}^{\tau_2}, X_{\ve}^{\tau_3}$
for trees $\tau_1,\tau_2,\tau_3 \in \mT_3$,
we define $X_\ve^\tau$ for the tree $\tau=[\tau_1 \  \tau_2 \ \tau_3] \in \mT_3$ as
\begin{align*}
    X_\ve^{ \tau} (t, x)&:=
     -a( \tau)  P^{(\mathfrak{m})} \ast \big(\,
X^{\tau_1}_{\ve} X^{\tau_2}_{\ve}X^{
\tau_3}_{\ve} \, \big)  \,,
\end{align*}
where $ a( \tau) $ is the combinatorial factor
\begin{align}\label{def:a-n}
    a(\tau) = \ \begin{cases}
    1\;,\qquad \text{if} \quad \tau_1=\tau_2=\tau_3\;,\\
    3\;,\qquad \text{if exactly two of the trees $\tau_1,\tau_2,\tau_3$
coincide}\;, \\
    6\;,\qquad \text{if all trees $\tau_1,\tau_2,\tau_3$ are distinct}\;,
    \end{cases}
\end{align}
which appears because we are
considering unordered trees.
The Wild expansion associated to the
Allen--Cahn equation \eqref{e:ac2}
is then defined as the (formal) series $\sum_{\tau\in \mT_3} X^\tau_\ve.$
\end{definition}

Now we will connect the terms $ X^{\tau}_{\ve} $ from
Definition~\ref{def:Wild} to explicit stochastic integrals. This will follow
in two steps: first we generalize the diagrammatic representation that we have
started to introduce in \eqref{e:Pic5}. Then we show (see
Lemma~\ref{lem_expansion_factor} below) how an iterated stochastic integral
relates, up to an appropriate combinatorial factor, to the associated element of the Wild
expansion.

To this end, for any tree $\tau\in \mT_3$ with vertices $\mV( \tau)$,
partitioned into leaves $\mL(\tau)$, inner vertices $\mI(\tau)$, and
edges $\mE(\tau)$, we associate the Wiener--Stratonovich integral \cite[Remark 7.38]{Janson1997}
\begin{equation}\label{e:tau-int}
[\tau]_\ve (t,x)=
\int_{D_{t}^{\mV (\tau)}}
\prod_{u \in \mV( \tau)}
p_{s_{\mathfrak{p}(u)} - s_{u}}^{(\mathfrak{m})}(y_{\mathfrak{p}(u)} - y_{u})\ud s_{\mI(\tau)} \ud y_{\mI(\tau)}
 \prod_{v \in \mL(\tau)}
\delta_{0}(\ud s_v)\eta_\ve( \ud y_v)\;,
\end{equation}
where $[ \tau]$ denotes the planted version of $ \tau$, cf. \eqref{e_plant_def}.
Here
\begin{equation} \label{e:Dt}
\begin{aligned}
 D_{t} := [0, t] \times \RR^{2}
\end{aligned}
\end{equation}
and $D_{t}^{\mV (\tau)} $ is the Cartesian
product of $ D_{t} $ over the index set $\mV( \tau)$.
Moreover, we recall that
$\mathfrak{p}(u)$ denotes the unique parent of $u$ with
$\mathfrak{p}(\mf{o}_{ \tau}) = \mf{o}_{[ \tau]}$ the node
associated to the time-space point $(t,x)$, and $ \mf{o}_{[ \tau]}$ being the root of
the planted tree $ [\tau] $.
Note that in the integral \eqref{e:tau-int} we are \emph{not} integrating over the root
variable of the tree $ [\tau] $ (which is assigned to the point $
(t, x) $), but integrate instead over all other vertices of $ [\tau] $, i.e.\
all vertices of $ \tau$. 
Further, observe that with this definition, by construction, if $ [\tau] = [ [\tau_{1}] \,
[\tau_{2}] \, [\tau_{3}]] $, then
\begin{equation} \label{e:new}
\begin{aligned}
[ \tau]_{ \ve} =P^{(\mathfrak{m})} \ast ( [ \tau_{1}]_{ \ve} [ \tau_{2}]_{
\ve} [ \tau_{3}]_{ \ve}) \;.
\end{aligned}
\end{equation}
Next, we remark that in general, a term $ X^{\tau}_{\ve} $ of the Wild expansion is related to the
Wiener integral represented by the associated planted tree $ [\tau]_{\ve} $ via a combinatorial
factor $ c_\tau \in \NN$ as $X_\ve^{ \tau} =c_{\tau}\cdot [\tau]_{\ve}$.
With the formulation above we have for example:
\begin{equation*}
\begin{aligned}
X_\ve^{\scalebox{0.7}{\<3>} }= -\scalebox{0.9}{ \<30_3>}_{\!\ve}\;,
\qquad X_\ve^{\! \! \scalebox{0.6}{\<113_noise>}} = 3 \scalebox{0.9}{\<113_noise_0>}_{\!\ve} \;.
\end{aligned}
\end{equation*}
The factors $ c_\tau$ appear because we are considering unordered trees (and
ultimately because of the commutative property of the product).
It will be important to obtain a precise expression for $ c_{\tau} $, and 
this is the objective of the remainder of this section.
Our eventual expression for $ c_{\tau} $ contains tree derivatives
with respect to a ``trimmed'' version of $ \tau $. Therefore, we start by
introducing a ``trimming'' operator on trees.
\begin{definition}\label{def:trim}
We call the map
\begin{equation}\label{trimming}
\begin{aligned}
\mathscr{T}: \mT_{3} \rightarrow \mT_{\leqslant 3} \;, \qquad
\mathscr{T}(\tau) = \red{\tau} \;,
\end{aligned}
\end{equation}
the
\emph{trimming} operator,
where $\mathscr{T}(\tau)$ is the tree that is spanned by the inner nodes of $\tau$, i.e.
$ \mathscr{T} $ ``cuts off'' all the leaves and the edges attached to them.
\end{definition}
To lighten later notation, we have also used the chromatic notation by which $
\red{\tau} = \mathscr{T}(\tau)$,
for example
\begin{align*}
    \<r0> = \mathscr{T}(\<0>) = 1\,, \quad \mathscr{T}\left(
\,  \vcenter{\hbox{\scalebox{0.85}{\<3>}}} \, \right)  = \<0>\,, \quad \text{and }\quad
    \mathscr{T}\left( \,  \vcenter{\hbox{\scalebox{0.85}{\<113_noise>}}} \,
\right) = \<10>\,.
\end{align*}
The reader should have the following pictorial description of $
\mathscr{T}$ in mind
\begin{equation}\label{e:trimExample}
\begin{aligned}
 \vcenter{\hbox{
        \begin{tikzpicture}[scale=0.45]
	\draw[thick, densely dotted]  (-1.05,2.75) -- (-1.05, 3.35) node[dot] {} ;
	\draw[thick, densely dotted] (-1.45, 3.35) node[dot] {}  -- (-1.05,2.75) -- (-0.65, 3.35) node[dot] {} ;
	\draw[thick, densely dotted]  (-1.75,2.75) node[dot] {} -- (-1.75,2) ;
	\draw[thick, densely dotted]  (-1.75,2) -- (-2.45,2.75) node[dot] {};
	\draw (-1.05,2.75) node[idot] {} -- (-1.75,2) ;
	\draw[thick, densely dotted]  (1.75,2.75) node[dot] {} -- (1.75,2) ;
	\draw[thick, densely dotted]  (1.05,2.75) node[dot] {} -- (1.75,2) -- (2.45,2.75) node[dot] {} ;
	\draw (-1.75,2) node[idot] {}  -- (0,1) -- (1.75,2) node[idot] {} ;
	\draw[thick, densely dotted] (0,1) node[idot] {}  -- (0,2) node[dot] {};
	\end{tikzpicture}  }}
	\mapsto
        \vcenter{\hbox{
        \begin{tikzpicture}[scale=0.45]
	\draw (-1.05,2.75) node[idot] {} -- (-1.75,2) ;
	\draw (-1.75,2) node[idot] {}  -- (0,1) -- (1.75,2) node[idot] {} ;
	\draw (0,1) node[idot] {};
	\end{tikzpicture}  }}
\mapsto
\vcenter{\hbox{
        \begin{tikzpicture}[scale=0.4]
	\draw (-0.7,3) node[dot] {} -- (-0.7,2) ;
	\draw (-0.7,2) node[idot] {}  -- (0,1) -- (0.7,2) node[dot] {} ;
	\draw (0,1) node[idot] {};
	\end{tikzpicture}  }}\,,
\end{aligned}
\end{equation}
where we again coloured leaves black in the last expression according to the
chosen convention.
A first result regarding the trimming operation guarantees that it is a
bijection between finite families of ternary and sub-ternary trees of the
following form, for arbitrary \( N \in \NN \)
\begin{equation}\label{e:def_T3N}
\begin{aligned}
 \mT_3^N &  = \{ \tau \in \mT_{3}  \; \colon \; i(\tau) \leqslant N \} \subseteq
\mT_{3}\,, \\
\mT_{\leqslant 3}^{N} & = \{ \tau \in \mT_{\leqslant 3}  \; \colon \; | \tau
| \leqslant N \} \subseteq \mT_{\leqslant 3} \,.
\end{aligned}
\end{equation}
The next lemma summarises this and other properties of $\mathscr{T}$.
\begin{lemma}\label{lem_trimming_bij}
The following hold.
\begin{enumerate}[label=\textup{(\roman*)}]
\item The map $\mathscr{T}$ is a bijection from $\mT_3^N$ to $
\mT^{N}_{\leqslant 3} $, for every $N\in \NN$.
\item Let $ \tau \in \mT_{3}\setminus \{\<0>\}$ and $ \tau_{1}, \tau_{2} , \tau_{3} \in
\mT_{3} $ such that $ \tau = [ \tau_{1}\ \tau_{2}\ \tau_{3}]$, then
\begin{equation*}
\begin{aligned}
\red{ \tau} = [ \red{ \tau_{1}}\ \red{\tau_{2}}\ \red{ \tau_{3}}]\,.
\end{aligned}
\end{equation*}
In other words, trimming via $\mathscr{T}$ and grafting via $[ \ldots]$ commute.
\end{enumerate}
\end{lemma}

\begin{proof}
By definition of $\mT_{3}^N$, its image under $\mathscr{T}$ is a subset of $
\mT^{N}_{\leqslant 3}$. On the other hand, for any
$\tau\in \mT_{\leqslant 3}^{N}$ we can construct a $\sigma \in \mT_{3}^N$ such
that $ \T(\sigma) = \tau $ as follows.
To each node that has $ 3 - k $ descendants, for $ k \in \{ 1,2,3 \} $, we
append exactly $ k $ lollipops $\<10>$, so that in the new tree that node has
exactly three outgoing edges.
The constructed tree $\sigma$ lies in $\mT_3^N$,
since $ i ( \sigma) = | \tau| \leqslant N$ and every inner node of $ \sigma$ has exactly
three descendants.
Moreover, it satisfies
$ \mathscr{T}(\sigma)=\tau$.
This concludes the proof of the first part of the statement.

In order to prove (ii), let $ \tau \in \mT_{3} \setminus \{\<0>\}$ and $ \tau_{1},
\tau_{2}, \tau_{3} \in \mT_{3} $ such that $ \tau = [ \tau_{1}\
\tau_{2}\ \tau_{3}]$.
Now, because $\mathscr{T}$ does not act on the root of $ \tau$, cf.
Definition~\ref{def:trim}, we necessarily
have
\begin{equation*}
\begin{aligned}
\red{ \tau}=
\mathscr{T}( \tau) = \left[ \mathscr{T}( \tau_{1})\ \mathscr{T}( \tau_{2})\
\mathscr{T}( \tau_{3}) \right] =  [ \red{ \tau_{1}}\ \red{\tau_{2}}\ \red{
\tau_{3}}]\,.
\end{aligned}
\end{equation*}
In order to avoid confusion, let us discuss explicitly the case where
 $ \tau_{i}= \<0>$ and thus $\red{ \tau_{i}}= \mathbf{1}$, for some $i \in \{1,2,3\}$.
Without loss of generality let us assume that $ \tau_{2}= \<0> $, as in the
example  \eqref{e:trimExample} displayed above.
Then, by convention of $[\ldots]$, we have
\begin{equation*}
\begin{aligned}
\red{ \tau}= [ \red{\tau_{1}}\ \red{\tau_{2}}\ \red{\tau_{3}}]
=
[\red{\tau_{1}}\ \red{\tau_{3}}]\,.
\end{aligned}
\end{equation*}
This identity propagates to $\red{ \tau}= [\red{ \tau_{1}}]$ if additionally $
\tau_{3}= \<0>$.
Moreover, in the most extreme case $ \tau = \<3>$, this reduces further to
$\red{ \tau} =  \<0> = [\mathbf{1}]= [\mathbf{1}\ \mathbf{1}\ \mathbf{1} ]$.
\end{proof}

The following result establishes the link between the Butcher series and the
Wild expansion, according to our definitions.

\begin{lemma}\label{lem_expansion_factor}
The following identity holds for $ h(y) = - y^{3} $ and any $ \tau \in
\mT_{3}$:
\begin{align*}
    X^{\tau}_{ \ve}(t,x)
    =
    \frac{h^{(\red{\tau})}(1)}{s(\red{\tau})}
    [\tau]_{ \ve}(t,x)\,, \quad \forall \ve \in (0,\tfrac{1}{2})\,, \ (t,x) \in
(0, \infty) \times \RR^{2} \;,
\end{align*}
with the symmetry factor $s(
\red{ \tau})$ and elementary differential $h^{( \red{ \tau})}$  defined in
\eqref{eq_def_symfactor} and \eqref{e:def_elDiff}, respectively.
\end{lemma}

\begin{proof}
The statement is true for $X^{\<0>}$ since $ h^{( \<r0>)} (1) =
h^{(\mathbf{1})}(1) = 1 $. Now we proceed by induction. Assume that the statement is true for all trees
$\tau\in \mT_3$ with $|\tau|\leqslant n$, for some given $ n \in \NN $,
and let $\tau\in \mT_3$ be a tree with $n+1$ nodes such that $\tau = [\tau_1\
\tau_2\ \tau_3]$, $ \tau_{i} \in \mT_{3}$.
Furthermore, observe that the combinatorial factor $ a(\tau) $ appearing in
Definition~\ref{def:Wild} can be expressed as follows:
\begin{equation*}
\begin{aligned}
a(\tau) = 3! \frac{s(\tau_1) s(\tau_2) s(\tau_3)}{s(\tau)} \;,
\end{aligned}
\end{equation*}
with the symmetry factor $ s(\tau) $ defined in
\eqref{eq_def_symfactor}. Or, in other words
\begin{equation}\label{e:symFac_ratio}
\begin{aligned}
\frac{s(\tau_1) s(\tau_2) s(\tau_3)}{s(\tau)}
 =\
    \begin{cases}
    \tfrac{1}{6},\ \ \text{if} \ \ \tau_1=\tau_2=\tau_3 \;,\\
    \tfrac{1}{2},\ \ \text{if exactly two of the trees $\tau_1,\tau_2,\tau_3$
coincide}\;, \\
    1,\ \ \text{if all trees $\tau_1,\tau_2,\tau_3$ are distinct}\;.
    \end{cases}
\end{aligned}
\end{equation}
Therefore, from the definition of the terms of the Wild
expansion, see again Definition~\ref{def:Wild}, we have
\begin{equation}\label{e:Wild2But_supp}
\begin{aligned}
 X^{\tau}_{ \ve}
    &= -6\frac{s(\tau_1) s(\tau_2) s(\tau_3)}{s(\tau)}\,
    P^{(\mathfrak{m})} \ast (X_{ \ve}^{\tau_1}X_{ \ve}^{\tau_2}X_{ \ve}^{\tau_3})\\
    &=
    -6\frac{{s(\red{\tau})}\, s(\tau_1) s(\tau_2) s(\tau_3)}{s(\tau)\,
{s(\red{\tau_1})s(\red{\tau_2})s(\red{\tau_3})}}
    \frac{{h^{( \red{\tau_1})}(1)h^{(\red{\tau_2})}(1)h^{( \red{\tau_3})}(1)}}{{s(\red{\tau})}}
    [\tau]_{ \ve}\,,
\end{aligned}
\end{equation}
by our induction hypothesis, and by \eqref{e:new}.
Now, let $k \in \{0,1,2,3\}$ be the number of trees $ \tau_{i}$ satisfying
$ \tau_{i} \neq \<0>$. As we consider unordered trees, we can write
\begin{equation*}
\begin{aligned}
 \tau = [ \tau_{1}\  \cdots\  \tau_{k}\ \<0>\ \cdots \ \<0>]\,,
\end{aligned}
\end{equation*}
without loss of generality. Thus
\begin{equation*}
\begin{aligned}
-6\frac{{s(\red{\tau})}\, s(\tau_1) s(\tau_2) s(\tau_3)}{s(\tau)\,
{s(\red{\tau_1})s(\red{\tau_2})s(\red{\tau_3})}}
& =
-6
\frac{s(\red{\tau})
\, s(\tau_1) \cdots s(\tau_{k})
}{s(\tau)\, s(\red{\tau_1})\cdots s(\red{\tau_{k}})} \\
& =
- \frac{6}{(3-k)!}
\frac{ s(\tau_1) \cdots s(\tau_{k})}{s([\tau_{1}\  \cdots\ \tau_{k}])}
\frac{{s(\red{\tau})}}{{s(\red{\tau_1})\cdots s(\red{\tau_{k}})}} \;,
\end{aligned}
\end{equation*}
since $s
(\<0>)=s(\mathbf{1}) = 1$ and by using the fact that
\begin{equation*}
\begin{aligned}
 s ( \tau)=
(3-k)!\, s(\<0>)^{3-k} \, s([\tau_{1}\ \cdots\ \tau_{k}])
 = (3-k)!\, s([\tau_{1}\ \cdots\ \tau_{k}])\,,
\end{aligned}
\end{equation*}
which is a direct consequence of the symmetry factor's definition
\eqref{eq_def_symfactor}.
Next, due to commutativity of the trimming and grafting cf. Lemma~\ref{lem_trimming_bij}(ii), we have
$\red{ \tau} = [\red{ \tau_1}\  \red{  \tau_2}\  \red{\tau_3}]$.
In particular this implies, using Lemma~\ref{lem_trimming_bij}(i), that
if $j$ of the $\tau_i$'s agree, then also $j$ of the
$\red{\tau_i}$'s agree.
Therefore, since as in \eqref{e:symFac_ratio}, the ratio
\begin{equation*}
\begin{aligned}
\frac{ s(\tau_1) \cdots s(\tau_{k})}{s([\tau_{1}\  \cdots\ \tau_{k}])}
\end{aligned}
\end{equation*}
only depends on the number of
identical subtrees, we obtain
\begin{equation*}
\begin{aligned}
-6\frac{{s(\red{\tau})}\, s(\tau_1) s(\tau_2) s(\tau_3)}{s(\tau)\,
{s(\red{\tau_1})s(\red{\tau_2})s(\red{\tau_3})}}
=
-\frac{3!}{(3-k)!}
=
h^{(k)}(1)
 \,.
\end{aligned}
\end{equation*}
Overall, \eqref{e:Wild2But_supp} can therefore be rewritten as
\begin{equation*}
\begin{aligned}
X^{\tau}_{ \ve}
      &=
    \frac{h^{(k)}(1)
\,{h^{( \red{\tau_1})}(1) \cdots h^{( \red{\tau_k})}(1)}}{{s(\red{\tau})}}
    [\tau]_{ \ve}
=
\frac{h^{( \red{\tau})}(1) }{{s(\red{\tau})}}
[\tau]_{ \ve}
\,,
\end{aligned}
\end{equation*}
where we used the definition of the elementary differential \eqref{e:def_elDiff} together
with the fact that $h^{(\<r0>)}(1)=h^{(\mathbf{1})}(1)=1$.
This concludes the proof.
\end{proof}

With this, we have introduced all the elements which allow us to discuss the
proof of Theorem~\ref{thm:main}, without entering into technical details.
These are deferred to later sections and require the introduction of additional
tools.

\section{Outline and proof of main result}\label{sec:prf-main}

The first step towards the proof of Theorem~\ref{thm:main} is the
analysis of single terms in our Wild expansion. The key result in this
direction (Proposition~\ref{prop_single_tree}) is presented in the
upcoming Section~\ref{sec_outline}. The proof of
Theorem~\ref{thm:main} is then carried out in Section~\ref{sec:prf-thm}.

\subsection{From Wild expansion to single-tree estimates}\label{sec_outline}
Formally, the solution of \eqref{e:ac2} can be represented in terms of the Wild series expansion as
\begin{equation*}
\begin{aligned}
u_\ve &= \sum_{\tau\in \mT_3} X^\tau_\ve = \sum_{\tau\in \mT_3^N } X^\tau_\ve +
\sum_{\tau \in \mT_{3} \setminus \mT_3^N } X^\tau_\ve
=: u_\ve^N + \big( u_\ve-u_\ve^N\big)\,,
\end{aligned}
\end{equation*}
where $u_\ve^N$ is defined to be the series $\sum_{\tau\in \mT_3^N }
X^\tau_\ve$ truncated at level $ N \in \NN $. Unlike the full Wild expansion,
$ u_{\ve}^{N} $ is well defined as it is a finite sum.
The structure of the Allen--Cahn equation and in particular the fact that the
non-linearity $-u^3$ is monotone in $u$, allows to circumvent a direct treatment
of  the infinite part of the series. More precisely,
we have that  $u_\ve^N$ solves 
\begin{align}\label{eq:wild_exp_pde}
    \partial_t u_\ve^{N} = \frac{1}{2} \Delta u_\ve^{N}+ \mathfrak{m} \,
u_{ \ve}^{N}- (u_\ve^{N})^3
    + R_\ve^{N}\,, \quad u_\ve^{N}(0,\cdot) = \eta_{\ve} ( \cdot)\,,
\end{align}
where the error term $ R_\ve^{ N} $ depends only on trees at the ``boundary'' of
$ \mT^{N}_{3} $:
\begin{equation}\label{e:def-R}
    R_\ve^{N} =
    \sum_{\substack{ \tau_1, \tau_2, \tau_3\in \mT_3^N \\ [\tau_1\, \tau_2 \, \tau_3]\notin \mT_3^N}}
    X_\ve^{\tau_1}X_\ve^{ \tau_2}X_\ve^{ \tau_3}\,.
\end{equation}
Here we used the fact that any tree
in $\tau\in \mT_3\setminus \{\<0> \}$ can be written recursively as
$\tau =[\tau_1 \, \tau_2\, \tau_3]$, for some smaller $\tau_1, \tau_2,\tau_3
\in \mT_3$, hence, 
\begin{align*}
    \Big(\sum_{\tau\in \mT_3^N} X_\ve^{\tau} \Big)^3
    =
    \sum_{\substack{\tau \in \mT_3^N\\ \tau = [\tau_1\, \tau_2\, \tau_3]}}
    a(\tau)\,
    X_\ve^{\tau_1}X_\ve^{\tau_2}X_\ve^{\tau_3}
    +R_\ve^{N}\,.
\end{align*}
Utilising a maximum principle in combination with more structural estimates,
which we will describe in detail below,
we are able to control the error of the approximation as follows.

\begin{proposition}\label{prop_max_principal}
Let $ \hat{\lambda}>0$ and $T>0$ satisfy
\begin{equation}\label{e_choice_lam}
\begin{aligned}
 \Hat{\lambda} e^{ \overline{\mathfrak{m}} \,
T}  < \frac{1}{10\sqrt{C}} \;,
\end{aligned}
\end{equation}
with $ \overline{\mathfrak{m}}=  \max \{ \mf{m}\;, 0 \}$ and  
\begin{equation}\label{e_def_CT}
\begin{aligned}
C:=
\frac{6e^{ 2 + 2 \pi}}{\pi}
\,.
\end{aligned}
\end{equation}
 Then
 uniformly over all
$(t,x) \in (0,T]\times \RR^{2}$, $ \ve \in (0, \tfrac{1}{T}
\wedge \tfrac{1}{2})$ and $N \leqslant
\lfloor \log{ \tfrac{1}{\ve}} \rfloor$
    \begin{align}\label{errror-estim}
        \sqrt{\log \tfrac{1}{\ve}} \left\| \, u_\ve^{ N}(t,x) - u_{\ve}(t,x) \,
\right\|_{L^{2}(\PP)} \leqslant
\frac{ C_{0}}{\log{ \frac{1}{\ve}}}
\frac{ \big( \sqrt{C}\hat{\lambda} e^{ \overline{\mathfrak{m}}\, t}\big)^{N}}{\ve}
\,,
    \end{align}
with $C_{0}=
C_{0}(T, \mathfrak{m}, \hat{\lambda}) \in (0, \infty)
$ the constant defined in
\eqref{e_def_C0}.
\end{proposition}

The proof of this proposition is deferred to Section~\ref{sec:prf-thm}.
In order for this estimate to help us prove Theorem~\ref{thm:main},
we would like the right-hand side of \eqref{errror-estim} to vanish as $ \ve
\to 0 $. This forces us to choose a cut-off level $N= N_{ \ve}$ that grows to
infinity 
as $ \ve \to 0 $, and a suitably small coupling constant $ \hat{\lambda} $
satisfying \eqref{e_choice_lam}.
In particular, we fix the cut-off level $ N_{\ve} $ given by
 \begin{align}\label{Nchoice}
 N_\ve= \left\lfloor \log \tfrac{1}{\ve} \right\rfloor \,.
 \end{align}
We note that this estimate runs along the same lines as \cite[Proposition
3.15]{HLR22}, with the fundamental difference that we need to push the estimate here to be uniform over
a growing $ N $, more precisely $ N \leqslant \lfloor \log{ \tfrac{1}{\ve} }
\rfloor $. Indeed in \cite{HLR22}, for $\eta_\ve$ scaled (in the
two-dimensional setting) as $\ve^{1-\alpha} p_\ve\star\eta$ with $\alpha\in
(0,1)$, the bound that one obtains is of the form:
\begin{equation*}
\begin{aligned}
\left\| \, u_\ve^{ N}(t,x) - u_{\ve}(t,x) \,
\right\|_{L^{2}(\PP)} \leqslant C(N, \mf{m}, t) \ve^{2 - 3\alpha} \left(
e^{\mf{m} t} \ve^{1- \alpha} (1+\log{(t \ve^{-2})}) \right)^{N}\,,
\end{aligned}
\end{equation*}
for some constant $ C $. The parameter $ \alpha \in
(0, 1)$ modulates the sub-critical level of the noise. We see here that if $ \alpha < 1 $, then in order to make the
left-hand side small, it suffices to choose $ N $ finite, but sufficiently
large. Instead, in our setting which corresponds to $\alpha=1$, this bound degenerates.
and it is replaced, instead, by \eqref{errror-estim}. Having a control on the error, then, leads to the need of a growing  
choice of $N_\ve$.
We collect the essential elements of this discussion in the following remark. 

\begin{remark}\label{rem:needWild}
The estimate in Proposition~\ref{prop_max_principal} is crucial to understand
our approach. First, this bound forces us to control the Wild expansion up to
the (diverging) level $ N_{\ve} = \lfloor \log{ \tfrac{1}{\ve} } \rfloor $, uniformly over
$ \ve $. It is for this reason that we require the precise estimates on the
stochastic integrals associated to trees up to size $ | \tau | \leqslant
N_{\ve} $, which are at the heart of our work. Second, this bound imposes us to
work with a \emph{small} coupling
constant $ \hat{\lambda} $ and small times (if $ \mf{m} > 0 $), although we
expect our main result to hold for all times and coupling constants.
\end{remark}
Given the error estimate in Proposition~\ref{prop_max_principal} above, the next task
is to identify the convergence of the truncated sequence $u_\ve^{N_\ve}= \sum_{\tau\in
\mT^{N_\ve}_3} X^\tau_\ve$. This convergence is very delicate,
in particular because the number of terms in the sum now grows with $ \ve $.
The next proposition contains the key estimate that allows us to overcome this
issue.
\begin{proposition}\label{prop_single_tree}
Let $T>0$ and $ \hat{\lambda}>0$. Then,
uniformly over any $ \ve \in (0, \tfrac{1}{T}
\wedge \tfrac{1}{2}) $,  $ \tau \in
\mT_{3}^{N_{\ve}} $, with $ N_{\ve} = \lfloor
\log{\tfrac{1}{\ve}}\rfloor $, and uniformly over all $ ( t,x)
\in [0, T] \times \RR^{2}$, we have
\begin{align*}
    & \Bigg\| \, \sqrt{\log \tfrac{1}{\ve} } \cdot X^{\tau}_\varepsilon(t,x)-
\frac{h^{( \red{ \tau})} (1)}{\red{\tau}!\, s(\red{\tau})}
\left(\frac{3\Hat{\lambda}^2}{2\pi}\right)^{|\red{\tau}|} \hat{\lambda}
e^{ \mathfrak{m} \, t}
 P_{t +
\ve^{2}} \eta (x) \, \Bigg\|_{L^2(\PP)}\\
    &  \qquad \leqslant
 \frac{|h^{( \red{ \tau})} (1)|}{\red{\tau}!\, s(\red{\tau})}
\left(C \, \hat{\lambda}^{2} e^{2 \overline{\mathfrak{m}}\, t} \right)^{ |\red{ \tau}|}
\frac{e^{2 |\mathfrak{m}|\, t}+|\log{(t+
\ve^{2})}| + \sqrt{\log \tfrac{1}{\ve}} }{ 2\log{ \tfrac{1}{\ve}}}
 \frac{ \hat{\lambda} e^{\mathfrak{m}\, t}
}{\sqrt{4 (t + \ve^{2})}}  \,,
\end{align*}
where $ \red{ \tau}$ is the trimmed tree $\mathscr{T} ( \tau)$ defined in
\eqref{trimming} and $C$ is the constant defined in
\eqref{e_def_CT}.
\end{proposition}
The above Proposition both identifies the limit of $ \sqrt{\log \tfrac{1}{\ve}
} \, X^{\tau}_\varepsilon$ and gives a quantitative estimate of its rate of
convergence.
The proof of Proposition~\ref{prop_single_tree} is at the heart of this article
and can be found at the end of Section~\ref{sec:estimates-single-tree}. It builds on
all the results that are derived on the way. We
highlight that the bound we obtain is uniform over all trees $ \tau \in
\mT_{3} $ with $ | \tau | =O( \log{ \tfrac{1}{\ve}} )$. This is rather
remarkable: As $ | \tau | $ grows, every tree consists of a growing number of
components living in distinct homogeneous chaoses and it requires precise
estimates to bound all of them at once.
It is thus crucial that the right hand-side is summable over $ \tau $ and decays for $ \ve \to 0 $,
in order to justify the asymptotics
\begin{align*}
\sqrt{\log \tfrac{1}{\ve}}
u_{ \ve}^{N_{\ve}}(t,x)=
\sqrt{\log \tfrac{1}{\ve}} \sum_{\tau\in\mT_3^{N_\ve}} X^{\tau}_\ve(t,x)
&\sim \hat{\lambda} \, \sum_{\tau\in\mT_3^{N_\ve}}  \frac{h^{( \red{ \tau})} (1)}{\red{\tau}!\, s(\red{\tau})}
\Big(\frac{3\Hat{\lambda}^2}{2\pi}\Big)^{|\red{\tau}|}
e^{\mathfrak{m}\, t}
P_{t +
\ve^{2}} \eta (x)  \\
&\sim \hat{\lambda} \,\left\{ \sum_{\tau\in\mT_3}  \frac{h^{( \red{ \tau})} (1)}{\red{\tau}!\, s(\red{\tau})}
\Big(\frac{3\Hat{\lambda}^2}{2\pi}\Big)^{|\red{\tau}|} \right\}
P_{t}^{(\mathfrak{m})} \eta (x)\, ,
\end{align*}
as $ \ve \to 0$.
The series appearing in the expression above equals
\begin{equation}\label{sigmalambda}
  \sum_{\tau\in\mT_3}  \frac{h^{( \red{ \tau})} (1)}{\red{\tau}!\, s(\red{\tau})}
\Big(\frac{3\Hat{\lambda}^2}{2\pi}\Big)^{|\red{\tau}|} 
 =  y\Big(\frac{3\hat\lambda^2}{2\pi} \Big)
=\sigma_{\hat\lambda}
\end{equation}
where $y( \cdot)$ is the solution of the differential equation $\dot{y}=- y^{3}$ with
$y(0)=1$, cf. \eqref{e_bseries_examp}, which leads to the expression for the
limiting variance in Theorem~\ref{thm:main}.
Making the previous argument rigorous
is the content of the following proposition, which is the final step towards
the proof of Theorem~\ref{thm:main}.

\begin{proposition}\label{prop_exp_close_to_but}
Let $ \hat{\lambda}>0$ and $T>0$ satisfy
\begin{equation}\label{e:condition}
\begin{aligned}
 \hat{\lambda} e^{ \overline{\mathfrak{m}}\, T} < \frac{1}{\sqrt{2C}} \;,
\end{aligned}
\end{equation}
with $ \overline{\mathfrak{m}}=  \max \{ \mf{m}\;, 0 \}$ and $C$ be the
constant defined in \eqref{e_def_CT}.
Then for all $ (t, x) \in
(0, T] \times \RR^{2}  $
    \begin{align*}
        \lim_{\ve\to0}  \Bigg\| \sqrt{\log{\tfrac{1}{\ve}}} \, u_\ve^{
N_{ \ve}}(t, x)  - \hat{\lambda}\,  \sigma_{\Hat{\lambda}} P_t^{(\mathfrak{m})} \eta(x)
\Bigg\|_{L^2(\PP)}= 0\,,
    \end{align*}
where $\sigma_{\hat\lambda}$ is as in \eqref{sigmalambda}.
\end{proposition}
We will provide the proof of Proposition~\ref{prop_exp_close_to_but} at the end
of the
next subsection.
Before we pass to the proof of Theorem~\ref{thm:main}, let us explain  the
structure that underlies our main estimate, which is contained in Proposition
\ref{prop_single_tree}.

\subsubsection{Outline of the structure governing the terms $X^\tau_\ve$}
\label{sec:outline}
Each Wiener integral $X^\tau_\ve$ is an element of an {\it inhomogeneous
Wiener chaos} (cf. \cite{Nualart}). An element of an inhomogeneous Wiener chaos
admits a decomposition into its {\it homogeneous Wiener chaos} components.
These projections are obtained via all possible
{\it pairwise contractions} of noises. In  \eqref{e:tau-int}, this means that $X^\tau_\ve$ can be written as a sum
over all possible subsets of pairs of leaves $\kappa \subset \mL(\tau)\times \mL(\tau) $, where
for each $(u,v)\in \kappa$ we replace the product of noise terms
$\eta_\ve( y_u) \eta_\ve( y_v) $ by $\EE [ \eta_{\ve}(y_{u}) \eta_{\ve}(y_{v}) ]=   \lambda_{ \ve}^{2} p_{2\ve^2}(y_u-y_v)$. In the $\ve\to0$ limit
this corresponds to
``contracting'' the noises to the ``diagonal'' $y_v=y_u$, as $ p_{2 \ve^{2}}$ approximates a
Dirac \(\delta\). A diagrammatic example of a possible contraction  (or homogeneous Wiener chaos) configuration
 is the following:
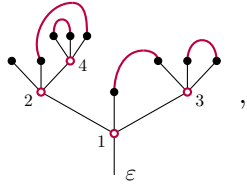
\begin{equation*}
\begin{aligned}
        \vcenter{\hbox{
        \begin{tikzpicture}[scale=0.55]
	\draw[purple, thick] (-1.45,3.35) to[out=90,in=180] (-1.25,3.75);
        \draw[purple, thick] (-1.25,3.75) to[out=0,in=90] (-1.05,3.35);
	\draw[purple, thick] (1.75,2.75) to[out=90,in=180] (2.1,3.25);
        \draw[purple, thick] (2.1,3.25) to[out=0,in=90] (2.45,2.75);
	\draw[purple, thick] (0,2) to[out=90,in=180] (0.7,3);
        \draw[purple, thick] (0.7,3) to[out=0,in=90] (1.05,2.75);
	\draw[purple, thick] (-1.75,2.75) to[out=120,in=180] (-1,4.15);
        \draw[purple, thick] (-1,4.15) to[out=0,in=90] (-0.65,3.35);
	\draw (-1.05,2.75) -- (-1.05, 3.35) node[dot] {} ;
	\draw (-1.45, 3.35) node[dot] {}  -- (-1.05,2.75) -- (-0.65, 3.35) node[dot] {} ;
	\draw (-1.75,2.75) node[dot] {} -- (-1.75,2) ;
	\draw (-1.05,2.75) node[idot] {} -- (-1.75,2) -- (-2.45,2.75) node[dot] {} ;
	\draw (1.75,2.75) node[dot] {} -- (1.75,2) ;
	\draw (1.05,2.75) node[dot] {} -- (1.75,2) -- (2.45,2.75) node[dot] {} ;
	\draw (-1.75,2) node[idot] {}  -- (0,1) -- (1.75,2) node[idot] {} ;
	\draw (0,1) -- (0,2) node[dot] {};
        \draw  (0,0)   -- (0,1) node[idot] {}  ;
	\node at (-.3,0.8) {\scalebox{0.6}{$1$}};
	\node at (-2.05,1.8) {\scalebox{0.6}{$2$}};
	\node at (2.05,1.8) {\scalebox{0.6}{$3$}};
	\node at (-.75,2.65) {\scalebox{0.6}{$4$}};
        \node at (.4,0) {\scalebox{0.8}{$\ve$}};
\end{tikzpicture}  }} \;,
\end{aligned}
\end{equation*}
which lies in the first homogeneous Wiener chaos (as only one leaf is
uncontracted) and is therefore Gaussian.
At the heart of our approach lies the observation that in a contracted tree as
depicted above, the eventual contribution to the limit is determined by the
total number of so-called \emph{$ 1 $-cycles} that one can iteratively extract from the tree. These
are cycles that involve two leaves and one inner vertex, such as the one
incident on the inner vertex $ 4 $ (or equivalently $ 3 $) above. The reason for their importance is a
time-space decoupling. Indeed the contribution of the cycle based at $ 4
$ can be computed explicitly as follows:
     \begin{equation} \label{e:time-kernel}
	\lambda_{ \ve}^{2}
     \int_{(\RR^2)^2} p_{s_4}^{(\mathfrak{m})}(x_4-y) p_{\ve^2}(y-y')
p_{s_4}^{(\mathfrak{m})}(x_4-y') \ud y \ud y' = \lambda_{ \ve}^{2}
e^{2\mathfrak{m}\, s_{4}}
 p_{2(s_4+\ve^2)}(0)\;,
     \end{equation}
by using the Chapman--Kolmogorov equations, where we denoted the
time-space variable associated to the node $i$ by $(s_{i},x_{i})$.
Notably, the result
is independent of the space variable $ x_{4} $. We can therefore replace the
original kernel through a time-dependent kernel based at $ x_{4} $ (all the
rest unchanged), which graphically we visualize as a red loop around $
x_{4} $:
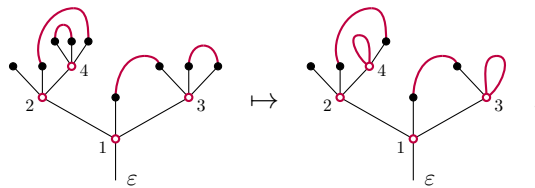
\begin{equation*}
\begin{aligned}
        \vcenter{\hbox{
        \begin{tikzpicture}[scale=0.55]
	\draw[purple, thick] (-1.45,3.35) to[out=90,in=180] (-1.25,3.75);
        \draw[purple, thick] (-1.25,3.75) to[out=0,in=90] (-1.05,3.35);
	\draw[purple, thick] (1.75,2.75) to[out=90,in=180] (2.1,3.25);
        \draw[purple, thick] (2.1,3.25) to[out=0,in=90] (2.45,2.75);
	\draw[purple, thick] (0,2) to[out=90,in=180] (0.7,3);
        \draw[purple, thick] (0.7,3) to[out=0,in=90] (1.05,2.75);
	\draw[purple, thick] (-1.75,2.75) to[out=120,in=180] (-1,4.15);
        \draw[purple, thick] (-1,4.15) to[out=0,in=90] (-0.65,3.35);
	\draw (-1.05,2.75) -- (-1.05, 3.35) node[dot] {} ;
	\draw (-1.45, 3.35) node[dot] {}  -- (-1.05,2.75) -- (-0.65, 3.35) node[dot] {} ;
	\draw (-1.75,2.75) node[dot] {} -- (-1.75,2) ;
	\draw (-1.05,2.75) node[idot] {} -- (-1.75,2) -- (-2.45,2.75) node[dot] {} ;
	\draw (1.75,2.75) node[dot] {} -- (1.75,2) ;
	\draw (1.05,2.75) node[dot] {} -- (1.75,2) -- (2.45,2.75) node[dot] {} ;
	\draw (-1.75,2) node[idot] {}  -- (0,1) -- (1.75,2) node[idot] {} ;
	\draw (0,1) -- (0,2) node[dot] {};
        \draw  (0,0)   -- (0,1) node[idot] {}  ;
	\node at (-.3,0.8) {\scalebox{0.6}{$1$}};
	\node at (-2.05,1.8) {\scalebox{0.6}{$2$}};
	\node at (2.05,1.8) {\scalebox{0.6}{$3$}};
	\node at (-.75,2.65) {\scalebox{0.6}{$4$}};
        \node at (.4,0) {\scalebox{0.8}{$\ve$}};
\end{tikzpicture}  }}
\ & \mapsto \
  \vcenter{\hbox{
        \begin{tikzpicture}[scale=0.55]
	\draw[purple, thick] (-1.05,2.75) to[out=150,in=180] (-1.3,3.5);
        \draw[purple, thick] (-1.3,3.5) to[out=0,in=90] (-1.05,2.75);
	\draw[purple, thick] (1.75,2) to[out=90,in=180] (2,3);
        \draw[purple, thick] (2,3) to[out=0,in=30] (1.75,2);
	\draw[purple, thick] (0,2) to[out=90,in=180] (0.7,3);
        \draw[purple, thick] (0.7,3) to[out=0,in=90] (1.05,2.75);
	\draw[purple, thick] (-1.75,2.75) to[out=120,in=180] (-1,4.15);
        \draw[purple, thick] (-1,4.15) to[out=0,in=90] (-0.65,3.35);
	\draw (-1.05,2.75) -- (-0.65, 3.35) node[dot] {} ;
	\draw (-1.75,2.75) node[dot] {} -- (-1.75,2) ;
	\draw (-1.05,2.75) node[idot] {} -- (-1.75,2) -- (-2.45,2.75) node[dot] {} ;
	\draw (1.05,2.75) node[dot] {} -- (1.75,2);
	\draw (-1.75,2) node[idot] {}  -- (0,1) -- (1.75,2) node[idot] {} ;
	\draw (0,1) -- (0,2) node[dot] {};
        \draw  (0,0)   -- (0,1) node[idot] {}  ;
	\node at (-.3,0.8) {\scalebox{0.6}{$1$}};
	\node at (-2.05,1.8) {\scalebox{0.6}{$2$}};
	\node at (2.05,1.8) {\scalebox{0.6}{$3$}};
	\node at (-.75,2.65) {\scalebox{0.6}{$4$}};
         \node at (.4,0) {\scalebox{0.8}{$\ve$}}; 
\end{tikzpicture}  }} \;.
\end{aligned}
\end{equation*}
Similar approach has been followed for removing the cycle rooted at $3$.
Now integrating over $ x_{4} $, again by the Chapman--Kolmogorov equations as
\begin{equation}\label{e:time-kernel2}
\begin{aligned}
\lambda_{ \ve}^{2}
     &\int_{(\RR^2)^2} p_{s_2-s_4}^{(\mathfrak{m})}(x_4-x_2) p_{s_4}^{(\mathfrak{m})}(y'-x_4)
     p_{s_2} ^{(\mathfrak{m})}(y-x_2)
     p_{2\ve^2}(y-y')
 \ud y \ud y' \ud x_4\\
&= \lambda_{ \ve}^{2}
e^{2\mathfrak{m}\, s_{2}}
 p_{2(s_2+\ve^2)}(0)\;,
\end{aligned}
\end{equation}
 we
are left with the product between the time-only-dependent kernels in
\eqref{e:time-kernel}  and \eqref{e:time-kernel2}  and the kernel associated to the tree in which we remove
node $2$ and $ 4 $ (and the cycles that are incident to  them). This tree now has $ 1 $-cycles based at nodes $ 2 $ and $4$.
We can follow the same procedure iteratively for the rest of the cycles as indicated in:
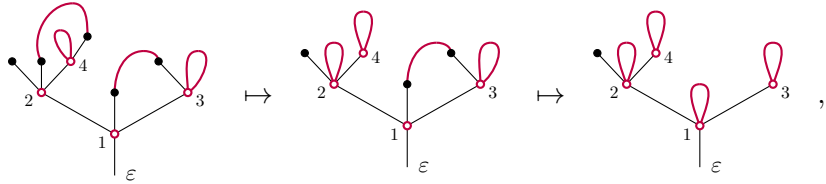
\begin{equation*}
\begin{aligned}
 \vcenter{\hbox{
        \begin{tikzpicture}[scale=0.55]
	\draw[purple, thick] (-1.05,2.75) to[out=150,in=180] (-1.3,3.5);
        \draw[purple, thick] (-1.3,3.5) to[out=0,in=90] (-1.05,2.75);
	\draw[purple, thick] (1.75,2) to[out=90,in=180] (2,3);
        \draw[purple, thick] (2,3) to[out=0,in=30] (1.75,2);
	\draw[purple, thick] (0,2) to[out=90,in=180] (0.7,3);
        \draw[purple, thick] (0.7,3) to[out=0,in=90] (1.05,2.75);
	\draw[purple, thick] (-1.75,2.75) to[out=120,in=180] (-1,4.15);
        \draw[purple, thick] (-1,4.15) to[out=0,in=90] (-0.65,3.35);
	\draw (-1.05,2.75) -- (-0.65, 3.35) node[dot] {} ;
	\draw (-1.75,2.75) node[dot] {} -- (-1.75,2) ;
	\draw (-1.05,2.75) node[idot] {} -- (-1.75,2) -- (-2.45,2.75) node[dot] {} ;
	\draw (1.05,2.75) node[dot] {} -- (1.75,2);
	\draw (-1.75,2) node[idot] {}  -- (0,1) -- (1.75,2) node[idot] {} ;
	\draw (0,1) -- (0,2) node[dot] {};
        \draw  (0,0)   -- (0,1) node[idot] {}  ;
	\node at (-.3,0.8) {\scalebox{0.6}{$1$}};
	\node at (-2.05,1.8) {\scalebox{0.6}{$2$}};
	\node at (2.05,1.8) {\scalebox{0.6}{$3$}};
	\node at (-.75,2.65) {\scalebox{0.6}{$4$}};
         \node at (.4,0) {\scalebox{0.8}{$\ve$}};
\end{tikzpicture}  }} 
\ & \mapsto \
  \vcenter{\hbox{
        \begin{tikzpicture}[scale=0.55]
	\draw[purple, thick] (-1.05,2.75) to[out=110,in=180] (-1.05,3.75);
        \draw[purple, thick] (-1.05,3.75) to[out=0,in=70] (-1.05,2.75);
	\draw[purple, thick] (1.75,2) to[out=90,in=180] (2,3);
        \draw[purple, thick] (2,3) to[out=0,in=30] (1.75,2);
	\draw[purple, thick] (0,2) to[out=90,in=180] (0.7,3);
        \draw[purple, thick] (0.7,3) to[out=0,in=90] (1.05,2.75);
	\draw[purple, thick] (-1.75,2) to[out=110,in=180] (-1.75,3);
        \draw[purple, thick] (-1.75,3) to[out=0,in=70] (-1.75,2);
	\draw (-1.05,2.75) node[idot] {} -- (-1.75,2) -- (-2.45,2.75) node[dot] {} ;
	\draw (1.05,2.75) node[dot] {} -- (1.75,2);
	\draw (-1.75,2) node[idot] {}  -- (0,1) -- (1.75,2) node[idot] {} ;
	\draw (0,1) -- (0,2) node[dot] {};
        \draw  (0,0)   -- (0,1) node[idot] {}  ;
	\node at (-.3,0.8) {\scalebox{0.6}{$1$}};
	\node at (-2.05,1.8) {\scalebox{0.6}{$2$}};
	\node at (2.05,1.8) {\scalebox{0.6}{$3$}};
	\node at (-.75,2.65) {\scalebox{0.6}{$4$}};
         \node at (.4,0) {\scalebox{0.8}{$\ve$}};
\end{tikzpicture}  }}
\ \mapsto \
  \vcenter{\hbox{
        \begin{tikzpicture}[scale=0.55]
	\draw[purple, thick] (-1.05,2.75) to[out=110,in=180] (-1.05,3.75);
        \draw[purple, thick] (-1.05,3.75) to[out=0,in=70] (-1.05,2.75);
	\draw[purple, thick] (1.75,2) to[out=110,in=180] (1.75,3);
        \draw[purple, thick] (1.75,3) to[out=0,in=70] (1.75,2);
	\draw[purple, thick] (0,1) to[out=110,in=180] (0,2);
        \draw[purple, thick] (0,2) to[out=0,in=70] (0,1);
	\draw[purple, thick] (-1.75,2) to[out=110,in=180] (-1.75,3);
        \draw[purple, thick] (-1.75,3) to[out=0,in=70] (-1.75,2);
	\draw (-1.05,2.75) node[idot] {} -- (-1.75,2) -- (-2.45,2.75) node[dot] {} ;
	\draw (-1.75,2) node[idot] {}  -- (0,1) -- (1.75,2) node[idot] {} ;
        \draw  (0,0)   -- (0,1) node[idot] {}  ;
	\node at (-.3,0.8) {\scalebox{0.6}{$1$}};
	\node at (-2.05,1.8) {\scalebox{0.6}{$2$}};c	\node at (2.05,1.8) {\scalebox{0.6}{$3$}};
	\node at (-.75,2.65) {\scalebox{0.6}{$4$}};
         \node at (.4,0) {\scalebox{0.8}{$\ve$}};
\end{tikzpicture}  }} \;,
\end{aligned}
\end{equation*}
removing all $ 1 $-cyclcs until there are none left.
Hence, in the last diagram we remain with a tree which
depicts an iterated (time-only) integral, with (time-only dependent) weights
$\lambda_{ \ve}^{ 2}e^{2\mathfrak{m}\, s_{i}}
 p_{2(s_i+\ve^2)}(0)$ at every node $i$.
Note that the ordering of the time variables within this integral is inherited from the tree
decorated by loops.
A crucial observation will be that only homogeneous chaos
configurations which share precisely this property, contribute in the limit $\ve\to 0$.
This will be the content of Section~\ref{sec:estimates-single-tree}. We also note that because
of the particular structure we have found, determining the eventual limiting
contribution is now a simpler task, as we are left with only a time-dependent
integral.

How to rigorously determine or estimate the contribution of contracted trees
through the removal of certain cycles is the content of the next sections.
We will formally introduce contractions (and pairings) and the notion of
${\rm v}$-cycles in Section~\ref{sec:exmpl}.
In Section~\ref{sec:estimates-single-tree} we will then finally prove
Proposition~\ref{prop_single_tree}.

\subsection{Proof of Theorem~\ref{thm:main}} \label{sec:prf-thm}
We are now ready to prove our main result, given the estimates in
Proposition~\ref{prop_max_principal} and Proposition~\ref{prop_exp_close_to_but},
the proofs of which are postponed to further below in the section.

\begin{proof}[Proof of Theorem~\ref{thm:main}]
We define
$ \hat{\lambda}_{\mathrm{fin}}:= \tfrac{1}{10 \sqrt{C}}$, where $C$ is the
positive constant defined in \eqref{e_def_CT}.
Let
 $ \hat{\lambda} \in (0, \hat{\lambda}_{\mathrm{fin}})$ and $T>0$ such that
\eqref{e_main_ass} is satisfied, which equivalently reads
\begin{equation}\label{e_m_supp}
\begin{aligned}
 \Hat{\lambda} e^{ \overline{\mathfrak{m}} \,
T}  < \frac{1}{10\sqrt{C}} \;.
\end{aligned}
\end{equation}
By the triangle inequality, for $ \ve \in (0,\tfrac{1}{T}
\wedge \tfrac{1}{2})$ and $(t,x) \in
(0,T] \times \RR^{2}$
\begin{equation*}
\begin{aligned}
\big\| \mathcal{U}_{\ve}(t, x) - \hat{\lambda}\, \sigma_{ \hat{\lambda}}  \,  P_{t}^{(\mf{m})} \eta (x) \big\|_{L^{2}(\PP)}
&\leqslant
\sqrt{\log\tfrac{1}{ \ve}}  \left\|  u_{\ve}(t,x)  - u_\ve^{ N_{
\ve}}(t,x)
\right\|_{L^{2}(\PP)}\\
&\qquad +
\Big\| \sqrt{\log\tfrac{1}{\ve}} \,\,
  u^{ N_{ \ve}}_{ \ve}(t, x) -  \hat{\lambda}\, \sigma_{\Hat{\lambda}} \,
P_t^{(\mf{m})} \eta(x) \Big\|_{L^{2}(\PP)}\,.
\end{aligned}
\end{equation*}
The second term on the right-hand side vanishes as $ \ve \to 0 $ by
Proposition~\ref{prop_exp_close_to_but}, since \eqref{e_m_supp} implies
\eqref{e:condition}.
On the other hand, by Proposition~\ref{prop_max_principal}, the first term is
upper bounded by
\begin{equation*}
\begin{aligned}
\sqrt{\log \tfrac{1}{\ve}}  \, \left\|  u_{\ve}(t,x)  - u_\ve^{ N_{ \ve}}(t,x)
\right\|_{L^{2}(\PP)}
\leqslant
\frac{ C_{0}}{\log{ \frac{1}{\ve}}}
\frac{ \big( \sqrt{C} \hat{\lambda} e^{ \overline{\mathfrak{m}}\,
T}\big)^{N_{ \ve}}}{\ve}
\,.
\end{aligned}
\end{equation*}
The blow-up on the right-hand side must be compensated, and here we will crucially use that
$N_{ \ve} \sim \log\tfrac{1}{\ve} $ as $ \ve \to 0$, so that we have a compensating effect from the term $
\big( \sqrt{C} \hat{\lambda} e^{ \overline{\mathfrak{m}}\,
T}\big)^{N_{ \ve}}$.
More precisely, by the choice of $T$ in \eqref{e_m_supp}, we have
\begin{equation*}
\begin{aligned}
- \log{\big( \sqrt{C} \hat{\lambda} e^{ \overline{\mathfrak{m}}\,
T}  \big)} \geqslant \log{10}  >1\,,
\end{aligned}
\end{equation*}
thus,
\begin{equation*}
\begin{aligned}
 \big(\sqrt{C} \hat{\lambda} e^{ \overline{\mathfrak{m}}\, T}\big)^{N_{\ve}}
\leqslant
\exp \left(\log{\big( \sqrt{C} \hat{\lambda} e^{ \overline{\mathfrak{m}}\,
T}\big)}
\left( \log{ \tfrac{1}{\ve}} -1 \right)\right)
 \leqslant
10\,
\ve^{ \log{10} }
\;, \qquad \forall \ve \in (0, \tfrac{1}{T}
\wedge \tfrac{1}{2})  \;.
\end{aligned}
\end{equation*}
Hence, we obtain that
\begin{equation*}
\begin{aligned}
\sqrt{\log \tfrac{1}{\ve}}  \left\|  u_{\ve}(t,x)  - u_\ve^{ N_{ \ve}}(t,x)
\right\|_{L^{2}(\PP)}
\leqslant
\frac{10\, C_{0}}{\log{ \frac{1}{\ve}}}
\ve^{(\log{10}) -1}
 \,,
\end{aligned}
\end{equation*}
which vanishes in the limit $ \ve \to 0 $. This concludes the proof.
\end{proof}

In the remainder of this section we prove
that the truncated Wild expansion $ u_{ \ve}^{ N_{ \ve}}$ indeed approximates the solution $
u_{ \ve}$ (Proposition~\ref{prop_max_principal})
and that $\sqrt{ \log{ \tfrac{1}{\ve}}}  u_{ \ve}^{N_{ \ve}}(t,x)$ is close to $ \hat{\lambda}\, \sigma_{ \hat{\lambda
} }P_{t} \eta(x)$  in $L^{2}(\PP)$ (Proposition~\ref{prop_exp_close_to_but}).
The proof of Proposition~\ref{prop_single_tree} will be given at the end of
Section~\ref{sec:estimates-single-tree}.

\begin{proof}[Proof of Proposition~\ref{prop_max_principal}]
Let $T>0$,  $(t,x) \in (0 , T] \times \RR^2$ and $ \ve \in (0,\tfrac{1}{T}
\wedge \tfrac{1}{2})$.
From \eqref{e:ac2} and \eqref{eq:wild_exp_pde}, we obtain that the difference $
w^{N}_{\ve} = u^{N}_{\ve} - u_{\ve} $ solves the equation
\begin{equation}\label{e:sup-prop2}
\begin{aligned}
\partial_{t} w^{N}_{\ve} =  \frac{1}{2} \Delta w^{N}_{\ve} + \mf{m} w^{N}_{\ve} -
(u_\ve^N)^3 + u_\ve^{3} + R_\ve^N\,, \quad w^{N}_{\ve} (0, \cdot)  = 0\,,
\end{aligned}
\end{equation}
with $ R_\ve^N $ defined in \eqref{e:def-R}. Defining $V^{N}_{\ve}(t,x)
:=\tfrac{(u_\ve^N)^3 - u_\ve^{3}}{u_\ve^N - u_\ve}$, we can write \eqref{e:sup-prop2} as
\begin{align*}
\partial_{t} w^{N}_{\ve}  =  \frac{1}{2} \Delta w^{N}_{\ve} + \mf{m}
w^{N}_{\ve}  - V^{N}_{ \ve} \cdot w^{N}_{\ve} + R_\ve^N\,, \quad w^{N}_{\ve}
(0, \cdot)  = 0\,.
\end{align*}
The Feynman--Kac formula \cite[Theorem 5.7.6]{karatzas1991brownian} then allows
to represent $ w^{N}_{\ve}$
as
\begin{align*}
w^{N}_{\ve} (t, x)=  \mathbf{E}_x\Bigg[\int_0^t R_\ve^N(t-s,\beta(s)) \exp\Big(
\mf{m} s -\int_0^s V^{N}_{ \ve }(s-r,\beta(r)) \ud r\Big) \ud s \Bigg] \;,
\end{align*}
where $\beta(\cdot)$ is a two dimensional Brownian path and $\mathbf{E}_x$ is the expectation with respect to it
when the path starts from $x\in\RR^2$.
Using the fact that $V^{N}_{ \ve}\geqslant 0$, which is due to the monotonicity of the
mapping $u\mapsto u^3$, we obtain that\footnote{We note that this estimate, via triangle inequality and dropping part of the exponential term
is not expected to be optimal and is the place where we lose. This  leads subsequently to the need of a growing $N_\ve$ and 
the restrictions to the time horizon.}
\begin{align*}
\big|\, u_\ve^N (t,x)- & u_\ve(t,x) \,\big| \\
&\leqslant \mathbf{E}_x\Bigg[ \int_0^t \big|R_\ve^N(t-s,\beta(s)) \big|
\exp\Big(\mf{m} \, s -\int_0^s V^{N}_{ \ve}(s-r,\beta(r)) \ud r\Big) \ud s \Bigg] \\
&\leqslant \mathbf{E}_x\Bigg[ \int_0^t e^{\mf{m}\, s}  \big| R_\ve^N(t-s,\beta(s))
\big|  \ud s \Bigg]\;.
\end{align*}
Writing the latter in terms of the heat kernel we conclude that
\begin{align*}
\big|\, u_\ve^N (t,x)- u_\ve(t,x) \,\big|
&\leqslant e^{ \overline{\mf{m}}\, t} \int_0^t \int_{\RR^d}  p_{t-s}(y-x)\, |R_\varepsilon^{N}(s,y)| \ud y
\ud s \;.
\end{align*}
Taking the $L^2(\PP)$-norm, we arrive at the bound that we will be working
with:
\begin{align}\label{max-prin-est}
     \| u^{N}_{\ve} (t,x)- u_\varepsilon(t,x) \|_{L^{2} (\PP)}
        & \leqslant e^{\overline{\mf{m}}\, t}\int_0^t   \| R_\varepsilon^{N}(s,0) \|_{L^{2} (\PP)}
\ud s\,,
\end{align}
where we have used that $ R^{N}_{\ve} $ is spatially homogeneous.
To continue, we use the definition of $R_\ve^N$ from \eqref{e:def-R}, the triangle inequality
and  H\"older's inequality, to obtain
\begin{align}\label{eq_supp1_maximPrinc}
    \|R_\varepsilon^{N}(s,0)\|_{L^2(\PP)}
    &\leqslant
    \sum_{\substack{ \tau_1, \tau_{2}, \tau_3\in \mT_3^{N} \\ [\tau_1\,
\tau_2 \, \tau_3]\notin \mT_3^{N}}}
    \|
    (X_\varepsilon^{\tau_1}X_\varepsilon^{\tau_2}X_\varepsilon^{\tau_3})(s,0)\|_{L^2(\PP)}\\
    &\leqslant
    \sum_{\substack{ \tau_1, \tau_{2}, \tau_3\in \mT_3^{N} \\ [\tau_1\,
\tau_2 \, \tau_3]\notin \mT_3^{N}}}
    \|
    X_\varepsilon^{\tau_1}(s,0)\|_{L^6(\PP)}
    \|
    X_\varepsilon^{\tau_2}(s,0)\|_{L^6(\PP)}
    \|
    X_\varepsilon^{\tau_3}(s,0)\|_{L^6(\PP)} \;. \nonumber
\end{align}
At this point we use
hypercontractivity, namely estimates of the $L^q(\PP)$-norm by the $L^p(\PP)$-norm, for $q>p>1$, for random variables in a fixed inhomogeneous Wiener
chaos \cite[Theorem 5.10]{Janson1997}. In our case it is important to quantify
the constant appearing in the hypercontractivity estimates in terms of the chaos level in
which the random variable lies. In particular, we will use the following
estimate, which is an immediate consequence of \cite[Remark~5.11]{Janson1997}:
\begin{align*}
    \| X_\varepsilon^{\tau_i}(s,y)\|_{L^6(\PP)} \leqslant
5^{\frac{\ell(\tau_i)}{2}}\,  \|
X_\varepsilon^{\tau_i}(s,y)\|_{L^2(\PP)}\,.
\end{align*}
To bound the $ L^{2} (\PP) $-norm, we will make use of
Proposition~\ref{prop_single_tree} to obtain
\begin{equation}\label{e:sup-prop3}
       \|X^{\tau}_\varepsilon(s,y)\|_{L^2(\PP)}
\leqslant \frac{\Tilde{c} (T,\mathfrak{m})}{ \sqrt{ \log{ \frac{1}{ \ve} }}}
\frac{|h^{( \red{ \tau})} (1)|}{\red{\tau}!\, s(\red{\tau})}
 \left( C\,
\Hat{\lambda}^2 e^{2 \overline{\mathfrak{m}} \, s }\right)^{|\red{\tau}|}
\frac{ \hat{\lambda} e^{ \mathfrak{m}\, s}}{2 \sqrt{s + \ve^{2}}}
   \,,
    \end{equation}
with  $ \Tilde{c} (T,\mathfrak{m}):= e^{ 2|\mathfrak{m}|\, T}+4$.
The verification of this bound is deferred to the bottom of this proof.
Assuming \eqref{e:sup-prop3} is true, and using the identity \( \ell( \tau_{i}) = 2|
\red{\tau_{i}}| +1\) (which holds since $ \tau \in \mT_{3}$), we obtain
\begin{align}\label{e:sup-prop3.1}
\| X_\varepsilon^{\tau_i}(s,y)\|_{L^6(\PP)}
 \leqslant 5^{|\red{\tau_i}|+ \frac{1}{2} } \frac{\Tilde{c} (T,\mathfrak{m})}{ \sqrt{ \log{ \frac{1}{ \ve} }}}
 \frac{|h^{( \red{ \tau_{i}})}
(1)|}{\red{\tau_{i}}!\, s(\red{\tau_{i}})}
 \left(C\, \Hat{\lambda}^2 e^{ 2 \overline{\mathfrak{m}}\, s}\right)^{|\red{\tau_{i}}|}
\frac{\hat{\lambda}  e^{ \mathfrak{m}\, s}}{2 \sqrt{s + \ve^{2}}}  \,.
\end{align}
Combining \eqref{max-prin-est} with \eqref{eq_supp1_maximPrinc} and \eqref{e:sup-prop3.1}, we therefore conclude
that
\begin{equation}\label{e:prf-sup5}
\begin{aligned}
        \sqrt{\log \tfrac{1}{\ve}}  &\cdot   \| u^{N}_{ \ve} (t,x)-
u_\varepsilon(t,x) \|_{L^{2} (\PP)} \\
& \leqslant \sqrt{\log \tfrac{1}{\ve}} \cdot   e^{ \overline{\mf{m}} t}
\int_0^{\infty}   \|
R^\varepsilon(s,0) \|_{L^{2} (\PP)} \ud s  \\
&\leqslant
\frac{\big(\sqrt{5} \hat{\lambda}\, \Tilde{c} (T,\mathfrak{m}) \big)^{3}\,
e^{4
\overline{\mathfrak{m}} \,t}}{8 \, \log{ \frac{1}{ \ve}}}   \left\{ \int_{0}^{\infty} \frac{1}{(s + \ve^{2})^{\frac{3}{2}}} \ud s \right\}
\cdot \\
& \qquad  \qquad \cdot \sum_{\substack{
\tau_1, \tau_{2}, \tau_3\in \mT_3^{N} \\ [\tau_1\, \tau_2 \, \tau_3]\notin \mT_3^{N}}}
    \prod_{i=1}^3  \left\{ \frac{|h^{( \red{ \tau_i})} (1)|}{\red{\tau_i}!\, s(\red{\tau_i})}
     \left( 5 C \Hat{\lambda}^2 e^{2 \overline{\mathfrak{m}} \, t} \right)^{|\red{\tau_i}|}
\right\}\,.
\end{aligned}
\end{equation}
At this point we notice that the time integral appearing in the last estimate blows up
polynomially in $ \ve $, since
\begin{align}\label{e:prf-prop4}
 \int_0^\infty  \frac{1}{(s+\varepsilon^2)^{\frac{3}{2}} } \ud s =
    \left[ -2(s+\varepsilon^2)^{-\frac{1}{2}} \right]_{s=0}^{s= \infty}
    = \frac{2
}{\varepsilon}\,.
\end{align}
On the other hand, for any $ \{ \tau_{i} \}_{i=1}^{3} $ such that $ [\tau_{1}
\ \tau_{2} \ \tau_{3}] \not\in \mT^{N}_{3} $ we have that
\begin{equation*}
\begin{aligned}
 N < i ( [\tau_{1} \ \tau_{2} \ \tau_{3}] ) & = i( \tau_{1} ) + i(
\tau_{2} ) + i(
\tau_{3} ) + 1
= | \red{\tau_{1}} | + | \red{\tau_{2}} | + | \red{\tau_{3}} | + 1  \;.
\end{aligned}
\end{equation*}
Moreover, by assumption,
$ \hat{\lambda}$ is sufficiently small to satisfy
$5\sqrt{ C}\, \Hat{\lambda} e^{ \overline{\mathfrak{m}} \,
t}  < \frac{1}{2}$.
Therefore, we can estimate the sum in the last line of \eqref{e:prf-sup5} as
follows. First, observe that
\begin{equation}\label{e:prf_supp10}
\begin{aligned}
\sum_{\substack{
\tau_1, \tau_{2}, \tau_3\in \mT_3^{N} \\ [\tau_1\, \tau_2 \, \tau_3]\notin
\mT_3^{N}}} &
    \prod_{i=1}^3
         \left\{ \frac{|h^{( \red{ \tau_i})} (1)|}{\red{\tau_i}!\, s(\red{\tau_i})}
        \left( 5 C \Hat{\lambda}^2 e^{2 \overline{\mathfrak{m}} \,
t} \right)^{|\red{\tau_i}|}
\right\} \\
& \leqslant  \big(\sqrt{C}\hat{\lambda} e^{ \overline{\mathfrak{m}}\, t}\big)^{N}\sum_{\substack{
\tau_1, \tau_{2}, \tau_3\in \mT_3^{N} \\ [\tau_1\, \tau_2 \, \tau_3]\notin
\mT_3^{N}}}
    \prod_{i=1}^3
         \left\{ \frac{|h^{( \red{ \tau_{i}})} (1)|}{\red{\tau_i}!\, s(\red{\tau_i})}
        \left( 5 \sqrt{C} \Hat{\lambda} e^{ \overline{\mathfrak{m}} \,
t} \right)^{|\red{\tau_i}|}
\right\} \;.
\end{aligned}
\end{equation}
Then we complete the remaining sums on the right-hand side to Butcher series:
\begin{equation*}
\begin{aligned}
\sum_{\substack{
\tau_1, \tau_{2}, \tau_3\in \mT_3^{N} \\ [\tau_1\, \tau_2 \, \tau_3]\notin
\mT_3^{N}}}
    \prod_{i=1}^3
         \left\{ \frac{|h^{( \red{ \tau_i})} (1)|}{\red{\tau_i}!\, s(\red{\tau_i})}
        \left(5 \sqrt{ C}\, \Hat{\lambda} e^{ \overline{\mathfrak{m}} \,
t}  \right)^{|\red{\tau_i}|}
\right\} & \leqslant  \left(
    \sum_{\substack{ \tau\in  \mT_3^{N}}}
        \frac{|h^{( \red{ \tau_i})} (1)|}{\red{\tau}!\, s(\red{\tau})}
        \left( 5\sqrt{ C}\,  \Hat{\lambda} e^{ \overline{\mathfrak{m}} \,
t}  \right)^{|\red{\tau}|}
    \right)^3 \\
& \leqslant \left(
    \overline{y}  \left(5 \sqrt{ C}\,  \Hat{\lambda} e^{ \overline{\mathfrak{m}} \,
t}  \right)\right)^3
=
\frac{1}{( 1- 2\cdot 5 \sqrt{  C}\,  \Hat{\lambda} e^{ \overline{\mathfrak{m}} \,
t}   )^{\frac{3}{2}}}
 \;,
\end{aligned}
\end{equation*}
where $ \overline{y}$ is the solution to the ODE $ \dot{
\overline{y}} =\overline{y}^{3} $ with \emph{positive} initial condition $ \overline{y}(0) =1 $.
The solution of this ODE is $\overline{y}(\zeta)=(1-2\zeta)^{-1/2}$ and so the associated Butcher
series converges for $ \zeta < 1/2 $, i.e.\ if \eqref{e_choice_lam} holds in the
present case.
Thus, with \eqref{e:prf_supp10}, we have
\begin{equation}\label{e:prf_supp8}
\begin{aligned}
\sum_{\substack{
\tau_1, \tau_{2}, \tau_3\in \mT_3^{N} \\ [\tau_1\, \tau_2 \, \tau_3]\notin
\mT_3^{N}}} &
    \prod_{i=1}^3
         \left\{ \frac{|h^{( \red{ \tau_i})} (1)|}{\red{\tau_i}!\, s(\red{\tau_i})}
        \left( 5 C \Hat{\lambda}^2 e^{2 \overline{\mathfrak{m}} \,
t} \right)^{|\red{\tau_i}|}
\right\} 
& \leqslant  \big(\sqrt{C}\hat{\lambda} e^{ \overline{\mathfrak{m}} \,
t}\big)^{N}
\frac{1}{( 1- 10 \sqrt{C} \Hat{\lambda} e^{ \overline{\mathfrak{m}} \,
t}  )^{\frac{3}{2}}}
\;.
\end{aligned}
\end{equation}
Hence,  combining \eqref{e:prf-sup5},
\eqref{e:prf-prop4} and
\eqref{e:prf_supp8}, we obtain
\begin{equation*}
\begin{aligned}
& \sqrt{\log \tfrac{1}{\ve}}  \cdot   \| u^{N}_{ \ve} (t,\cdot)-
u_\varepsilon(t,\cdot) \|_{L^{2} (\PP)}
\leqslant
\frac{C_{0}(T, \mathfrak{m}, \hat{\lambda})}{ { \log{ \frac{1}{\ve}
}}}
\frac{\big(\sqrt{C}\hat{\lambda} e^{ \overline{\mathfrak{m}} \,
t}\big)^{N}
}{\ve}
 \end{aligned}
\end{equation*}
with
\begin{equation}\label{e_def_C0}
\begin{aligned}
C_{0}(T, \mathfrak{m}, \hat{\lambda})
:=
\frac{
e^{4
\overline{\mathfrak{m}} \,T}}{4}
\left(\frac{\sqrt{5} \hat{\lambda}\, \Tilde{c} (T,\mathfrak{m})
}{\sqrt{
1- 10\sqrt{ C}\, \Hat{\lambda} e^{ \overline{\mathfrak{m}} \,
T}  }}  \right)^{3}
\,.
\end{aligned}
\end{equation}
This concludes the proof of
the proposition, modulo the proof of \eqref{e:sup-prop3}.
The latter bound follows simply from the triangle inequality and
Proposition~\ref{prop_single_tree} (which we can apply in view of the
constraint $ | \tau | \leqslant \lfloor \log{\tfrac{1}{\ve}} \rfloor $):
\begin{align}\label{eq_supp2_pos_max_princ}
 \|& X^{\tau}_\varepsilon(t,x)\|_{L^2(\PP)} \\
& \leqslant
 \frac{|h^{( \red{ \tau})} (1)|}{\red{\tau}!\, s(\red{\tau})}
        \left(
 \left(\frac{3\Hat{\lambda}^2}{2\pi}\right)^{|\red{\tau}|}
\left\|\<10>_{\, \varepsilon}(t,x)\right\|_{L^2(\PP)}
+
 \left(C\, \hat{\lambda}^{2} e^{2 \overline{\mathfrak{m}}\, t} \right)^{ |\red{ \tau}|}
\frac{ 2 e^{ 2|\mathfrak{m}|\, T}+
4}{ \sqrt{ \log{ \frac{1}{\ve}}}}
\frac{ \hat{\lambda} e^{\mathfrak{m}\, t}
}{\sqrt{4 (t + \ve^{2})}} \right) \nonumber
\,,
\end{align}
where we made use of the crude estimate
\begin{equation*}
\begin{aligned}
\frac{ e^{ 2|\mathfrak{m}|\, t}+|\log{(t+
\ve^{2})}| + \sqrt{\log \tfrac{1}{\ve}} }{2 \log{ \tfrac{1}{\ve}}}
\leqslant
\frac{ e^{ 2|\mathfrak{m}|\, t}}{2 \log{2}} +2+ \frac{1}{2 \sqrt{\log{2}} }
\leqslant
e^{ 2|\mathfrak{m}|\, T}+3\,,
\end{aligned}
\end{equation*}
which is a consequence of $ \ve \in (0, \tfrac{1}{T}
 \wedge \tfrac{1}{2}) $ and the uniform estimate
\begin{equation}\label{e:def_const_c}
\begin{aligned}
\sup_{0 \leqslant  t \leqslant T}
\frac{|\log (t+\varepsilon^2)|}{2\,\log \frac{1}{\varepsilon}}
&=
\sup_{ 0 \leqslant t \leqslant 1- \ve^{2}}
\frac{|\log (t+\varepsilon^2)|}{2\,\log \frac{1}{\varepsilon}}
\vee
\sup_{ 1- \ve^{2} <t< T}
\frac{\log (t+\varepsilon^2)}{2\,\log \frac{1}{\varepsilon}}\\
& \leqslant
1+
\frac{ \log{( \tfrac{1}{\ve}+1)}}{2\,\log \frac{1}{\ve}  }
\leqslant 2 \,.
\end{aligned}
\end{equation}
Now the statement follows, since
\begin{equation*}
\begin{aligned}
\left(\frac{3\Hat{\lambda}^2}{2\pi}\right)^{|\red{\tau}|}
\left\| \<10>_{\, \ve} (t, x) \right\|_{L^{2}(\PP)}
& =
\left(\frac{3\Hat{\lambda}^2}{2\pi}\right)^{|\red{\tau}|}
\left(
\hat{\lambda}^{2}_{\ve} e^{2 \mathfrak{m}\, t}
 \int_{\RR^{2}}
p_{t+ \ve^{2} }(y)^{2}
\ud y  \right)^{\frac{1}{2}} \\
 & =
\frac{1}{ \sqrt{ \log{ \tfrac{1}{\ve}}}}
\left(\frac{3\Hat{\lambda}^2}{2\pi}\right)^{|\red{\tau}|}
 \frac{
\hat{\lambda} e^{ \mathfrak{m}\, t}}{
\sqrt{4 \pi (t + \ve^{2})} } \\
& \leqslant
\frac{
 \left(C\,
\Hat{\lambda}^2 e^{2 \overline{\mathfrak{m}}\,
t}\right)^{|\red{\tau}|}}{ \sqrt{ \log{ \frac{1}{ \ve}}}}
 \frac{
\hat{\lambda} e^{ \mathfrak{m} \, t}}{2
\sqrt{t + \ve^{2}} }
\;,
\end{aligned}
\end{equation*}
where $C$ is the  constant from \eqref{e_def_CT}.
Thus together with \eqref{eq_supp2_pos_max_princ}, we obtain
\begin{equation*}
\begin{aligned}
 \| X^{\tau}_\varepsilon(t, x)\|_{L^2(\PP)}
\leqslant
\frac{\Tilde{c} (T,\mathfrak{m})}{ \sqrt{ \log{ \frac{1}{ \ve} }}}
\frac{|h^{( \red{ \tau})} (1)|}{\red{\tau}!\, s(\red{\tau})}
 \left( C\,
\Hat{\lambda}^2 e^{2 \overline{\mathfrak{m}} \, t }\right)^{|\red{\tau}|}
\frac{ \hat{\lambda} e^{ \mathfrak{m}\, t}}{2 \sqrt{t + \ve^{2}}}\,.
\end{aligned}
\end{equation*}
with $ \Tilde{c} (T,\mathfrak{m})=e^{ 2|\mathfrak{m}|\, T}+4$.
This completes the proof.
 \end{proof}

\begin{proof}[Proof of Proposition~\ref{prop_exp_close_to_but}]
For $h(y)=-y^3$, we introduce the truncated Butcher series
\begin{equation}\label{e_trunc_b_series}
\begin{aligned}
    B^{\ve}_{h} (\zeta, 1)  =
    \sum_{\tau \in \mT_{\leqslant 3}^{N_{\ve}}} \frac{h^{( \tau)} (1)}{\tau!\, s(\tau)} \zeta^{|\tau|}
=
\sum_{\substack{\tau \in \mT_{ 3}^{N_{ \ve}}}} \frac{h^{( \red{ \tau})} (1)}{\red{\tau}!\, s(\red{\tau})} \zeta^{|\red{\tau}|}
     \,,
\end{aligned}
\end{equation}
where the second equality is a consequence of Lemma~\ref{lem_trimming_bij}.
Therefore, Proposition~\ref{prop_exp_close_to_but} will follow if we can show that the
following two limits hold true with $ \zeta_{ \hat{\lambda}} = \tfrac{3 \hat{\lambda}^{2}}{2 \pi} $:
\begin{align}
&\lim_{ \ve \to 0}  \,\, \Big\| \sqrt{\log{\tfrac{1}{\ve}}} \, u_\ve^{N_{ \ve}}
(t, x) -
\hat{\lambda} B^{\ve}_{h} (\zeta_{ \hat{\lambda}}, 1) \, e^{\mathfrak{m}\, t}
P_{t+ \ve^{2}} \eta (x) \,\Big\|_{L^{2} (\PP)} = 0 \;,
\label{e:prf-prop-1} \\
&\lim_{\ve \to 0}  \,\, \Big \| \hat{\lambda} B^{\ve}_{h} (\zeta_{
\hat{\lambda}}, 1)\, e^{\mathfrak{m}\, t} P_{t+ \ve^{2}} \eta (x) -  \hat{\lambda}\, \sigma_{
\hat{\lambda}} P^{(\mathfrak{m})}_{t} \eta (x)  \,\Big\|_{L^{2} (\PP)} = 0
\label{e:prf-prop2}\;.
\end{align}
The limit \eqref{e:prf-prop-1} follows from Proposition~\ref{prop_single_tree}, provided \eqref{e:condition} holds.
Indeed, via the named proposition, we can bound for $ \ve \in (0,\tfrac{1}{T}
\wedge \tfrac{1}{2})$
\begin{equation}\label{e_supp1clBut}
\begin{aligned}
&\Big \| \sqrt{\log\tfrac{1}{\ve}} \, u_\ve^{ N_\ve}(t, x)
 - \hat{\lambda} B^{\ve}_{h} (\zeta_{ \hat{\lambda}}, 1) \, e^{\mathfrak{m}\, t} P_{t+ \ve^{2}} \eta (x) \Big \|_{L^{2} (\PP)} \\
& \leqslant    \sum_{ \substack{ \tau \in \mT_{3}^{N_{\ve}}}}
 \left\|\sqrt{ \log \tfrac{1}{\ve}}\cdot
X^{\tau}_{\ve}(t, x) - \hat{\lambda}
\frac{h^{( \red{ \tau})} (1)}{\red{\tau}!\, s(\red{\tau})}
 \zeta_{\hat\lambda}^{|\red{\tau}|} e^{\mathfrak{m}\, t} P_{t + \ve^{2}} \eta (x) \right\|_{L^{2} (\PP)} \\
& \leqslant \sum_{  \tau \in
\mT_{\leqslant 3}^{N_{\ve}} }\frac{|h^{(  \tau)} (1)|}{ s(\tau)\  \tau!}
    \left(C \Hat{\lambda}^2 e^{2 \overline{\mathfrak{m}}\, t} \right)^{|\tau|}
\frac{ e^{2| \mathfrak{m}|\, t}+|\log{(t+
\ve^{2})}| + \sqrt{\log \tfrac{1}{\ve}} }{2 \log{ \tfrac{1}{\ve}}}
 \frac{ \hat{\lambda} e^{ \mathfrak{m}\, t} }{ \sqrt{4(t + \ve^{2})}}
\,,
\end{aligned}
\end{equation}
where we  used \eqref{e_trunc_b_series}.
Now, \eqref{e:prf-prop-1} will follow, if we can show that the series on the
right-hand side is summable, that is if
\begin{equation}\label{particularB}
\begin{aligned}
\sum_{ \tau \in
\mT_{\leqslant 3} }\frac{|h^{( \tau)} (1)|}{ s(\tau)\  \tau!}
    \left(C \hat{\lambda}^2 e^{2 \overline{\mathfrak{m}}\, t}\right)^{|\tau|} < \infty
\;.
\end{aligned}
\end{equation}
This is the Butcher series associated to the ODE $ \dot{
\overline{y}} =\overline{y}^{3} $ with initial condition $ \overline{y}(0) =1
$, which converges as long as \eqref{e:condition} holds.
See also the discussion in the proof of
Proposition~\ref{prop_max_principal}.
Hence, \eqref{e_supp1clBut}, and thus \eqref{e:prf-prop-1}, vanish for
arbitrary fixed $t \in (0,T]$, because the second-to-last ratio on the right-hand side in \eqref{e_supp1clBut} vanishes in the
limit $ \ve \to 0 $.

To complete the proof of the proposition we must now check \eqref{e:prf-prop2}.
Here we observe that
\begin{equation*}
\begin{aligned}
&\Big \| \, \hat{\lambda} B^{\ve}_{h} (\zeta_{ \hat{\lambda}}, 1) \,
e^{ \mathfrak{m} \, t} P_{t+ \ve^{2}} \eta (x) -  \hat{\lambda}\,
\sigma_{\hat{\lambda}}\,  P_{t}^{( \mathfrak{m})}\eta (x)  \, \Big\|_{L^{2}
(\PP)}\\
&\leqslant
 \hat{\lambda}\,
 | B^{\ve}_{h} (\zeta_{ \hat{\lambda}}, 1)  - \sigma_{
\hat{\lambda}}|\cdot
 \big \| e^{ \mathfrak{m} \, t} P_{t+ \ve^{2}} \eta (x) \,\big \|_{L^{2} (\PP)}\\
&\qquad +  \hat{\lambda}\, \sigma_{ \hat{\lambda}}
\, \big \| e^{ \mathfrak{m} \, t} P_{t+ \ve^{2}} \eta (x) -
P_{t}^{(\mathfrak{m})} \eta (x) \big \|_{L^{2} (\PP)}\,.
\end{aligned}
\end{equation*}
The second term is converging to $0$ by the continuity properties of the
heat semigroup. Instead, for the first term we observe that \( B^{\ve}_{h}
(\zeta, 1) \) is an approximation to the Butcher series associated to
the solution $ y(\zeta) $ of the ODE $ \dot{ y} = - y^{3}, y(0)=1 $, which is
given by $y(\zeta) = (1 + 2 \zeta)^{-1/2}$. This solution is analytic for
$ | \zeta | < 1/2 $ and the associated Butcher series converges (see the discussion in Section~\ref{sec:butcher}).
Thus, recalling the form of $\sigma_{\hat\lambda}$ \eqref{sigmalambda},
we have that  $\lim_{ \ve \to 0}|  B^{\ve}_{h} (\zeta_{ \hat{\lambda}}, 1)  - \sigma_{ \hat{\lambda}}| =0 $, as long as
$\frac{3 \hat{\lambda}_{\mathrm{fin}}^{2}}{2 \pi} < \frac{1}{2} \;,$
which is implied by \eqref{e:condition}.
This concludes the proof.
 \end{proof}

\section{Contracted trees, Wiener chaoses and their structure}\label{sec:exmpl}

In Section~\ref{sec_outline}, and in particular in 
Subsection~\ref{sec:outline}, we outlined the
structure underlying the main estimate contained in Proposition~\ref{prop_single_tree}.
In this section we will introduce the notion of contraction and
present estimates on the integration kernels associated to the Wild expansion
terms $X_{ \ve}^{ \tau}$, which will allow us to analyze them rigorously in
Section~\ref{sec:estimates-single-tree}.

\subsection{Wiener chaos decomposition, contractions and cycles}

The multiple stochastic integrals $X^{\tau}$ appearing in the Wild expansion
\eqref{e:ac2} lie in
the $\ell(\tau)$-th {\it inhomogeneous} Wiener chaos.
Elements in a finite inhomogeneous Wiener chaos can be decomposed into
terms belonging to distinct {\it homogeneous} Wiener chaoses.
We refer to \cite[Chapter 2]{Janson1997} and  \cite[Chapter 1]{Nualart} for a more detailed discussion about Wiener spaces and
their decomposition to homogeneous components.

Our asymptotic analysis builds on a precise understanding of the decomposition
of the components of the Wild expansion into
its homogeneous chaos terms. The goal of this detailed study will be to show
 that only terms in the first chaos (and not all of them) contribute to the
Gaussian limit in Theorem \ref{thm:main}.

 In our setting, homogeneous components of the Wild expansion will be represented by stochastic integrals
 indexed by trees with additional {\it contraction in pairs} between elements of
 a subset of their leaves. A \emph{contraction} of a given tree is a
paring among the elements of an arbitrary subset of the leaves of the tree.
Unlike the stochastic integrals indexed by the initial tree,
 the stochastic integral indexed by a contracted tree lies in a homogeneous
chaos, whose order is given by the number of
uncontracted leaves. One can then recover the integral associated to the original tree
 by summing over integrals indexed by the same tree with all possible contractions.
Let us now be more precise and start with the definition of a contraction.
\begin{definition}\label{def:contractions}
For any $ \tau \in \mT $\ we define a {\bf contraction} to be a subset 
$\kappa \subset \binom{\mL(\tau)}{2}$, where $\binom{\mL(\tau)}{2}$ denotes the
set of all unordered pairs of leaves of the tree $\tau$, such that
every $v\in \mL(\tau)$ lies in at most one element of  $\kappa$.
We define the corresponding {\it set of contractions} by
\begin{align}\label{def:cont}
    \mathcal{K}(\tau):=\left\{ \kappa \subset \binom{\mL(\tau)}{2}\,:\, \kappa
\text{ is a contraction of }\tau \right\}\,.
\end{align}
Furthermore, we denote a tree $\tau=  (\mV, \mE)$ that is being contracted according to a
contraction $\kappa$ by $\tau_\kappa$:
\begin{equation*}
\begin{aligned}
\tau_{\kappa} := (\mV, \mE \cup \kappa) \;,
\end{aligned}
\end{equation*}
and call this a $\kappa$-{\it contracted tree} or simply a {\it contracted tree}.
We will also denote by $\mL(\tau_\kappa)$  the set of
leaves of $\tau_\kappa$, namely, the set of leaves of $\tau$ which are not included in the contraction $\kappa$.
\end{definition}
If all leaves of $\tau$ are contracted via $\kappa$ we will call $\kappa$ a
{\it complete contraction} and $\tau_\kappa$ a {\it completely
contracted} tree. If this is not the case, we will often talk of {\it partial contraction} and a {\it partially contracted} tree.

Graphically, a contracted tree $\tau_\kappa$ is represented by the original graph of $\tau$ augmented
with edges connecting the pairs of vertices in $\kappa$. We will colour the
additional edges arising from $\kappa$ in red.
For example, the possible contractions of the tree
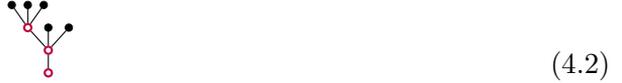
\begin{align}\label{e:emptylist}
 \<1130_noise>
\end{align}
   are (up to symmetries)
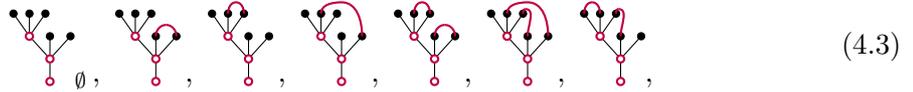
\begin{equation}\label{e:list}
\begin{aligned}
    \<1130_noise>_{\emptyset}
    \,,\
    \<1130_3i>
    \,,\
    \<1130_3ii>
    \,,\
    \<1130_3iii>
    \,,\
    \<1130_1i>
    \,,\
    \<1130_1ii>
    \,,\
    \<1130_1iii> ,
\end{aligned}
\end{equation}
where we denoted the tree without any contraction with a subscript $\emptyset$
to emphasize the empty contraction.\footnote{In the $\emptyset$-contracted
tree $\tau_\emptyset$, the iterated stochastic integral will correspond to a {\it homogeneous chaos},
 and the purpose of the $\emptyset$ subscript in its graphical depiction is to
distinguish it from the
 graph representation \eqref{e:emptylist}, which corresponds to an element of the {\it inhomogeneous chaos}.
 On the other hand, there is no such danger of confusion in the rest of the graphical depictions of $\tau_\kappa$ with $\kappa\neq \emptyset$
 and so we do not use any similar subscript in order not to overload notation.}
Intuitively, when seen as a stochastic integral,
the uncontracted vertices in $\tau_\kappa$ will have all assigned space variables
being distinct, while the edge with space variables $(y_{u_1},y_{u_2})$
connecting a
pair of $(u_1,u_2) \in \kappa$ will be assigned a weight
$p_{2\ve^2}(y_{u_2}-y_{u_1})$.
More precisely,  
to any contracted tree $\tau_\kappa$ we associate a stochastic integral lying in
a {\it homogeneous} Wiener chaos through the following
definition:
\begin{align}
   &{ \tau}_{\kappa, \ve} (t, x)
   : =  \int_{ D_{t}^{\mV(\tau)\setminus \mf{o}}  }
K_{\tau_{\kappa}, \ve}^{t, x} (s_{\mV (\tau)}, y_{\mV(\tau)} )
\, \ud y_{\mV(\tau)\setminus (\mf{o} \cup
\mL(\tau_{\kappa})) }\ud s_{\mV ( \tau) \setminus \mf{o}} \ \eta_\ve( \ud y_{\mL(\tau_\kappa)}) \label{e:tk}
\end{align}
with $D_{t}= [0,t]\times \RR^{2}$ and
\begin{align}
  K_{\tau_{\kappa}, \ve}^{t, x} (s_{\mV (\tau)}, y_{\mV(\tau)})
 & :=
\prod_{u \in \mV( \tau)\setminus \mf{o}}
p_{s_{\mathfrak{p}(u)} - s_{u}}^{( \mathfrak{m })}(y_{\mathfrak{p}(u)} - y_{u})
  \Bigg\{ \prod_{v \in\mL(\tau_\kappa)} \delta_{0}(s_v)\Bigg\}  \label{e:tk2}\\
&  \hskip 3cm \times  \prod_{(u_1,u_2) \in \kappa }
\lambda_{\ve}^{2} \, \delta_{0}(s_{u_1}) \, \delta_{0}(s_{u_2}) \,
p_{2\ve^2}(y_{u_2} - y_{u_1})   \, ,  \notag
\end{align}
with $\mf{p}(u)$ denoting the parent of $u$ and $ ( s_{\mf{o}}, y_{\mf{o}})=
(t,x)$.
To lighten notation, we will often drop the index $t,x$ that indicates the time-space coordinates of the root,
if the explicit indication is not necessary.
Note that \eqref{e:tk} is to be interpreted as an Wiener--It\^o integral. We
stress the difference in notation once and for all here: For the
Wiener--It\^o integral we write $\eta_\ve( \ud y_{\mL(\tau)})$, whereas for the
Wiener--Stratonovich integral we use $ \prod_{v \in \mL ( \tau)} \eta(\ud
y_{v})$, cf. \eqref{e:tau-int}.

With this definition, $\tau_{\kappa,\ve}$
 lies in the homogeneous Wiener chaos of order
$\ell(\tau_k)=|\mL(\tau_\kappa)|$.
Given the decomposition of an element in a
Wiener chaos into its homogeneous
components, see \cite[Remark 7.38]{Janson1997}, we have that for any $\tau=
[ \tau_{1} \ \ldots \ \tau_{n}]$, the associated stochastic integral $\tau_\ve$
 can be decomposed as follows
\begin{equation}\label{e:homchaos}
    \tau_\ve\,
:=
[\tau_{1}]_{\ve}\cdots [ \tau_{n}]_{ \ve}
 = \sum_{\kappa \in \mathcal{K}(\tau)} \tau_{\kappa, \ve}\,,
\end{equation}
with $ [ \tau_{i}]_{ \ve}$ defined in \eqref{e:tau-int}.
In the example of the trees in \eqref{e:emptylist} and \eqref{e:list},
we have that the decomposition of the inhomogeneous element represented by the tree \eqref{e:emptylist} to its
homogeneous components is given by
\begin{align*}
    \<1130_noise>_{\! \ve}
    =
    \<1130_noise>_{\! \ve, \emptyset}
    +
    \<1130_3i>_{\! \ve }
    +
    3
    \<1130_3ii>_{\! \ve}
    +
    6
    \<1130_3iii>_{\! \ve}
    +
    3
    \<1130_1i>_{\! \ve}
    +
    6
    \<1130_1ii>_{\! \ve}
    +
    6
    \<1130_1iii>_{\! \ve}\,,
\end{align*}
where the right-hand side corresponds to the homogeneous stochastic integrals
indexed by the contracted trees in \eqref{e:list}.
Here we have taken into account multiplicities of homogeneous components due
to equivalent contractions. For example, the contractions
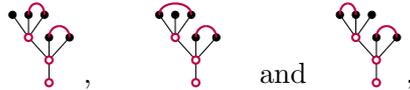
\begin{equation*}
\begin{aligned}
\<1130_1i>\ , \qquad
\<1130_1ia>
\quad \text{ and}
\quad 
\<1130_1ib>\,,
\end{aligned}
\end{equation*}
are all different. However, they correspond to the same stochastic integrals.
Lastly, let us mention that for a planted tree $[ \tau]$ we use both notations $[
\tau]_{\kappa} $ and $ [ \tau_{ \kappa}]$ for a contracted version of that tree.

\vskip 2mm
\noindent
{\bf Contractions between trees and $ L^{2}(\PP)$ estimates.}
We now want to extend the notion of contraction from within a single tree to a pair of trees.
This will be necessary in order to encode second moments of stochastic
integrals of the form \eqref{e:tk}.
The Gaussianity and correlation structure of the white noise, imply via Wick's theorem, that the second moment
can be expressed as the sum over all possible {\it pairwise contractions} over
the (uncontracted) leaves (or precisely over the noise
variables that lie on the leaves) of two copies of the tree, connected to the same root with
time-space variables $(t,x)$. In other words, we look at the stochastic integral $[\tau,\tau]_\ve(t,x)$
corresponding to the tree $[\tau,\tau]$ with root variable $(t,x)$. Let us look at the example of computing the second moment of the stochastic integral
$[\tau]_\ve(t,x)=\scalebox{0.6}{\<30_3>}_{\!\ve}(t,x)$, evaluated at a time-space point $(t,x)$.
Its second moment will be represented by 
\begin{align}\label{intercontractions:pic}
    \EE \big[ [\tau]_\ve(t,x)^{2}\big] & = \EE \big[ [\tau, \tau]_\ve(t, x)
\big]
    = \EE \big[ \big[\<3>\ \<3>\big]_\ve ({t, x}) \big] \notag \\
& = 6 \underset{(t, x)}{
    \vcenter{\hbox{
        \begin{tikzpicture}[scale=0.3]
        \draw  (0,0)   -- (-.7,1) node[dot] {}  ;
        \draw  (0,0)   -- (0,1) node[dot] {}  ;
        \draw (0,0) -- (.7,1) node[dot] {};
        \draw  (3,0)   -- (2.3,1) node[dot] {}  ;
        \draw  (3,0)   -- (3,1) node[dot] {}  ;
        \draw (3,0) -- (3.7,1) node[dot] {};
        \draw[purple, thick] (-.7,1) to[out=90,in=180] (1.5,2.5);
        \draw[purple, thick] (1.5,2.5) to[out=0,in=90] (3.7,1);
        \draw[purple, thick] (0,1) to[out=90,in=180] (1.5,2);
        \draw[purple, thick] (1.5,2) to[out=0,in=90] (3,1);
        \draw[purple, thick] (.7,1) to[out=90,in=180] (1.5,1.5);
        \draw[purple, thick] (1.5,1.5) to[out=0,in=90] (2.3,1);
	 \draw (0,0)  node[idot] {}  -- (1.5,-1)   ;
        \draw (3,0)  node[idot] {}   -- (1.5,-1) node[idot] {};
    \end{tikzpicture}  }}}
    +
    9 \underset{(t, x)}{
    \vcenter{\hbox{
        \begin{tikzpicture}[scale=0.3]
        \draw  (0,0)   -- (-.7,1) node[dot] {}  ;
        \draw  (0,0)   -- (0,1) node[dot] {}  ;
        \draw (0,0) -- (.7,1) node[dot] {};
        \draw  (3,0)   -- (2.3,1) node[dot] {}  ;
        \draw  (3,0)   -- (3,1) node[dot] {}  ;
        \draw (3,0) -- (3.7,1) node[dot] {};
        \draw[purple, thick] (.7,1) to[out=90,in=180] (1.5,1.5);
        \draw[purple, thick] (1.5,1.5) to[out=0,in=90] (2.3,1);
        \draw[purple, thick] (-.7,1) to[out=90,in=180] (-.35,1.5);
        \draw[purple, thick] (-.35,1.5) to[out=0,in=90] (0,1);
        \draw[purple, thick] (3,1) to[out=90,in=180] (3.35,1.5);
        \draw[purple, thick] (3.35,1.5) to[out=0,in=90] (3.7,1);
 	\draw (0,0)  node[idot] {}  -- (1.5,-1)   ;
        \draw (3,0)  node[idot] {}   -- (1.5,-1) node[idot] {};
    \end{tikzpicture}  }}} \,.
\end{align}
In other words, the computation of the second moment of a stochastic integral $[\tau]_\ve$ gives rise to
a {\it completely} contracted tree $[\tau,\tau]_\kappa$ in accordance with the definition of \eqref{def:cont}.
We note that, if we first decompose  $[\tau]_\ve(t,x)$
into its homogeneous Wiener chaos components, then an alternative computation
would yield
\begin{equation*}
\begin{aligned}
\EE [ [\tau](t,x)^{2}] =
    \EE\bigg| \; \underset{(t, x)}{\<30_3>_{\!\!\emptyset}} \, \bigg|^2
    + \EE\bigg| \, 3 \underset{(t, x)}{\<30_1>} \, \bigg|^2 \;,
\end{aligned}
\end{equation*}
where we used the orthogonality between different homogeneous chaos
components.

Let us  introduce a notation that will allow us to encode contractions between
trees that are glued together,
in a way that distinguishes them from the contractions of Definition~\ref{def:contractions}.
This will be useful to encode covariances between $[\tau]_\ve{(t, x)}$ and $ [\tau']_\ve{(t, x)}$.
\begin{definition}\label{def:doublecont}
For two rooted trees $\tau, \tau' \in \mT $ define the set of {\bf pairings} among the union of leaves as
\begin{align*}
    \mY(\tau, \tau')
    :=
    \big\{\gamma \in \mathcal{K}([\tau,\tau'])\,:\,
\text{ $ \gamma$ is a complete contraction }
 \, \big\}\,.
\end{align*}
We also define the subsets of pairings which complete a given pair of contractions
$ (\kappa, \kappa^{\prime} ) \in \mK (\tau) \times \mK (\tau^{\prime}) $ by
\begin{align*}
    \mY(\tau_\kappa, \tau'_{\kappa'})
    :=
    \big\{\gamma \in \mY(\tau,\tau')\,:\, \kappa \cup \kappa' \subset \gamma \text{
and all pairs in $\gamma \setminus ( \kappa \cup \kappa')$ connect $\tau$ to $\tau'$} \big\}\,.
\end{align*}
We will write $[\tau,\tau']_\gamma$ to denote the tree $[\tau,\tau']$ where all leaves
are contracted according to $\gamma\in \mY(\tau,\tau')$.
\end{definition}
The pictorial representations in \eqref{intercontractions:pic} show all possible elements (up to symmetries) of
$\mY(\tau,\tau)$ for that example.
Furthermore, note that for any two contracted trees $ \tau_{\kappa} ,
\tau_{\kappa^{\prime}}^{\prime}  $, and  $ \gamma \in \mY(\tau_\kappa,
\tau'_{\kappa'})$, the pairing $ [\tau_{\kappa} ,
\tau^{\prime}_{\kappa^{\prime}}]_{\gamma}  $  gives rise to a {\it
completely contracted} tree and, therefore,
$\mY(\tau_\kappa, \tau'_{\kappa'})=\emptyset$, if the number of
uncontracted leaves in $\tau_\kappa$ and $\tau'_{\kappa'}$ differ. This agrees with the fact that homogeneous chaoses are orthogonal with
respect to one another.
We can now express covariances between contracted trees as follows:
\begin{equation}\label{e_cov_contr_tree}
\begin{aligned}
 \EE\Big[\,[\tau]_{\kappa, \ve} \, [\tau']_{\kappa',\ve}\,\Big]
=
\sum_{\gamma \in \mY(\tau_{\kappa},\tau'_{\kappa'})}
    [\, \tau,\tau']_{\gamma, \ve}\,.
\end{aligned}
\end{equation}
Moreover, it is clear that
$\mY (\tau_{\kappa}, \tau'_{ \kappa'} )$ allows to partition $\mY ( \tau, \tau')$  as
\begin{equation}\label{e:partitionPairings}
\begin{aligned}
\mY ( \tau, \tau') = \bigsqcup_{ (\kappa, \kappa') \in \mK ( \tau) \times
\mK ( \tau')} \mY (\tau_{\kappa}, \tau'_{ \kappa'} )\,,
\end{aligned}
\end{equation}
where $\bigsqcup$ denotes a disjoint union. This is clear since, if we want to
find all pairwise contractions of $[\tau,\tau']$,
we can first identify the contractions that are internal to each $\tau,\tau'$ and then identify the contractions that connect the
leaves of one tree to those of the other.
This partitioning then allows us to express covariances in terms of
\begin{equation*}
\begin{aligned}
 \EE\Big[\,[\tau]_\ve \, [\tau']_\ve \,\Big]
    &=
    \sum_{\kappa \in \mathcal{K}(\tau)}
    \sum_{\kappa' \in \mathcal{K}(\tau')}
    \EE\Big[\,[\tau]_{\kappa, \ve} \, [\tau']_{\kappa',\ve}\,\Big] \\
    & =
    \sum_{\kappa \in \mathcal{K}(\tau)} \,
    \sum_{\kappa' \in \mathcal{K}(\tau')} \,
    \sum_{\gamma \in \mY(\tau_{\kappa},\tau'_{\kappa'})}
    [\, \tau,\tau']_{\gamma, \ve}\\
& =
    \sum_{\gamma\in \mY(\tau,\tau')} [ \,\tau,\tau']_{\gamma, \ve}\,,
\end{aligned}
\end{equation*}
where we used \eqref{e_cov_contr_tree} in the second step, and
\eqref{e:partitionPairings} in the last.
Finally, the partitioning \eqref{e:partitionPairings} allows us to recover
the internal contractions associated to a given pairing. This motivates the following
definition.

\begin{definition}\label{def:PairingTocontraction}
For any $ \tau, \tau' \in \mT_{3}$,
let
$\mathfrak{s}_{ [ \tau, \tau']} : \mY( \tau, \tau') \to \mK( \tau ) \times \mK(
\tau')$ be the map that for any $ \gamma \in \mY ( \tau, \tau')$
identifies the unique pair $\mathfrak{s}_{ [\tau, \tau']}(\gamma) := (\kappa_1(\gamma),
\kappa_{2}( \gamma))$ such that
\begin{equation*}
\begin{aligned}
\gamma \in \mY ( \tau_{ \kappa_{1}( \gamma)}, \tau'_{ \kappa_{2} (\gamma )})\,.
\end{aligned}
\end{equation*}
In other words, the map $\mathfrak{s}$ identifies the subset of edges in $\gamma$ that only connect within
$\tau$ and $ \tau'$, respectively.
\end{definition}

\subsection{$1$-cycles and their removal}

Let us now introduce the notion of a {\it $1$-cycle}. Suppose that a contracted
tree $\tau_\kappa$
contains a component of the form
\begin{equation}\label{example1cycle}
\begin{aligned}
\tau_{\kappa} = \<contain1cycle> \;,
\end{aligned}
\end{equation}
namely where we observe a cycle consisting of an inner vertex $(s_1,y_1)$ connected to two leaves
that are themselves connected to one another by a red edge (part of $\kappa$).
We call such a cycle a $1$-{\it cycle}.
Let us remark that in the above picture, the point $(s_2,y_2)$ denotes the coordinates of
the basis of the sub-tree $\tau_1$ and $(s_0,y_0)$ denotes the coordinates of the parent
of the inner vertex with coordinates $(s_1,y_1)$. A more formal definition is
the following.
\begin{definition}\label{def:onecycle}
Given a tree $\tau$ and a contraction $\kappa \in \mK(\tau)$, we call a $1$-{\it cycle}
a connected component of $\tau_\kappa$ which consists of two leaves,
which are connected by an element of $\kappa$, and the inner vertex, which is the parent of these
leaves, as well as the three edges that connect these three vertices. We call the inner vertex of the
cycle the basis of the $1$-cycle.
\end{definition}
Given a contracted tree $\tau_\kappa$ with a $1$-cycle $\mC$,
we write $ (\tau \setminus \mC)_{ \widetilde{\kappa}} $ for the contracted
tree that is obtained by ``removing'' the cycle $ \mC $ from $\tau_\kappa$.
That is, the graph that remains after removing all
edges and nodes that belong to $ \mC $ and replacing the remaining two edges
which used to connect to the basis of the $ 1 $-cycle by a new single edge, which connects the only
remaining descendant of the basis we removed to the parent of this basis.
The contraction $ \widetilde{\kappa} $ is the one  induced naturally on $ \tau
\setminus \mC $ by $ \kappa $ after the removal of the element that connects the leaves of $\mC$.
The removal is simply described by the following picture:
\begin{equation}\label{remove1cycle}
\begin{aligned}
\tau_{\kappa} = \<contain1cycle> \,\,\,\,
\longmapsto \,\,\,\,
(\tau\setminus \mC)_{\tilde \kappa} = \<containno1cycle> \;.
\end{aligned}
\end{equation}
Observe that if $ \tau \in\mT_{3} $, then also $ \tau \setminus \mC \in \mT_{3} $.
An important lemma is the following, which records the effect of the $1$-cycle on the associated
stochastic integrals.
\begin{lemma}\label{lem:cncl1cycl}
Let $ \tau \in \mT $ be of the form $ \tau= [ \tau_{1} \cdots \tau_{n}]$, $
\tau_{i} \in \mT_{3}$, and $ \kappa \in \mK ( \tau) $. Further, let $ \mC $ be
a 1-cycle in the contracted tree $ \tau_{\kappa} $ with the
coordinates  of its root being $(t,x) \in (0, \infty) \times \RR^{2}$. Denote by
$(s_{1}, y_{1}) $ the coordinates of the basis of $\mC$, by
$(s_{0}, y_{0}) $ the coordinates of the parent of $(s_{1}, y_{1}) $
and by $(s_{2}, y_{2}) $ the  coordinates of the only descendant of $(s_{1}, y_{1})$ that does not belong to $\mC $.
Denote also by $ z_{3} = (s_{3}, y_{3}) $ and $ z_{4} = (s_{4} , y_{4}) $ the  coordinates
of the leaves of the 1-cycle $ \mC $ (where we recall that the time coordinates $s_3$ and $s_4$
of the leaves will coincide with $0$).  Then
\begin{equation}\label{integral1removal}
\begin{aligned}
\int_{D_{t}^{2}} \int_{\RR^{2}} & K^{t,x}_{\tau_{\kappa}, \ve} (s_{\mV}, y_{\mV}) \ud y_{1} \ud z_{3} \ud
z_{4}=  \<cycle1>_\ve (s_{1})
\mathds{1}_{\{s_{2} \leqslant  s_{1} \leqslant s_{0} \}}
K^{t,x}_{ \widetilde{\tau}_{ \widetilde{\kappa}}, \ve} (s_{\mV \setminus \mC }, y_{\mV \setminus \mC}) \;,
\end{aligned}
\end{equation}
where the kernel $ K $ is defined in \eqref{e:tk2}, $ D_{t} $ is defined in
\eqref{e:Dt} and
\begin{align}\label{def:1cyclekernel}
\<cycle1>_\varepsilon(s_1) := \lambda_{\ve}^{2} e^{2 \mathfrak{m} \,
s_{1}} \, p_{2(s_1+\varepsilon^2)}(0)\;,
\end{align}
$ \widetilde{\tau} := \tau \setminus \mC $ and $ \widetilde{\kappa} $ the
contraction induced on $\tilde\tau$ by the removal of $\mC$ from $\tau_\kappa$.
\end{lemma}
\begin{proof}
We start by performing the integration over the spatial coordinates $
y_{3}, y_{4} $ of the part of the kernel $K_{\tau_{\kappa, \ve}}^{t,x} (s_{\mV}, y_{\mV})$
that depends on the variables $z_3$ and $z_4$. This corresponds to the
following integral (recall form \eqref{e:tk2} that the kernel
$K_{\tau_{\kappa, \ve}}^{t,x} (s_{\mV}, y_{\mV} )$
 contains factors $\delta_0(s_3) \,\delta_0(s_4)$):
\begin{equation}\label{e:1cycle}
\begin{aligned}
\lambda_{\ve}^{2} \int_{(\RR^{2})^{2}} p_{s_1}^{( \mathfrak{m })}(y_3 - y_1)
p_{s_1}^{( \mathfrak{m })} (y_4 - y_1) p_{2\ve^{2}}(y_3- y_4)  \ud y_3 \ud y_4
& = \lambda_{\ve}^{2} e^{2 \mathfrak{m}\, s_{1}}\, p_{2(s_1+\varepsilon^2)}(0)
\\
& =  \<cycle1>_\varepsilon(s_1)  \;.
\end{aligned}
\end{equation}
Next, we integrate the remaining part of the kernel  over $y_1$. This reduces to
the Chapman--Kolmogorov identity (refer also to the pictures in
\eqref{example1cycle} and \eqref{remove1cycle} for guidance):
\begin{align*}
\int_{\RR^2} p_{s_1-s_2}^{( \mathfrak{m })}(y_2 - y_1) p_{s_0-s_1}^{( \mathfrak{m })} (y_1 - y_0) \,\ud y_1
= p_{s_0-s_2}^{( \mathfrak{m })}(y_2 - y_0) \;.
\end{align*}
Combining the results of the  two integrations above with the remaining components
of the kernel  $K_{\tau_{\kappa, \ve}}^{t,x} (s_{\mV}, y_{\mV})$, yields the expression on the right-hand side of
\eqref{integral1removal}.
\end{proof}
As it turns out, 1-cycles play an important role in our analysis. We will
see that the contracted trees in the Wild expansion that
contribute to the limiting fluctuations, are exactly those
whose contraction consists of  only  1-cycles (which may also emerge in an
iterative way, see the second example below).
To get an idea of this phenomenon, let us look at the following examples.
\vskip 2mm
\begin{example}\label{example:loli}
{\rm
Consider the contracted tree $\scalebox{0.8}{\<30_1>}_{\!\! \ve}$.
Using Lemma \ref{lem:cncl1cycl} (or in this case even a by-hand computation) we
find that
\begin{align*}
   \scalebox{0.9}{ \<30_1>}_{\!\!\!\ve}(t,x) &=
    \int_{[0,t]}  \int_{\RR^{2}}  p_{t-s}^{( \mathfrak{m })}(y-x)\,  \<10>_{\, \ve}(s,y)\,
    \<cycle1>_\varepsilon(s) \ud y\ud s \;,
 \end{align*}
where inside the integral $(s,y)$ are the time-space coordinates associated to the basis of the trident.
We can next compute the spatial integral via Chapman--Kolmogorov as
\begin{align*}
    \int_{\RR^{2}}
    p_{t-s}^{( \mathfrak{m })}(y-x) \, \<10>_{\, \ve}(s,y) \ud y
    &=  \int_{(\RR^{2})^{2}} p_{t-s}^{( \mathfrak{m })}(y-x)\,p_{s}^{( \mathfrak{m })}(z-y) \, \eta_\varepsilon(z) \ud y \ud z \\
    &= \int_{\RR^{2}} p_{t}^{( \mathfrak{m })}(z-x)  \, \eta_\varepsilon(z) \ud z
    =\<10>_{\, \ve}(t,x)\,.
\end{align*}
Therefore, we obtain
 \begin{align*}
    \scalebox{0.9}{\<30_1>}_{\!\!\ve}(t,x)
    &= \<10>_{\, \ve}(t,x) \int_0^t \<cycle1>_\varepsilon(s) \ud s
    = \<10>_{\, \ve}(t,x)\,  \lambda_\varepsilon^2 \int_0^t
\frac{e^{2 \mathfrak{m} \, s}}{4\pi(s+\varepsilon^2)} \ud s\,,
  \end{align*}
and hence, by Lemma~\ref{l_app_expInt} and definition of $ \lambda_{ \ve}$, for
$ t \in (0, \infty)$ we find that
\begin{equation*}
\begin{aligned}
\scalebox{0.9}{\<30_1>}_{\!\!\ve}(t,x)
=
\frac{ \hat{\lambda}^{2}}{2 \pi}
 \<10>_{\, \ve}(t,x) \cdot
( 1+ o(1))
 \,,
\end{aligned}
\end{equation*}
where the $ o(1) $ is with respect to $ \ve \to 0 $.
}
\end{example}
\begin{example}
{\rm
This example demonstrates the iterative appearance of 1-cycles, after successive extractions, and their overall contribution.
Consider the contracted tree
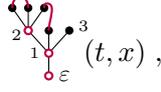
\begin{align*}
\<lab1130_1iii>\!\!\!(t,x) \;,
\end{align*}
 where we have tagged some vertices for reference in the following integrals. In particular,
 the coordinates of vertex $i$ will be $(s_i,y_i)$.
First, extracting the 1-cycle with basis $2$, we have  that, using
Lemma~\ref{lem:cncl1cycl},
\begin{align*}
\<lab1130_1iii>\!\!\! (t,x) & =
 \underset{\substack{[0,t]^{3}\times (\RR^2)^2\\ \{ s_2,s_3 \leq s_1 \leq t\} } }{ \int\int}
\<cycle1>_{\ve}(s_{2}) \,\,
  K_{\!\!\!\<s3m>} \!\! (y_1,y_3; s_1,s_3 ) \,\eta_{\ve}(\ud y_{3}) \ud y_1 \ud
s_{2} \ud s_1 \ud s_3 \;.
\end{align*}
Now, applying once more Lemma~\ref{lem:cncl1cycl} on the kernel, or just  the previous
example, the above integral equals
\begin{align*}
 \<10>_{\, \ve} (t, x)\int_{0}^{t} \int_{0}^{s_{1}}  \<cycle1>_{\ve} (s_{2})
\<cycle1>_{\ve} (s_{1})\ud s_{2} \ud s_{1}
= \frac{1}{2}\Big(\frac{\hat \lambda^2}{2\pi}\Big)^2 \<10>_{\, \ve} (t, x)
\cdot \big( 1+o(1) \big)\;.
\end{align*}
Thus, the contribution of this diagram is of the same order as in the previous example (albeit with a different constant)
and will also  contribute to the limiting Gaussian fluctuations.

Following the same steps as above, we can determine similarly
the contribution of the contracted tree
\begin{equation*}
\begin{aligned}
    \<1130_1i>_{\ve}(t,x)
= \frac{1}{2}\Big(\frac{\hat \lambda^2}{2\pi}\Big)^2 \<10>_{\, \ve} (t, x) \cdot \big( 1+o(1) \big)
\,.
\end{aligned}
\end{equation*}
}
\end{example}

\subsection{${\rm v}$-cycles and their removal}\label{sec_rem_vcycle}

Contrary to the above two examples, where only 1-cycles appeared, the next example will demonstrate a
different cycle structure, which will lead to lower order contributions.
This will motivate the study of $\rm{v}$-cycles of arbitrary length, which will play an
important role in our analysis.

\begin{example}\label{example:tidentbound}
{\rm
Let us look at the order of magnitude of  $\scalebox{0.8}{\<30_3>}_{\!\! \emptyset, \ve}$.
Its second moment has the diagrammatic representation in terms of the
completely contracted tree
\begin{align*}
    \EE\left|\<30_3>_{\!\!\! \emptyset, \varepsilon} (t,x)\right|^2
    &= 6
    \vcenter{\hbox{
\begin{tikzpicture}[scale=0.3]
\draw  (0,0) -- (0,1) node[dot] {};
\draw (0,0) -- (.9,1) node[dot] {};
 \draw (0,0)--(-.9,1) node[dot]{} ;
\draw  (4,0) -- (4,1) node[dot] {};
\draw (4,0) -- (4.9,1) node[dot] {};
\draw (4,0)--(3.1,1) node[dot]{} ;
\draw[purple, thick] (0,1) to[out=90,in=90] (4,1);
\draw[purple, thick] (0.9,1) to[out=90,in=90] (3.1,1);
\draw[purple, thick] (-0.9,1) to[out=90,in=90] (4.9,1);
\node at (2.4, -1.5) {\scalebox{0.7}{$(t,x), \ve$}};
\draw  (0,0) node[idot] {}--(2, -1) node[idot] {} --(4,0) node[idot] {} ;
\end{tikzpicture}  }}\,,
\end{align*}
where the factor $6$ counts the number of symmetries of the pairing at hand.
Denoting by $(s_1,y_1)$ and $(s_2,y_2)$ the time-space coordinates of the bases of the left and right tridents, respectively,
we can explicitly write the integral corresponding to the above diagram as
\begin{align*}
     6 \lambda_\varepsilon^6 \int_{D_{t}^2} p_{t-s_1}^{(\mathfrak{m})}(y_1-x)
   \, \big( e^{\mathfrak{m}\, ( s_{1}+ s_{2})} p_{s_1+s_2+2\varepsilon^2}(y_1-y_2) \big)^3
    \, p_{t-s_2}^{( \mathfrak{m})}(y_2-x) \ud y_{1} \ud y_2 \ud s_1 \ud s_2 \;.
\end{align*}
Then, using the estimate $( e^{\mathfrak{m}\, ( s_1 + s_2)} p_{s_1+s_2+2\varepsilon^2}(y_1-y_2))^2\leqslant
 e^{4 \overline{\mathfrak{m}}\, t}(2\pi(s_1+s_2+2\varepsilon^2))^{-2} $ together with Chapman--Kolmogorov, we can bound this by
\begin{align*}
     6 \lambda_\varepsilon^6  e^{4 \overline{\mathfrak{m}}\, t} \int_{D_{t}^2}&\frac{ p_{t-s_1}^{(\mathfrak{m})}(y_1-x)
   \,  e^{ \mathfrak{m}\, ( s_{1}+ s_{2})} p_{s_1+s_2+2\varepsilon^2}(y_1-y_2)
    \, p_{t-s_2}^{(\mathfrak{m})}(y_2-x)}{\big( 2\pi
(s_1+s_2+2\varepsilon^2)\big)^2} \ud y_{1} \ud y_2 \ud s_1 \ud s_2  \\
   & \leqslant    \frac{6 \lambda_\varepsilon^6}{(2\pi)^2}  e^{6
\overline{\mathfrak{m}} \, t}
    p_{2(t+\varepsilon^2)}(0)
    \int_{[0,t]^2}
    \frac{1}{(s_1+s_2+2 \ve^2)^2} \ud s_1 \ud s_2 \\
    &\leqslant  \frac{6 \lambda_\varepsilon^6}{(2\pi)^2} \log
(1+\tfrac{1}{2}t\varepsilon^{-2})\,  e^{6 \overline{\mathfrak{m}}\, t}  p_{2(t+\varepsilon^2)}(0) \,.
\end{align*}
Since $\lambda_\varepsilon^6 = O((\log \tfrac{1}{\ve})^{-3})$, we can conclude that
 $   \EE\Big[\,\Big| {\sqrt{ \log{\tfrac{1}{\ve}} } \<30_3>_{\!\! \emptyset, \ve}} (t,x) \Big|^2 \Big]
    \leqslant  \tfrac{C(t)}{\log{\frac{1}{\ve}}}\,,$
for some constant $C(t)$ only depending on $t>0$.
}
\end{example}
In the last example there was no 1-cycle appearing. Instead, the contracted tree
that emerged from the diagrammatic representation of
the second moment, presented cycles containing {\it more than one} inner
vertex, with
every edge of the cycle incident to at least one leaf.
We will call such cycles \emph{${\rm v}$-cycles}. The emergence of ${\rm
v}$-cycles and the quantitative estimate of their
contribution will play a crucial role.
The key observation is that contracted trees which do not consist
of 1-cycles {\it only}, will have their second moment represented by a paired
tree which necessarily contains a ${\rm v}$-cycle of length strictly greater
than one. Such trees will turn out to have a lower order contribution.
The main estimates in this section, which provide a quantitative control on
${\rm v}$-cycles, are given in Lemmas \ref{lem:cnclmcycl}  and
\ref{lem_v_cycle_order_bnd_new} below. Let us start with the rigorous
definition of a ${\rm v}$-cycle.
\begin{definition}\label{def:vcycle}
For a given contracted tree $ \tau_{\kappa}, $ a subgraph $ \mC = (\mV_{\mC},
\mE_{\mC}) \subseteq
\tau_{\kappa} $ is a \textbf{${\rm v}$-cycle} if it is a cycle in $ \tau_{\kappa} $
(viewed as a graph) in which every edge is incident to at least one leaf of
the tree $ \tau $. We define the \emph{length} of a ${\rm v}$-cycle to be the number of
inner nodes of $ \tau $ contained in $ \mC $ and we also denote by $\mI_{\mC}$ and $\mL_{\mC}$  the collection
of the inner vertices and leaves of $ \tau $ that belong to $\mC$, respectively.
We will call a ${\rm v}$-cycle of length $ m \in  \NN$ a $ m $-cycle for short.
\end{definition}
A pictorial example of a ${\rm v}$-cycle is the one that appears in the following component of a contracted tree:
 \begin{align*}
        \vcenter
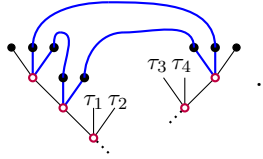
{\hbox{
        \begin{tikzpicture}[scale=0.4]
        \draw  (0,0)   -- (-.7,1) node[dot] {}  ;
        \draw[blue, thick]  (0,0)   -- (0,1) node[dot] {}  ;
        \draw[blue, thick] (0,0) -- (.7,1) node[dot] {};
        \draw (0,0) -- (1,-1);
        \draw[blue, thick] (1,0) node[dot] {} -- (1,-1);
        \draw[blue, thick] (1.7,0) node[dot] {} -- (1,-1);
        \draw[blue, thick]  (6,0)   -- (5.3,1) node[dot] {}  ;
        \draw (6,0) -- (6.7,1) node[dot] {};
        \draw[blue, thick] (6,0) -- (6,1) node[dot] {};
        \draw (6,0) -- (5,-1);
        \draw[thick, dotted] (4.9,-1.1) -- (4.5,-1.5);
        \draw[thick, dotted] (2.1,-2.1) -- (2.6,-2.6);
        \draw (5,-1)--(5,0); \node at (4.9,0.4) {\scalebox{0.8}{$\tau_4$}};
        \draw (5,-1)--(4.3,0); \node at (4.1,0.4) {\scalebox{0.8}{$\tau_3$}};
        \draw (1,-1) -- (2,-2);  \draw (2,-2) -- (2,-1);  \node at (2,.-0.8) {\scalebox{0.8}{$\tau_1$}};
         \draw (2,-2) -- (2.7,-1);  \node at (2.8,.-0.8) {\scalebox{0.8}{$\tau_2$}};
        \draw[blue, thick] (0,1) to[out=90,in=180] (2.5,2.5);
        \draw[blue, thick] (2.5,2.5) to[out=0,in=90] (6,1);
        \draw[blue, thick] (0.7,1) to[out=90,in=180] (.9,1.5);
        \draw[blue, thick] (0.9,1.5) to[out=0,in=90] (1,0);
        \draw[blue, thick] (1.7,0) to[out=90,in=180] (3,1.5);
        \draw[blue, thick] (3,1.5) to[out=0,in=90] (5.3,1);
	\node[idot] at (0,0) {};
	\node[idot] at (1,-1) {};
	\node[idot] at (2,-2) {};
	\node[idot] at (5,-1) {};
	\node[idot] at (6,0) {};
        \end{tikzpicture}  }}
        \,.
    \end{align*}
Note that a $1$-cycle (Definition~\ref{def:onecycle}) is simply a ${\rm v}$-cycle of length $1$.
    Consider a ${\rm v}$-cycle of length $m$ and denote its inner vertices by $v_1,...,v_m$,
    where we will always keep the convention that in such an encoding we start from the left-most
inner vertex of the cycle, in the graph picture of the tree, and register the following inner vertices
 as we trace the cycle clock-wise. Let us also denote the time coordinates of $v_1,...,v_m$ by
$s_1,...,s_m$, respectively. We introduce the following kernel, which will play
an important role in our estimates:
\begin{equation}\label{e_def_vident}
\begin{aligned}
	\<2cycle>^{\otimes m}(s_1, \cdots,  s_m)
	& := \lambda_{\ve}^{2m} \prod_{k=1}^{m} e^{\mathfrak{m} ( s_{k}+ s_{k+1})}
	p_{s_{k}+s_{k+1}+2\varepsilon^2}(0)\\
	& = \prod_{k=1}^{m} \frac{\lambda_{\ve}^{2} e^{\mathfrak{m} ( s_{k}+
	s_{k+1})}  }{2 \pi (s_{k}+ s_{k+1}+ 2 \ve^{2})}\,,
\end{aligned}
\end{equation}
with the convention that $s_{m+1}=s_1$. We note that the above kernel is invariant under
cyclic permutation of $s_1, s_2,...,s_m$.

The following lemma, which will be proved in Appendix~\ref{sec_v_cycle_exist},
establishes the existence of a ${\rm v}$-cycle in a completely contracted
tree of the form $[\tau_1,\tau_2]$.
\begin{lemma}\label{vcycleexists}
Let $\tau_1,\tau_2 \in \mT_3$ . Then, for every pairing
$\gamma\in\mY(\tau_1,\tau_2)$, the paired tree
$[\tau_1,\tau_2]_\gamma$ contains a ${\rm v}$-cycle.
\end{lemma}

An important procedure that we will follow in the remainder of the section, is
to spatially decouple ${\rm v}$-cycles from the rest of the integration kernel
encoded by the tree. We will be performing
such decoupling estimates sequentially until we exhaust all ${\rm v}$-cycles,
including the ${\rm v}$-cycles that will emerge through
this process. We will call this process {\it cycle removal}.
Below we formally define the cycle removal of a single cycle.

\begin{definition}[Cycle removal]\label{def:cycl-rmvl}
For any contracted  tree $ \tau_{\kappa} $, with $ \tau \in \mT $ of the form
$ \tau = [ \tau_{1} \cdots \tau_{n}]$, $ \tau_{i} \in \mT_{3}$, and any ${\rm v}$-cycle $ \mC \subseteq \tau_{\kappa}
$ (not passing through the root),  we define the contracted tree $ ( \tau \setminus \mC)_{ \widetilde{ \kappa}}
$ to be
the contracted tree obtained from $ \tau_{\kappa} $ through the following procedure:
\begin{enumerate}
\item Remove from $ \tau_\kappa $ all the edges and vertices that belong to $ \mC $.
\item We note that an inner vertex that belongs to $\mC$ will
always have four neighbours due to the tree being built from ternary trees.
Let $v$ be an inner vertex that belongs to $\mC$.
We then write $ v_{0}= \mf{p}(v)$ for its parent and $ v_{2}$ for the unique descendant
which is not connected to $v$ by an edge in $\mC$ (recall the representation
\eqref{remove1cycle}). We then distinguish between the following two cases:

\begin{enumerate}
\item Both $ v_{0}$ and $ v_{2}$ are not part of $\mC$, then we replace
the two edges $ \{ v_{0}, v \} $ and $\{ v_{2}, v \} $ by a single edge $
\{ v_{0}, v_{2} \} $.

\item If $ v_{0}$ (or $ v_{2}$) are part of $\mC$, we proceed going down (or
up) in the tree until we find a vertex $ w_{0}$ (or $ w_{2}$) not being part of $\mC$. Once, we
found such vertices we remove all edges crossed in this exploration (which
don't belong to $\mC$ since they connect inner nodes) and replace
them by the edge $ \{w_{0}, w_{2}\}$.\footnote{Note that we will always find vertices $
w_{0}, w_{2}$ (which might agree with $ v_{0}, v_{2} $), since going down we
will always hit the root last (which is not part of $\mC$), and going up we can
always proceed because the sub-trees are ternary.}
\end{enumerate}
\end{enumerate}
The result of this procedure is  a tree $ \tau \setminus \mC$ with the
contraction $ \widetilde{\kappa} $ consisting of the remaining edges
in $\kappa$ after the above two steps.
\end{definition}
For example, consider the following component of a tree
    \begin{align*}
        \vcenter
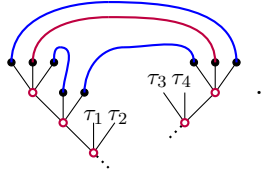
{\hbox{
        \begin{tikzpicture}[scale=0.4]
        \draw  (0,0)   -- (-.7,1) node[dot] {}  ;
        \draw  (0,0)   -- (0,1) node[dot] {}  ;
        \draw (0,0) -- (.7,1) node[dot] {};
        \draw (0,0) -- (1,-1);
        \draw (1,0) node[dot] {} -- (1,-1);
        \draw (1.7,0) node[dot] {} -- (1,-1);
        \draw  (6,0)   -- (5.3,1) node[dot] {}  ;
        \draw (6,0) -- (6.7,1) node[dot] {};
        \draw (6,0) -- (6,1) node[dot] {};
        \draw (6,0) -- (5,-1);
        \draw[thick, dotted] (4.9,-1.1) -- (4.5,-1.5);
        \draw[thick, dotted] (2.1,-2.1) -- (2.6,-2.6);
        \draw (5,-1)--(5,0); \node at (4.9,0.4) {\scalebox{0.8}{$\tau_4$}};
        \draw (5,-1)--(4.3,0); \node at (4.1,0.4) {\scalebox{0.8}{$\tau_3$}};
        \draw (1,-1) -- (2,-2);  \draw (2,-2) -- (2,-1);  \node at (2,.-0.8) {\scalebox{0.8}{$\tau_1$}};
         \draw (2,-2) -- (2.7,-1);  \node at (2.8,.-0.8) {\scalebox{0.8}{$\tau_2$}};
        \draw[blue, thick] (-.7,1) to[out=90,in=180] (2.5,3);
        \draw[blue, thick] (2.5,3) to[out=0,in=90] (6.7,1);
        \draw[purple, thick] (0,1) to[out=90,in=180] (2.5,2.5);
        \draw[purple, thick] (2.5,2.5) to[out=0,in=90] (6,1);
        \draw[blue, thick] (0.7,1) to[out=90,in=180] (.9,1.5);
        \draw[blue, thick] (0.9,1.5) to[out=0,in=90] (1,0);
        \draw[blue, thick] (1.7,0) to[out=90,in=180] (3,1.5);
        \draw[blue, thick] (3,1.5) to[out=0,in=90] (5.3,1);
	\node[idot] at (0,0) {};
	\node[idot] at (1,-1) {};
	\node[idot] at (2,-2) {};
	\node[idot] at (5,-1) {};
	\node[idot] at (6,0) {};
        \end{tikzpicture}  }}
        \,.
    \end{align*}
    In this example there are three ${\rm v}$-cycles:  one whose
contraction-edges consists of only the blue edges,
    one whose $\kappa$ edges consists of the top blue edge and the red one and, finally, one whose
    $\kappa$ edges consists of the red one and the two blue ones below the
red.
    The process of removing the blue ${\rm v}$-cycle is depicted below.
    The middle step shows the component after removing the leaves and edges that are part of the
${\rm v}$-cycle. The rightmost tree is the final outcome after also replacing the edges incident to the
inner vertices of the ${\rm v}$-cycle by a single edge.
    \begin{align*}
        \vcenter
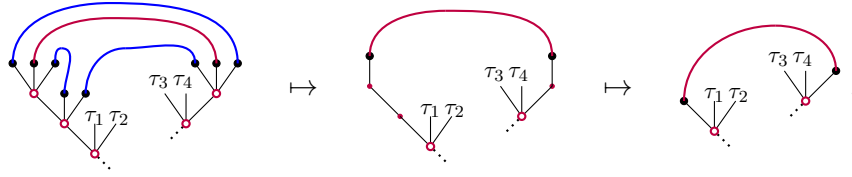
{\hbox{
               \begin{tikzpicture}[scale=0.4]
        \draw  (0,0)   -- (-.7,1) node[dot] {}  ;
        \draw  (0,0)   -- (0,1) node[dot] {}  ;
        \draw (0,0) -- (.7,1) node[dot] {};
        \draw (0,0) -- (1,-1);
        \draw (1,0) node[dot] {} -- (1,-1);
        \draw (1.7,0) node[dot] {} -- (1,-1);
        \draw  (6,0)   -- (5.3,1) node[dot] {}  ;
        \draw (6,0) -- (6.7,1) node[dot] {};
        \draw (6,0) -- (6,1) node[dot] {};
        \draw (6,0) -- (5,-1);
        \draw[thick, dotted] (4.9,-1.1) -- (4.5,-1.5);
        \draw[thick, dotted] (2.1,-2.1) -- (2.6,-2.6);
        \draw (5,-1)--(5,0); \node at (4.9,0.4) {\scalebox{0.8}{$\tau_4$}};
        \draw (5,-1)--(4.3,0); \node at (4.1,0.4) {\scalebox{0.8}{$\tau_3$}};
        \draw (1,-1) -- (2,-2);  \draw (2,-2) -- (2,-1);  \node at (2,.-0.8) {\scalebox{0.8}{$\tau_1$}};
         \draw (2,-2) -- (2.7,-1);  \node at (2.8,.-0.8) {\scalebox{0.8}{$\tau_2$}};
        \draw[blue, thick] (-.7,1) to[out=90,in=180] (2.5,3);
        \draw[blue, thick] (2.5,3) to[out=0,in=90] (6.7,1);
        \draw[purple, thick] (0,1) to[out=90,in=180] (2.5,2.5);
        \draw[purple, thick] (2.5,2.5) to[out=0,in=90] (6,1);
        \draw[blue, thick] (0.7,1) to[out=90,in=180] (.9,1.5);
        \draw[blue, thick] (0.9,1.5) to[out=0,in=90] (1,0);
        \draw[blue, thick] (1.7,0) to[out=90,in=180] (3,1.5);
        \draw[blue, thick] (3,1.5) to[out=0,in=90] (5.3,1);
	\node[idot] at (0,0) {};
	\node[idot] at (1,-1) {};
	\node[idot] at (2,-2) {};
	\node[idot] at (5,-1) {};
	\node[idot] at (6,0) {};
        \end{tikzpicture}  }}
        \quad \mapsto \quad
        \vcenter{\hbox{
               \begin{tikzpicture}[scale=0.4]
       \draw  (0,0)   -- (0,1) node[dot] {}  ;
       \draw (0,0) -- (1,-1);
       \draw [fill, purple] (0,0)  circle [radius=0.07];
       \draw [fill, purple] (1,-1)  circle [radius=0.07];
       \draw [fill, purple] (6,0)  circle [radius=0.07];
       \draw (6,0) -- (6,1) node[dot] {};
       \draw (6,0) -- (5,-1);
       \draw[thick, dotted] (4.9,-1.1) -- (4.5,-1.5);
       \draw[thick, dotted] (2.1,-2.1) -- (2.6,-2.6);
       \draw (5,-1)--(5,0); \node at (4.9,0.4) {\scalebox{0.8}{$\tau_4$}};
       \draw (5,-1)--(4.3,0); \node at (4.1,0.4) {\scalebox{0.8}{$\tau_3$}};
       \draw (1,-1) -- (2,-2);  \draw (2,-2) -- (2,-1);  \node at (2,.-0.8) {\scalebox{0.8}{$\tau_1$}};
       \draw (2,-2) -- (2.7,-1);  \node at (2.8,.-0.8) {\scalebox{0.8}{$\tau_2$}};
       \draw[purple, thick] (0,1) to[out=90,in=180] (2.5,2.5);
       \draw[purple, thick] (2.5,2.5) to[out=0,in=90] (6,1);
	\node[idot] at (2,-2) {};
	\node[idot] at (5,-1) {};
        \end{tikzpicture}  }}
        \quad \mapsto \quad
         \vcenter{\hbox{
               \begin{tikzpicture}[scale=0.4]
       \draw (6,0) node[dot] {} -- (5,-1);
       \draw[thick, dotted] (4.9,-1.1) -- (4.5,-1.5);
       \draw[thick, dotted] (2.1,-2.1) -- (2.6,-2.6);
       \draw (5,-1)--(5,0); \node at (4.9,0.4) {\scalebox{0.8}{$\tau_4$}};
       \draw (5,-1)--(4.3,0); \node at (4.1,0.4) {\scalebox{0.8}{$\tau_3$}};
       \draw (1,-1) node[dot] {} -- (2,-2);  \draw (2,-2) -- (2,-1);  \node at (2,.-0.8) {\scalebox{0.8}{$\tau_1$}};
       \draw (2,-2) -- (2.7,-1);  \node at (2.8,.-0.8) {\scalebox{0.8}{$\tau_2$}};
       \draw[purple, thick] (1,-1) to[out=90,in=180] (4,1.5);
       \draw[purple, thick] (4,1.5) to[out=0,in=90] (6,0);
	\node[idot] at (2,-2) {};
	\node[idot] at (5,-1) {};
        \end{tikzpicture}  }} .
    \end{align*}
The following lemma provides the central estimate that quantifies the contributions coming from ${\rm v}$-cycles.
\begin{lemma}\label{lem:cnclmcycl}
Let $ \tau \in \mT$ of the form $ \tau = [ \tau_{1}\cdots \tau_{n}]$, $
\tau_{i} \in \mT_{3}$, and $ \kappa \in \mK ( \tau)$. 
For any $ m \in \NN $, let $ \mC $ be an $ m $-cycle in the contracted tree $\tau_{\kappa} $
 and denote the inner vertices of $\mC$ by $v_1,...,v_m$ 
  with associated time-space
coordinates $(s_{v_1},y_{v_1}),...,(s_{v_m}, y_{v_m})$.
Recall the kernel $K_{\tau_{\kappa,\ve}} ( s_{\mV}, y_{\mV})$ from \eqref{e:tk2}
associated to a contracted tree.
Then
\begin{equation}\label{removal_estimate}
\begin{aligned}
\int_{( \RR^{2})^{\mI_{\mC}}}
\int_{D_{t}^{ \mL_{\mC} }}
& K_{\tau_{\kappa,\ve}} (s_{\mV}, y_{\mV})  \ud s_{\mL_{\mC}} \ud
y_{\mV_{\mC}} \\
& \leqslant  \<2cycle>^{\otimes m} (s_{v_1},...,s_{v_m})
\left\{
\prod_{ v \in \mI_{\mC}} \mathds{1}_{ s_{\mf{d}_{ \kappa}(v)} \leqslant s_v \leqslant s_{ \mathfrak{p}(v)}}
\right\}
K_{ \widetilde{\tau}_{ \widetilde{\kappa}}, \ve} (s_{\mV \setminus \mC }, y_{\mV \setminus \mC}) \;,
\end{aligned}
\end{equation}
where $ \widetilde{\tau} = \tau \setminus \mC $, $ \widetilde{\kappa} $ is the
contraction induced by $ \kappa $ on $ \widetilde{\tau} $, $\mathfrak{p}(v)$
denotes the parent of a vertex $v$ in $ \tau$ and $\mf{d}_{ \kappa}(v)$ denotes the
unique descendant of $v$ which is \emph{not} part of $\mC$.
\end{lemma}

\begin{proof}
The proof follows the steps of the computation in Example
\ref{example:tidentbound} by crucially applying
Chapman--Kolmogorov to the integration over the space variables associated to the leaves of the
${\rm v}$-cycle, combined with a uniform bound over the space variables on the resulting product of heat kernels.
In particular, we have
\begin{equation*}
\begin{aligned}
 \int_{(\RR^{2})^{\mL_{\mC}}}& \prod_{\{w,v\} \, \in \,
\mE_{\mC} \, \colon \, v \,\in \, \mI_\mC} p_{s_v }^{(\mathfrak{m})}(y_v - y_w)
\prod_{ \{w, v\} \, \in \, \kappa \cap \mE_{\mC}  } \lambda_{\ve}^{2} \,
    p_{2\ve^2}(y_{w} - y_{v})   \,  \ud y_{\mL_{\mC}}  \\
& = \lambda_{\ve}^{2 m} \prod_{k=1}^{m} e^{\mathfrak{m}( s_{v_{k}}+
s_{v_{k+1}})} p_{s_{v_{k}}+s_{v_{k+1}}+2\varepsilon^2}(y_{v_{k+1}}-y_{v_k})
\,\, \\
&\leqslant  \lambda_{\ve}^{2 m} \prod_{k=1}^{m} e^{\mathfrak{m}( s_{v_{k}}+
s_{v_{k+1}})}p_{s_{v_{k}}+s_{v_{k+1}}+2\varepsilon^2}(0) \\
& = \,\, \<2cycle>^{\otimes m}( (s_{v_{k}})_{k =1, \ldots, m}) \;,
\end{aligned}
\end{equation*}
where we omitted integration over the time variables associated to
leaves on the left-hand side, as $ s_{w}=0$ for all $ w \in \mL_{\mC}$, because
of the Dirac--\(\delta\) at zero.
Inserting the above into the left-hand side of \eqref{removal_estimate}
we obtain the desired estimate.
It is useful to have a pictorial representation of the estimate we have just
performed:
 \begin{equation*}
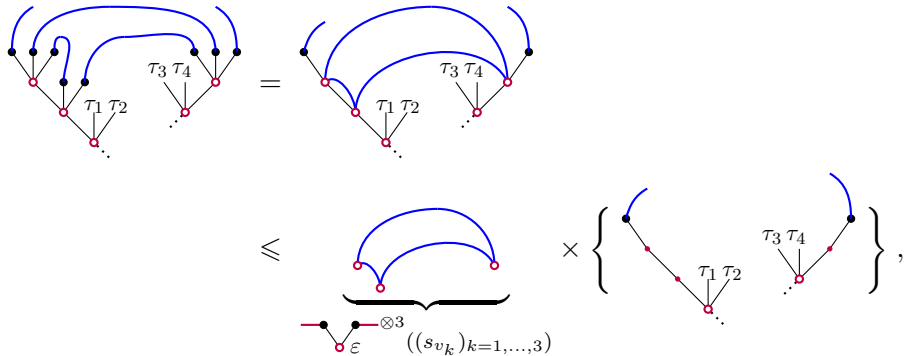
\begin{aligned}
        \vcenter{\hbox{
        \begin{tikzpicture}[scale=0.4]
        \draw  (0,0)   -- (-.7,1) node[dot] {}  ;
        \draw  (0,0)   -- (0,1) node[dot] {}  ;
        \draw (0,0) -- (.7,1) node[dot] {};
        \draw (0,0) -- (1,-1);
        \draw (1,0) node[dot] {} -- (1,-1);
        \draw (1.7,0) node[dot] {} -- (1,-1);
        \draw  (6,0)   -- (5.3,1) node[dot] {}  ;
        \draw (6,0) -- (6.7,1) node[dot] {};
        \draw (6,0) -- (6,1) node[dot] {};
        \draw (6,0) -- (5,-1);
        \draw[thick, dotted] (4.9,-1.1) -- (4.5,-1.5);
        \draw[thick, dotted] (2.1,-2.1) -- (2.6,-2.6);
        \draw (5,-1)--(5,0); \node at (4.9,0.4) {\scalebox{0.8}{$\tau_4$}};
        \draw (5,-1)--(4.3,0); \node at (4.1,0.4) {\scalebox{0.8}{$\tau_3$}};
        \draw (1,-1) -- (2,-2);  \draw (2,-2) -- (2,-1);  \node at (2,.-0.8) {\scalebox{0.8}{$\tau_1$}};
         \draw (2,-2) -- (2.7,-1);  \node at (2.8,.-0.8) {\scalebox{0.8}{$\tau_2$}};
        \draw[blue, thick] (0,1) to[out=90,in=180] (2.5,2.5);
        \draw[blue, thick] (2.5,2.5) to[out=0,in=90] (6,1);
        \draw[blue, thick] (0.7,1) to[out=90,in=180] (.9,1.5);
        \draw[blue, thick] (0.9,1.5) to[out=0,in=90] (1,0);
        \draw[blue, thick] (1.7,0) to[out=90,in=180] (3,1.5);
        \draw[blue, thick] (3,1.5) to[out=0,in=90] (5.3,1);
         \draw[blue, thick] (-0.7,1) to[out=90,in=210] (0,2.5);
        \draw[blue, thick] (6.7,1) to[out=90,in=-30] (6,2.5);
	\node[idot] at (0,0) {};
	\node[idot] at (1,-1) {};
	\node[idot] at (2,-2) {};
	\node[idot] at (5,-1) {};
	\node[idot] at (6,0) {};
        \end{tikzpicture}  }}
&=
        \vcenter{\hbox{
        \begin{tikzpicture}[scale=0.4]
        \draw  (0,0)   -- (-.7,1) node[dot] {}  ;
        \draw (0,0) -- (1,-1);
        \draw (6,0) -- (6.7,1) node[dot] {};
        \draw (6,0) -- (5,-1);
        \draw[thick, dotted] (4.9,-1.1) -- (4.5,-1.5);
        \draw[thick, dotted] (2.1,-2.1) -- (2.6,-2.6);
        \draw (5,-1)--(5,0); \node at (4.9,0.4) {\scalebox{0.8}{$\tau_4$}};
        \draw (5,-1)--(4.3,0); \node at (4.1,0.4) {\scalebox{0.8}{$\tau_3$}};
        \draw (1,-1) -- (2,-2);  \draw (2,-2) -- (2,-1);  \node at (2,.-0.8) {\scalebox{0.8}{$\tau_1$}};
         \draw (2,-2) -- (2.7,-1);  \node at (2.8,.-0.8) {\scalebox{0.8}{$\tau_2$}};
        \draw[blue, thick] (0,0) to[out=90,in=180] (3.5,2.5);
        \draw[blue, thick] (3.5,2.5) to[out=0,in=90] (6,0);
        \draw[blue, thick] (0,0) to[out=45,in=90] (1, -1);
        \draw[blue, thick] (1,-1) to[out=90,in=115] (6,0);
        \draw[blue, thick] (-0.7,1) to[out=75,in=210] (0,2);
        \draw[blue, thick] (6.7,1) to[out=90,in=-30] (6,2.5);
	\node[idot] at (0,0) {};
	\node[idot] at (1,-1) {};
	\node[idot] at (2,-2) {};
	\node[idot] at (5,-1) {};
	\node[idot] at (6,0) {};
        \end{tikzpicture}  }}\\
 &\leqslant \,\,
\underset{\<2cycle>^{\otimes 3}( (s_{v_{k}})_{k =1, \ldots, 3})}{
\underbrace{
          \vcenter{\hbox{
               \begin{tikzpicture}[scale=0.3]
          \draw[blue, thick] (0,0) to[out=90,in=180] (3.5,2.5);
        \draw[blue, thick] (3.5,2.5) to[out=0,in=90] (6,0);
        \draw[blue, thick] (0,0) to[out=45,in=90] (1, -1);
        \draw[blue, thick] (1,-1) to[out=90,in=115] (6,0);
	\node[idot] at (0,0) {};
	\node[idot] at (1,-1) {};
	\node[idot] at (6,0) {};
       \end{tikzpicture}  }}
       }
       }
       \times \Bigg\{
          \vcenter{\hbox{
        \begin{tikzpicture}[scale=0.4]
        \draw  (0,0)   -- (-.7,1) node[dot] {}  ;
        \draw (0,0) -- (1,-1);
        \draw (6,0) -- (6.7,1) node[dot] {};
        \draw (6,0) -- (5,-1);
        \draw[thick, dotted] (4.9,-1.1) -- (4.5,-1.5);
        \draw[thick, dotted] (2.1,-2.1) -- (2.6,-2.6);
        \draw (5,-1)--(5,0); \node at (4.9,0.4) {\scalebox{0.8}{$\tau_4$}};
        \draw (5,-1)--(4.3,0); \node at (4.1,0.4) {\scalebox{0.8}{$\tau_3$}};
        \draw (1,-1) -- (2,-2);  \draw (2,-2) -- (2,-1);  \node at (2,.-0.8) {\scalebox{0.8}{$\tau_1$}};
         \draw (2,-2) -- (2.7,-1);  \node at (2.8,.-0.8) {\scalebox{0.8}{$\tau_2$}};
        \draw[blue, thick] (-0.7,1) to[out=75,in=210] (0,2);
        \draw[blue, thick] (6.7,1) to[out=90,in=-30] (6,2.5);
	\node[idot] at (2,-2) {};
	\node[idot] at (5,-1) {};
	\draw [fill, purple] (0,0)  circle [radius=0.07];
 	\draw [fill, purple] (1,-1)  circle [radius=0.07];
	\draw [fill, purple] (6,0)  circle [radius=0.07];
        \end{tikzpicture}  }}
        \Bigg\} \;,
\end{aligned}
    \end{equation*}
with the equality representing the application of Chapman--Kolmogorov, with the
blue cycle appearing on the right-hand side in the first line
representing the weight $$\lambda_{\ve}^{2 m}
\prod_{k=1}^{m}  e^{\mathfrak{m}( s_{v_{k}}+
s_{v_{k+1}})} p_{s_{v_{k}}+s_{v_{k+1}}+2\varepsilon^2}(y_{v_{k+1}}-y_{v_k})
\,,$$
and the inequality depicting the application of the uniform bound
 $$\prod_{k=1}^{m}  e^{\mathfrak{m}( s_{v_{k}}+
s_{v_{k+1}})}p_{s_{v_{k}}+s_{v_{k+1}}+2\varepsilon^2}(y_{v_{k+1}}-y_{v_k}) \leqslant
 \prod_{k=1}^{m}  e^{\mathfrak{m}( s_{v_{k}}+
s_{v_{k+1}})}p_{s_{v_{k}}+s_{v_{k+1}}+2\varepsilon^2}(0) \,.$$
The blue cycle appearing in the right-hand side in the second line represents
the space-independent kernel
 $\<2cycle>^{\otimes m}( (s_{v_{k}})_{k =1, \ldots, m})$ (hence we call this a spatial
decoupling from the remaining integral). The small red nodes
indicate the remaining spatial integrals associated to inner nodes of the extracted
cycle.
By another application of Chapman--Kolmogorov, when
integrating over the  spatial variables associated to the small red
nodes, we obtain
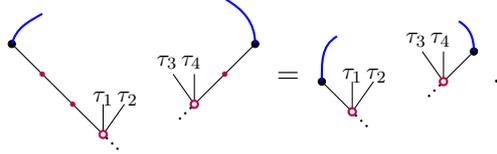
\begin{equation*}
\begin{aligned}
\vcenter{\hbox{
        \begin{tikzpicture}[scale=0.4]
        \draw  (0,0)   -- (-1,1) node[dot] {}  ;
        \draw (0,0) -- (1,-1);
        \draw (6,0) -- (7,1) node[dot] {};
        \draw (6,0) -- (5,-1);
        \draw[thick, dotted] (4.9,-1.1) -- (4.5,-1.5);
        \draw[thick, dotted] (2.1,-2.1) -- (2.6,-2.6);
        \draw (5,-1)--(5,0); \node at (4.9,0.4) {\scalebox{0.8}{$\tau_4$}};
        \draw (5,-1)--(4.3,0); \node at (4.1,0.4) {\scalebox{0.8}{$\tau_3$}};
        \draw (1,-1) -- (2,-2);  \draw (2,-2) -- (2,-1);  \node at (2,.-0.8) {\scalebox{0.8}{$\tau_1$}};
         \draw (2,-2) -- (2.7,-1);  \node at (2.8,.-0.8) {\scalebox{0.8}{$\tau_2$}};
        \draw[blue, thick] (-1,1) to[out=75,in=210] (0,2);
        \draw[blue, thick] (7,1) to[out=90,in=-30] (6,2.5);
	\node[idot] at (2,-2) {};
	\node[idot] at (5,-1) {};
	\draw [fill, purple] (0,0)  circle [radius=0.07];
 	\draw [fill, purple] (1,-1)  circle [radius=0.07];
	\draw [fill, purple] (6,0)  circle [radius=0.07];
        \end{tikzpicture}  }}
=
          \vcenter{\hbox{
               \begin{tikzpicture}[scale=0.4]
       \draw (6,0) node[dot] {} -- (5,-1);
       \draw[thick, dotted] (4.9,-1.1) -- (4.5,-1.5);
       \draw[thick, dotted] (2.1,-2.1) -- (2.6,-2.6);
       \draw (5,-1)--(5,0); \node at (4.9,0.4) {\scalebox{0.8}{$\tau_4$}};
       \draw (5,-1)--(4.3,0); \node at (4.1,0.4) {\scalebox{0.8}{$\tau_3$}};
       \draw (1,-1) node[dot] {}  -- (2,-2);  \draw (2,-2) -- (2,-1);  \node at (2,.-0.8) {\scalebox{0.8}{$\tau_1$}};
       \draw (2,-2) -- (2.7,-1);  \node at (2.8,.-0.8) {\scalebox{0.8}{$\tau_2$}};
        \draw[blue, thick] (1,-1) to[out=90,in=210] (1.5,0.5);
        \draw[blue, thick] (6,0) to[out=90,in=-30] (5.5,1);
	\node[idot] at (2,-2) {};
	\node[idot] at (5,-1) {};
        \end{tikzpicture}  }}\,.
\end{aligned}
\end{equation*}
This concludes the proof.
\end{proof}
The following lemma is crucial as it demonstrates that ${\rm v}$-cycles of length
larger than one have a vanishing contribution, as $\ve\to 0$, and in fact, the
contribution is even smaller the larger the cycle is (because of the factor $
\lambda_{\ve}^{2m} $).
\begin{lemma}\label{lem_v_cycle_order_bnd_new}
The following bound holds for any $m\geqslant 1$
\begin{align}\label{eq_lem_v_cycle_new}
    \int_{[0,t]^m} \<2cycle>^{\otimes m}(s_1, \dots, s_m)  \ud s_{1, \ldots, m}
    \leqslant
    \frac{(\lambda_\varepsilon e^{ \overline{\mathfrak{m}}\, t} )^{2m}}{2^m\, \pi}
   \log \big(
   1+\tfrac{t}{\varepsilon^2}
   \big)  \;,
\end{align}
where we recall that $ \overline{\mathfrak{m}} =  \max \{ \mf{m}\;, 0 \}$.
\end{lemma}
\begin{proof}

In the case $m=1$, the bound follows from \eqref{e:1cycle}, since
\begin{equation}\label{e:integral_1cycle}
\begin{aligned}
\int_{0}^{t} \<cycle1>_{\ve} (s) \ud s
= \lambda_\ve^2\int_0^t \frac{ e^{2\mathfrak{m}\,s} }{4\pi (s+\ve^2)} \ud s
 \leqslant \frac{\lambda_{\ve}^{2} e^{ 2\overline{\mathfrak{m}}\, t}}{2 \pi} \log{(1 + \tfrac{t}{ \ve^{2}}}) \;.
\end{aligned}
\end{equation}
Thus, we assume $m\geqslant 2$ for the remainder of the proof.
First, note that 
\begin{align*}
s_1+s_m+2\ve^2 \geqslant 2 \sqrt{(s_1+\ve^2)(s_m+\ve^2) } \geqslant
\sqrt{(s_1+2\ve^2)(s_m+2\ve^2) }\,.
\end{align*}
Therefore:
\begin{equation}\label{eq_cycle_bound_step1}
\begin{aligned}
    \int_{[0,t]^m} &\<2cycle>^{\otimes m}(s_1,...,s_m) \ud s_{1, \ldots, m}
    =
    \frac{\lambda_\varepsilon^{2m}}{(2\pi)^m}
    \int\limits_{[0,t]^m}
   \prod_{k=1}^{m}\frac{e^{\mathfrak{m}(s_{k}+
s_{k+1})}
}{s_{k}+s_{k+1}+2\varepsilon^2}   \ud  s_{1, \ldots, m} \\
   &\leqslant
   \frac{(\lambda_\varepsilon e^{ \overline{\mathfrak{m}} \, t} )^{2m}}{(2\pi)^m}
    \int\limits_{[0,t]^m}
   \frac{1}{\sqrt{s_1+2\varepsilon^2}}
   \frac{1}{\sqrt{s_m+ 2\ve^{2}}}
   \prod_{k=1}^{m-1}\frac{1}{s_{k}+s_{k+1}+2\varepsilon^2}  \ud  s_{1, \ldots,
m}\,,
\end{aligned}
\end{equation}
where we additionally used that $e^{\mathfrak{m} \, s_{k}} \leqslant
e^{ \overline{\mathfrak{m}} \, t } $.
Furthermore, for $2\leqslant k\leqslant m$, using the change of variables $r = \tfrac{s_{k}}{s_{k -1}+ 2\ve^{2}}$
together with the identity $\int_{0}^{\infty} \tfrac{1}{\sqrt{r} (1 + r)} \ud r = \pi$, we have that
\begin{align}\label{e:support1}
    \int_0^t \frac{1}{\sqrt{s_k+ 2 \ve^{2}}}
    \frac{ 1}{s_{k-1}+s_k+2\ve^2} \ud s_{k}
& \leqslant \frac{ 1}{\sqrt{s_{k-1} + 2 \ve^{2}}} \int_0^{t/(s_{k-1} + 2
\ve^{2})} \frac{1}{\sqrt{r}} \frac{1}{1+r} \ud r   \nonumber \\
& \leqslant \frac{\pi}{\sqrt{s_{k-1}+2\ve^2}}\,. 
\end{align}
Applying \eqref{e:support1} $(m-1)$-times to \eqref{eq_cycle_bound_step1},
starting from $k=m$ and going down to
$k=2$, yields
\begin{align*}
    \int_{[0,t]^m} \<2cycle>^{\otimes m}(s_1,...,s_m)  \, \ud s_{1, \ldots, m}
    &\leqslant \frac{(\lambda_\varepsilon e^{ \overline{\mathfrak{m}}\, t} )^{2m} \, \pi^{m-1}}{(2\pi)^m}
    \int_{0}^t
\frac{1}{\sqrt{s_1+2\varepsilon^2}}\frac{1}{\sqrt{s_1+2\varepsilon^2}} \ud
s_{1} \\
   & = \frac{(\lambda_\varepsilon e^{ \overline{\mathfrak{m}} \, t} )^{2m}}{2^m\, \pi}
   \log \big( 1+\tfrac{t}{2\varepsilon^2} \big)\,,
\end{align*}
which concludes the proof.
\end{proof}

Next we want to define an iterative process of extracting cycles from a paired tree and
record this process via a mapping to an element of the permutation group.
We will call this the {\it cycle extraction map} and define it below.
To define such algorithm, it will be convenient to label vertices of trees. For a given tree
$ \tau \in \mT_{3}$, we fix a representative ordered version of it and label the inner vertices of $[ \tau, \tau]$ (excluding
the root) with the numbers
$ \{ 1, \ldots, 2i( \tau ) \} $ in arbitrary order.
Once we have labeled all inner nodes, we label its leaves with the integers $\{ 2i(\tau)+1, \ldots, 4i(\tau)+2 \}$.
In particular, the arguments that follow depend on the
ordered structure of the trees under consideration. Yet, the eventual
estimates that we obtain are uniform over all
representatives of a given unordered tree, so this will not cause any problem.
Now, it will be convenient to define an ordering among sequences of labels.

\begin{definition}[Lexicographic ordering]\label{lexiorder}
For two vectors $\mathsf{V}=(v_1,...,v_d) \in \NN^{d} $ and $ \mathsf{U}=(u_1,...,u_e) \in \NN^{e} $
with $ d $ and $ e $ not necessarily equal, we say that $\mathsf{V}$ precedes
$\mathsf{U}$
in lexicographic order and write $\mathsf{V} \prec \mathsf{U}$ if
\begin{itemize}
\item either there exists a $ k \leqslant d \wedge e $ such
that $ v_{z} = u_{z} $ for $ z \leqslant k -1 $ and $ v_{k} < w_{k} $ ,
\item or $ d < e $ and $ v_{z} = u_{z} $ for $ z \leqslant d $.
\end{itemize}
\end{definition}
The lexicographic order extends naturally to a total
order on the set of ${\rm v}$-cycles of a tree.
Let $\mC$ be a ${\rm v}$-cycle, represented by the path-vector
\begin{equation}\label{e_def_pathcycle}
\begin{aligned}
\mathsf{V}_{\mC}:=
(v_{i_1}, v_{j_{1}} , v_{j_{2}}, v_{i_{2}},...,v_{i_m}, v_{j_{2m-1}},
v_{j_{2m}})\,, \quad \text{such
that} \quad v_{i_{k}} \in
\mI_{\mC}\,,\ v_{j_{k}} \in \mL_{\mC} \,,
\end{aligned}
\end{equation}
with consecutive vertices in the vector being connected by an edge in the
${\rm v}$-cycle.
Here $i_{1}$ is the minimal label in $\mI_{\mC}$, and $v_{j_{1}}$ the leaf with
minimal label neighbouring $ v_{i_{1}}$. This imposes a direction the
path-vector is represented in.
Now, let $ \mC '$ be a second ${\rm v}$-cycle represented by the vector $ \mathsf{V}_{\mC '}$,
then
\begin{equation}\label{cycleorder}
\mC \prec \mC'\qquad \text{if} \qquad \mathsf{V}_{\mC} \prec \mathsf{V}_{\mC'}\
\text{lexicographically}.
\end{equation}
In this setting we can introduce the \emph{cycle extraction map} $ \Pi_{\tau} $,
for any $ \tau \in \mT_{3} $.
\begin{definition}[Cycle extraction map]\label{def:cycle-extraction}
For any $ \tau \in \mT_{3} $ and $ \gamma \in \mY (\tau, \tau) $, we will
define inductively a sequence of ${\rm v}$-cycles extracted from $[\tau, \tau]_{\gamma}$
 as follows:
\begin{enumerate}
\item Start by defining  $\sigma_{1} := [\tau, \tau]_{\gamma} $ and $\gamma_1:=\gamma$ and denote by $\mC_1$
the minimal ${\rm v}$-cycle in $\sigma_1$ (whose existence is guaranteed by Lemma~\ref{vcycleexists})
with respect to the lexicographic order.
Define $\sigma_2 :=\sigma_1 \setminus \mC_1$ and on $\sigma_2$ the contraction $\gamma_2$ induced by
 $\gamma_1$ after the cycle removal of $\mC_1$, according to Definition \ref{def:cycl-rmvl}.
\item Assume that we have defined the contracted trees $(\sigma_i)_{\gamma_i}$, for $i=1,...,k$, as well as the
${\rm v}$-cycles $\mC_1,...,\mC_{k-1}$
belonging respectively to $(\sigma_1)_{\gamma_1},...,(\sigma_{k-1})_{\gamma_{k-1}}$. Then
proceed by defining $\mC_k$ to be the
minimal {\rm v}-cycle belonging to $\sigma_k$, with respect to the
lexicographic order. Further, define the contracted graph
 $(\sigma_{k+1})_{\gamma_{k+1}}:= (\sigma_k\setminus \mC_k)_{\tilde\gamma_k}$ via the cycle removal
 as in Definition \ref{def:cycl-rmvl}.
\item Stop at $  K(\tau, \gamma):=k $ if \(\sigma_{k+1} = \<cycle1> \).
\end{enumerate}
\end{definition}
\begin{definition}[Permutation extraction map]\label{def:permCycleMap}
For $ n \in \NN $, let $ S_{n}$ denote the symmetric group over $n$ elements, and
let $\tau\in \mT_3$.
We define as follows the permutation extraction map
\begin{equation*}
\begin{aligned}
\Pi_{\tau} \colon \mY (\tau, \tau)  \to S_{2 i (\tau)} \;.	
\end{aligned}
\end{equation*}
For any $\gamma \in  \mY (\tau, \tau)$ consider the sequence of ${\rm v}$-cycles
$ ( \mC_{k} )_{k =1}^{K (\tau, \gamma)} $ constructed from $[\tau,\tau]_\gamma$ via the cycle extraction map
from Definition \ref{def:cycle-extraction}.
For any $\mC_k$ belonging to this
sequence let $v_{i^{(k)}_1},...,v_{i^{(k)}_{m_k}}$ be the vertices in
$\mI_{\mC_{k}}$, listed in the same
order as in \eqref{e_def_pathcycle}.
We then map every cycle to a permutation cycle 
\begin{align*}
\mC_k  \mapsto
\widehat\mC_k:=
\big(i^{(k)}_1 \  i^{(k)}_2 \ \cdots  \  i^{(k)}_{m_k}\big) \in S_{2
i(\tau)}\,,
\end{align*}
where we used the cycle notation $ (i^{(k)}_1 \  i^{(k)}_2 \ \cdots  \
i^{(k)}_{m_k}) $ for the  permutation $i^{(k)}_j \mapsto
i^{(k)}_{j+1}$, for $j=1, \ldots , m_{k}$, with $ 
i^{(k)}_{m_{k}+1}=i^{(k)}_{1}$.
The image of $ \gamma $ under the permutation extraction map
$ \pi = \Pi_\tau(\gamma)$  is then defined as 
\begin{align*}
\uppi= \prod_{k=1}^{K(\tau, \gamma)}
\big(i^{(k)}_1 \  i^{(k)}_2 \ \cdots  \  i^{(k)}_{m_k}\big)\,.
\end{align*}
\end{definition}

To clarify the tools we have introduced so far, let us discuss an example.

\begin{example}\label{exampl_cycle_extr}
Consider the following paired tree
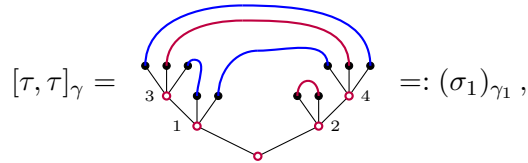
\begin{equation*}
\begin{aligned}
[ \tau , \tau]_{ \gamma}=
 \vcenter{\hbox{
               \begin{tikzpicture}[scale=0.4]
        \draw  (0,0)   -- (-.7,1) node[dot] {}  ;
        \draw  (0,0)   -- (0,1) node[dot] {}  ;
        \draw (0,0) -- (.7,1) node[dot] {};
        \draw (0,0) -- (1,-1);
        \draw (1,0) node[dot] {} -- (1,-1);
        \draw (1.7,0) node[dot] {} -- (1,-1);
        \draw  (6,0)   -- (5.3,1) node[dot] {}  ;
        \draw (6,0) -- (6.7,1) node[dot] {};
        \draw (6,0) -- (6,1) node[dot] {};
        \draw (6,0) -- (5,-1);
	\draw (5,-1) --(3,-2);
        \draw (5,-1)--(5,0)  node[dot] {}  ;
        \draw (5,-1)--(4.3,0)  node[dot] {}  ;
        \draw (1,-1) -- (3,-2);
        \draw[blue, thick] (-.7,1) to[out=90,in=180] (2.5,3);
        \draw[blue, thick] (2.5,3) to[out=0,in=90] (6.7,1);
        \draw[purple, thick] (0,1) to[out=90,in=180] (2.5,2.5);
        \draw[purple, thick] (2.5,2.5) to[out=0,in=90] (6,1);
        \draw[blue, thick] (0.7,1) to[out=90,in=180] (.9,1.2);
        \draw[blue, thick] (0.9,1.2) to[out=0,in=90] (1,0);
        \draw[blue, thick] (1.7,0) to[out=90,in=180] (3,1.5);
        \draw[blue, thick] (3,1.5) to[out=0,in=90] (5.3,1);
	\draw[purple, thick] (4.3,0) to[out=90,in=180] (4.65,0.5);
	\draw[purple, thick] (4.65,0.5) to[out=0,in=90] (5,0);
	\node (c) [label = left:{\tiny $3$}] at (0.3, 0) {};
	\node (a) [label = left:{\tiny $1$}] at (1.2, -1) {};
	\node (b) [label = right:{\tiny $2$}] at (4.7, -1) {};
	\node [label = right:{\tiny $4$}] at (5.7, 0) {};
	\node[idot] at (0,0) {};
	\node[idot] at (1,-1) {};
	\node[idot] at (3,-2) {};
	\node[idot] at (5,-1) {};
	\node[idot] at (6,0) {};
        \end{tikzpicture}  }}=:( \sigma_{1})_{ \gamma_1}\,,
\end{aligned}
\end{equation*}
where we marked the minimal ${\rm v}$-cycle (with respect to lexicographic ordering) in blue,
which will be removed in the first iteration of the cycle extraction.
For the sake of clarity we omitted the labels of the leaves in the diagram
above, and merely represent the corresponding ${\rm v}$-cycles by their inner nodes.
Removing the blue $\rm{v}$-cycle in the diagram above yields
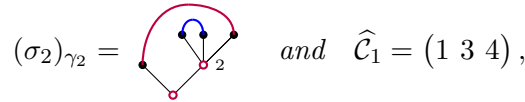
\begin{equation*}
\begin{aligned}
(\sigma_{2})_{ \gamma_{2}}
=
 \vcenter{\hbox{
        \begin{tikzpicture}[scale=0.4]
        \draw (6,0) -- (5,-1);
	\draw (5,-1) --(4,-2);
        \draw (5,-1)--(5,0)  node[dot] {}  ;
        \draw (5,-1)--(4.3,0)  node[dot] {}  ;
	\draw (5,-1)--(6,0)  node[dot] {}  ;
        \draw (3,-1) node[dot] {} -- (4,-2);
        \draw[purple, thick] (3,-1) to[out=90,in=180] (4.5,1);
        \draw[purple, thick] (4.5,1) to[out=0,in=90] (6,0);
	\draw[blue, thick] (4.3,0) to[out=90,in=180] (4.65,0.5);
	\draw[blue, thick] (4.65,0.5) to[out=0,in=90] (5,0);
	\node (b) [label = right:{\tiny $2$}] at (4.7, -1) {};
	\node[idot] at (4,-2) {};
	\node[idot] at (5,-1) {};
        \end{tikzpicture}  }}
\quad \text{and} \quad
\widehat{\mC}_{1} =
\big( 1\ 3 \ 4 \big)\,,
\end{aligned}
\end{equation*}
where once more we marked the new minimal ${\rm v}$-cycle in blue. Removing the
new blue cycle
then yields
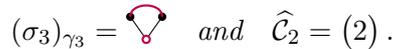
\begin{equation*}
\begin{aligned}
(\sigma_3)_{ \gamma_{3}}
=
 \<cycle1> \quad \text{and} \quad
\widehat{\mC}_{2} = \big(2\big)\,.
\end{aligned}
\end{equation*}
Overall, for the above example, the permutation extraction map $
\Pi_{ \tau}$ yields
\begin{equation}\label{eq_supp_exmpl1}
\begin{aligned}
\gamma \mapsto
\big( 1\ 3 \ 4 \big) \big(2 \big) \in S_{4}\,.
\end{aligned}
\end{equation}
\end{example}

Some remarks on the permutation extraction map are due.
First, we note that the mapping is well defined. This is because once a minimal ${\rm v}$-cycle is 
to be extracted, the integers indexing its base points are removed from the permutation and the
new tree  has inner vertices indexed by the remaining integers.
A second observation is that it is not bijective.
To see that the map is not
injective, consider the tree from Example~\ref{exampl_cycle_extr},
however, now with the pairing
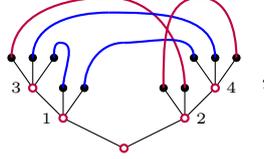
\begin{equation*}
\begin{aligned}
 \vcenter{\hbox{
               \begin{tikzpicture}[scale=0.4]
        \draw  (0,0)   -- (-.7,1) node[dot] {}  ;
        \draw  (0,0)   -- (0,1) node[dot] {}  ;
        \draw (0,0) -- (.7,1) node[dot] {};
        \draw (0,0) -- (1,-1);
        \draw (1,0) node[dot] {} -- (1,-1);
        \draw (1.7,0) node[dot] {} -- (1,-1);
        \draw  (6,0)   -- (5.3,1) node[dot] {}  ;
        \draw (6,0) -- (6.7,1) node[dot] {};
        \draw (6,0) -- (6,1) node[dot] {};
        \draw (6,0) -- (5,-1);
	\draw (5,-1) --(3,-2);
        \draw (5,-1)--(5,0)  node[dot] {}  ;
        \draw (5,-1)--(4.3,0)  node[dot] {}  ;
        \draw (1,-1) -- (3,-2);
        \draw[purple, thick] (-.7,1) to[out=90,in=180] (2.5,3);
        \draw[purple, thick] (2.5,3) to[out=0,in=90] (5,0);
        \draw[blue, thick] (0,1) to[out=90,in=180] (2.5,2.5);
        \draw[blue, thick] (2.5,2.5) to[out=0,in=90] (6,1);
        \draw[blue, thick] (0.7,1) to[out=90,in=180] (.9,1.5);
        \draw[blue, thick] (0.9,1.5) to[out=0,in=90] (1,0);
        \draw[blue, thick] (1.7,0) to[out=90,in=180] (3,1.5);
        \draw[blue, thick] (3,1.5) to[out=0,in=90] (5.3,1);
	\draw[purple, thick] (4.3,0) to[out=90,in=180] (5.5,3);
	\draw[purple, thick] (5.5,3) to[out=0,in=90] (6.7,1);
	\node (c) [label = left:{\tiny $3$}] at (0.3, 0) {};
	\node (a) [label = left:{\tiny $1$}] at (1.3, -1) {};
	\node (b) [label = right:{\tiny $2$}] at (4.7, -1) {};
	\node (b) [label = right:{\tiny $4$}] at (5.7, 0) {};
	\node[idot] at (0,0) {};
	\node[idot] at (1,-1) {};
	\node[idot] at (3,-2) {};
	\node[idot] at (5,-1) {};
	\node[idot] at (6,0) {};
        \end{tikzpicture}  }}\,,
\end{aligned}
\end{equation*}
where again we marked the minimal ${\rm v}$-cycles in blue.
Then $\Pi_{ \tau}$ maps the above pairing also to the permutation
\eqref{eq_supp_exmpl1}.
Moreover, the map is not surjective:
For example consider a paired tree $[ \tau, \tau]$ containing the component
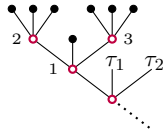
\begin{equation*}
\begin{aligned}
 \vcenter{\hbox{
        \begin{tikzpicture}[scale=0.4]
        \draw  (-.3,0)   -- (-1,1) node[dot] {}  ;
        \draw  (-.3,0)   -- (-.3,1) node[dot] {}  ;
        \draw (-.3,0) -- (.4,1) node[dot] {};
        \draw (-.3,0) -- (1,-1);
        \draw (1,0) node[dot] {} -- (1,-1);
        \draw (2.3,0) -- (1,-1);
	\draw  (2.3,0)   -- (1.6,1) node[dot] {}  ;
        \draw  (2.3,0)   -- (2.3,1) node[dot] {}  ;
        \draw (2.3,0) -- (3,1) node[dot] {};
        \draw[thick, dotted] (2.3,-2) -- (3.6,-3);
        \draw (1,-1) -- (2.3,-2);
	\draw (2.3,-2) -- (2.3,-1);
	\node at (2.4,.-0.8) {\scalebox{0.8}{$\tau_1$}};
        \draw (2.3,-2) -- (3.6,-1);
	\node at (3.7,.-0.8) {\scalebox{0.8}{$\tau_2$}};
	\node (c) [label = left:{\tiny $2$}] at (0, 0) {};
	\node (a) [label = left:{\tiny $1$}] at (1.2, -1) {};
	\node (b) [label = right:{\tiny $3$}] at (2, 0) {};
	\node[idot] at (-.3,0) {};
	\node[idot] at (1,-1) {};
	\node[idot] at (2.3,0) {};
	\node[idot] at (2.3,-2) {};
        \end{tikzpicture}  }}
\end{aligned}
\end{equation*}
and a permutation $ \uppi \in S_{n}$ containing the permutation cycle $ \big( 1\ 2\ 3\big)$.
Then there exists no pairing $ \gamma \in \mY ( \tau, \tau)$
such that $ \Pi_{\tau}( \gamma) = \uppi$,
 since it is not possible to construct
a ${\rm v}$-cycle according to this permutation cycle.
Indeed, let us try to construct a corresponding ${\rm v}$-cycle and see that this fails.
Necessarily, $\gamma$ would contain an edge contracting a leaf connected
to $v_{2}$ and $v_{3}$.
Moreover, we can only connect a single leaf neighbouring either $v_{2}$ or $
v_{3}$ to the isolated leave neighbouring $v_{1}$, say we choose a leaf at $v_2$.
Then the pairing $ \gamma$ would contain the following edges (in red):
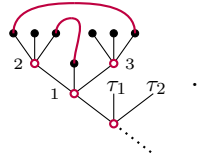
\begin{equation*}
\begin{aligned}
 \vcenter{\hbox{
        \begin{tikzpicture}[scale=0.4]
        \draw  (-.3,0)   -- (-1,1) node[dot] {}  ;
        \draw  (-.3,0)   -- (-.3,1) node[dot] {}  ;
        \draw (-.3,0) -- (.4,1) node[dot] {};
        \draw (-.3,0) -- (1,-1);
        \draw (1,0) node[dot] {} -- (1,-1);
        \draw (2.3,0) -- (1,-1);
	\draw  (2.3,0)   -- (1.6,1) node[dot] {}  ;
        \draw  (2.3,0)   -- (2.3,1) node[dot] {}  ;
        \draw (2.3,0) -- (3,1) node[dot] {};
        \draw[thick, dotted] (2.3,-2) -- (3.6,-3);
        \draw (1,-1) -- (2.3,-2);
	\draw (2.3,-2) -- (2.3,-1);
	\node at (2.4,.-0.8) {\scalebox{0.8}{$\tau_1$}};
        \draw (2.3,-2) -- (3.6,-1);
	\node at (3.7,.-0.8) {\scalebox{0.8}{$\tau_2$}};
        \draw[purple, thick] (-1,1) to[out=90,in=180] (1,2);
   	\draw[purple, thick] (1,2) to[out=0,in=90] (3,1);
	\draw[purple, thick] (0.4,1) to[out=90,in=180] (.9,1.5);
        \draw[purple, thick] (0.9,1.5) to[out=0,in=90] (1,0);
 	\node (c) [label = left:{\tiny $2$}] at (0, 0) {};
	\node (a) [label = left:{\tiny $1$}] at (1.2, -1) {};
	\node (b) [label = right:{\tiny $3$}] at (2, 0) {};
	\node[idot] at (-.3,0) {};
	\node[idot] at (1,-1) {};
	\node[idot] at (2.3,0) {};
	\node[idot] at (2.3,-2) {};
        \end{tikzpicture}  }}\,.
\end{aligned}
\end{equation*}
Now, it is not possible to close a ${\rm v}$-cycle with $v_3$, while also crossing
$v_{1}, v_{2}$. 
Note that this construction does not depend on the specific leaves we chose. In
particular, the roles of $v_{2}$ and $v_{3}$ can be reversed.\\

The main result of this section is an upper bound on the integral represented
by a paired tree.
This bound is obtained via the permutation cycle sequence extracted by the map $ \Pi_{\tau} $.
The upper bound will turn out to be sharp when $ \Pi_{\tau} (\gamma) =
\mathrm{Id} $, meaning that only cycles of length one are extracted.
 Before we state the result, let us introduce
the following notation for the time-simplex induced by a tree. For a tree $\tau
\in \mT$ we define
\begin{equation}\label{timesimplex}
\begin{aligned}
\mathfrak{D}_{ \tau}(t) :=
 \{s \in [0,t]^{\mI( \tau)\setminus \mf{o}} \,:\,
    \text{ if $\mathfrak{p}(u)=v$
, then $s_v\geq s_u$}
    \}\,,
\end{aligned}
\end{equation}
with the usual convention that $ s_{\mf{o}}=t$ and $\mf{p}(u)$ denoting the
parent of $u$.
\begin{lemma}\label{lem_unif_bound_contraction_perm}
Let $ \tau \in \mT_{3}$ and $i( \tau)$ be the number of internal nodes of $\tau$.
For $\uppi\in S_{2i(\tau)}$ with permutation cycle decomposition $\{\widehat{
\mC}_{i}\}_{i=1}^{K ( \uppi)}$
and  $\gamma\in \Pi^{-1}_\tau(\uppi )$,
we have
\begin{align*}
    [\tau,\tau]_{\gamma,\ve}(t,x)
    & \leqslant
    \lambda_\varepsilon^2\,\Big(
    \int_{\mathfrak{D}_{[\tau,\tau]}(t)}
    \Psi_{\uppi,\varepsilon}(s_{\mI}) \ud s_{\mI}\Big)
	e^{2 \mathfrak{m} \, t}
    p_{2(t+\varepsilon^2)}(0)\,,\quad & \forall \pi \in S_{n} \;,\\
    [\tau,\tau]_{\gamma,\ve}(t,x) 
    & =
    \lambda_\varepsilon^2\,\Big(
    \int_{\mathfrak{D}_{[\tau,\tau]}(t)}
    \Psi_{\uppi,\varepsilon}(s_{\mI}) \ud s_{\mI}\Big)
	e^{2 \mathfrak{m} \, t}
    p_{2(t+\varepsilon^2)}(0)\,, \quad & \text{ if } \pi = \mathrm{Id} \;,
\end{align*}
where $\mI:=\mI([\tau,\tau])\setminus \mf{o}$ and
\begin{equation}\label{e:defPsi}
\begin{aligned}
 \Psi_{\uppi,\varepsilon}(s)
    :=
    \prod_{i=1}^{K ( \uppi)} \<2cycle>^{\otimes |\widehat{\mC}_i|}(s_v; v\in
\mI_{\widehat{\mC}_i})  \,,
\end{aligned}
\end{equation}
with $ \<2cycle>^{\otimes n}$ defined in \eqref{e_def_vident}.
Here $\mI_{ \widehat{\mC}}$ denotes all those inner vertices whose labels lie
in the permutation cycle $ \widehat{\mC}$.
\end{lemma}
\begin{proof}
Let $\uppi \in S_{2i{(\tau)}}$ and $\gamma \in \Pi_\tau^{-1}(\uppi)$.
The proof works by extracting cycles successively from $[ \tau , \tau]$ and
using Lemma~\ref{lem:cnclmcycl} to
obtain a bound on $[ \tau , \tau]_{ \gamma, \ve}$ in terms of these cycles.
We write $\mV = \mV ([ \tau, \tau])\setminus \mf{o}$ and $\mI=\mI([\tau,\tau])
\setminus \mf{o}$.
Starting with the extraction of the minimal ${\rm v}$-cycle
$\mC_1$ in  $[ \tau , \tau]_{ \gamma}$, we have
\begin{align}\label{e:lem_extract_supp1}
&{[\tau,\tau]}_{\gamma, \varepsilon}(t,x)\nonumber \\
    &\leqslant
	\int_{D_{t}^{\mV \setminus \mC_1}}
   	K_{([\tau, \tau]\setminus \mC_1)_{\tilde\gamma}}^{t,x}
	(s_{\mV \setminus \mC_1}, y_{\mV  \setminus \mC_1})\\
    	& \qquad
    \int_{[0,t]^{\mI_{\mC_1}}}   \<2cycle>^{\otimes |\mI_{ \mC_1}|}(s_{v}; v\in \mI_{\mC_1})
	\left\{
	\prod_{ v \in \mI_{\mC_1}} \mathds{1}_{ s_{\mf{d}_{ \tau,  \gamma}(v)}
\leqslant  s_v \leqslant s_{ \mathfrak{p}_{ \tau} (v)}}
	\right\}	\ud s_{\mI_{\mC_1}} \ud s_{\mV \setminus \mC_1} \ud
y_{\mV \setminus \mC_1}\,,\nonumber 
\end{align}
where $K_{([\tau, \tau]\setminus \mC_1)_{\tilde\gamma}}^{t,x}$ denotes the kernel
corresponding to the fully contracted tree $([\tau,
\tau]\setminus\mC_1)_{\tilde\gamma}= (\sigma_{1})_{ \tilde\gamma} $, following
the notation in Definition \ref{def:cycl-rmvl}, and $ \mf{d}_{\tau,
\gamma}(v) $ indicates the
unique descendant of $ v \in \mI_{\mC_{1}} $ not in $ \mC_{1} $ (see also Lemma~\ref{lem:cnclmcycl}).
We can then proceed iteratively by extracting the ${\rm v}$-cycles via the cycle extraction
map from Definition \ref{def:cycle-extraction}, until we reach the tree
${\scalebox{0.8}{\<cycle1>}}$. In this way we obtain the upper bound, using
Lemma~\ref{lem:cnclmcycl},
\begin{align}\label{e_supp1extrub}
&[\tau,\tau]_{\gamma, \varepsilon}(t,x) \\
&\leqslant  \<cycle1>_\varepsilon(t,x) \int_{[0,t]^{\mI}}
 \prod_{i =1}^{K( \tau, \gamma)} \left\{ \<2cycle>^{\otimes |\mI_{ \mC_i}|}(s_{v}; v\in \mI_{\mC_i})
	\left\{ \prod_{ v \in \mI_{\mC_i}} \mathds{1}_{ 
s_{\mf{d}_{\sigma_{i}, \gamma}
(v)} \leqslant  s_v \leqslant s_{ \mathfrak{p}_{ \sigma_{i}}(v)}}   \right\}
\right\} \ud s_{\mI}\,,\nonumber
\end{align}
with  the sequence of ${\rm v}$-cycles $ ( \mC_{k} )_{k =1}^{K (\tau, \gamma)} $
and the sequence of reduced trees $
(\sigma_{i})_{i=1}^{K ( \tau, \gamma)}$  from the
cycle extraction map.
Note that $K ( \tau, \gamma)$ equals the number of permutation cycles $ K (\uppi)$.
In order to avoid confusion, we added a subindex $\mf{p}_{ \sigma_{i}}$ to
the parent map $\mf{p}$ (and also to the descendant map $\mf{d}_{\gamma}$), making clear with respect to which tree the map is to be
interpreted.
In Lemma~\ref{lem_2timesimplex}, we will see that the time-integral over the
indicator functions in \eqref{e_supp1extrub} preserves the original ordering
imposed by the tree. More precisely, it is independent of the chosen
pairing $ \gamma$ and equals the integral over the tree simplex
$\mathfrak{D}_{[ \tau, \tau]}(t)$.
Thus, by application of Lemma~\ref{lem_2timesimplex}, the inequality \eqref{e_supp1extrub} reads
\begin{align*}
[\tau,\tau]_{\gamma, \varepsilon}(t,x) & \leqslant  \<cycle1>_\varepsilon(t,x) \int_{\mathfrak{D}_{[ \tau, \tau]}(t)}
 \prod_{i =1}^{K( \uppi)}  \<2cycle>^{\otimes |\mI_{ \mC_i}|}(s_{v}; v\in \mI_{\mC_i}) \ud s_{\mI}\\
 &= \lambda_\varepsilon^2 e^{2 \mathfrak{m} \, t}
    p_{2(t+\varepsilon^2)}(0) \,\Big(
    \int_{\mathfrak{D}_{[ \tau, \tau]}(t)}
    \Psi_{\uppi,\varepsilon}(s_{\mI}) \ud s_{\mI}\Big)  \,,
\end{align*}
which yields the desired upper bound.
If $ \uppi = \mathrm{Id}$, then the inequality in the first line becomes an
equality as we are successively removing
1-cycles and  apply the identity in
Lemma~\ref{lem:cncl1cycl}, rather than the upper bound in Lemma~\ref{lem:cnclmcycl}.
The proof is complete.
\end{proof}

\section{Contributing and non-contributing trees and their structure}\label{sec:estimates-single-tree}

This section is dedicated to the proof of Proposition~\ref{prop_single_tree}.
Obtaining this result requires a precise quantitative control over the limiting
behavior of contracted and paired trees. Such control will build on a systematic
application of the bounds and ideas that we have introduced in
Section~\ref{sec:exmpl}, and in particular it will build
on Lemma~\ref{lem_unif_bound_contraction_perm} above.

In the previous section we analysed \emph{paired} trees and found that it was possible to identify
${\rm v}$-cycles and remove them iteratively to obtain an upper bound (or an exact
estimate, in the case when all
$\rm{v}$-cycles are $ 1$-cycles) on the integral associated to such a tree.
In this section we start instead with an arbitrary \emph{contracted} ternary tree $
[\tau]_{\kappa} $. Our
objective is to obtain a bound (or an exact estimate) on the second moment of
the Wiener integral associated to such contracted tree.
To obtain such estimate, we must sum over all possible pairings $ \gamma $ of
$ [\tau, \tau] $ which complete the contraction $ \kappa $, and for each such
pairing we can follow the procedure described in the previous section. One of
the key points of this section is therefore to keep track of all the
combinatorial factors that appear when counting pairings and contractions
associated to arbitrarily large trees.

This section will be split into two parts. First, we
study those contracted trees that do not vanish in the limit $
\ve \to 0 $ (i.e.\ that contribute) and identify them using properties of the
underlying graph $ \tau_{ \kappa}$. We also determine their precise limiting
contribution and the size of the set of all such contractions.
In the second part, we will instead state and prove a
uniform upper bound for the rate of
convergence of contracted trees that vanish as $ \ve \to 0 $ (i.e.\ that do not
contribute).

\subsection{Contributing contractions}
We start by studying those contracted trees that contribute to the fluctuations
in the limit $ \ve \to 0
$, and for which an exact estimate of the contribution is necessary. For this reason, let
us define \emph{contributing contractions} as follows.

\begin{definition}\label{def:contr_tree}
For any tree $\tau\in \mT_3$ and contraction $\kappa \in \mK(\tau) $, we say that $ \kappa $
\emph{contributes} if there exist $(t,x) \in (0, \infty) \times \RR^{2}$ such
that
\begin{equation}\label{e:def_contrCont}
\begin{aligned}
\limsup_{\varepsilon\to 0} \; (\log \tfrac{1}{\ve}) \cdot
\EE\Big[\big|[\tau]_{\ve, \kappa} (t, x)\big|^2\Big] >0\,.
\end{aligned}
\end{equation}
We denote the set of all contributing contractions by
\begin{align*}
    C(\tau):= \{\kappa \text{ such that } \tau_\kappa \text{ contributes}\}
\subseteq \mK (\tau)\,.
\end{align*}
\end{definition}

\subsubsection{Identifying contributing contractions}

Before determining the precise contribution of contracted trees, we show that the
abstract condition in \eqref{e:def_contrCont} can be replaced by a condition
on the underlying graph structure of the contracted tree.
More precisely, we will see that the estimates of Section~\ref{sec_rem_vcycle}
imply that contracted trees contribute if and only if the
corresponding integrals lie in the first homogeneous Wiener chaos and we can
iteratively remove 1-cycles from them.

\begin{lemma}\label{lem:characContr}
Let $ \tau \in \mT_{3}$, then
\begin{enumerate}[label=\textup{(\roman*)}]
\item $ C ( \tau) = \{ \kappa \in
\mK( \tau)\,:\,  \exists  \gamma \in
\mY( \tau_{\kappa}, \tau_{ \kappa}) \text{ with }  \Pi_{ \tau}( \gamma) =
\mathrm{Id}\}$.

\item If $ \kappa \in C( \tau) $, then $ \tau_{\kappa}$ has a single uncontracted
leaf. That is,  $ [ \tau]_{\kappa}$ lies in the first Wiener chaos
and  is therefore Gaussian.
\end{enumerate}

\end{lemma}

We postpone the proof of the lemma to the end of this section.
The property of being contributing, associated to a contracted tree $
[\tau]_{\kappa} $, is defined using the second moment condition
\eqref{e:def_contrCont}. We can express \eqref{e:def_contrCont} by summing $
[\tau, \tau]_{\gamma} $ over
all pairings in $ \gamma \in \mY( \tau_{\kappa},
\tau_{\kappa})$, recall \eqref{e_cov_contr_tree}.
The paired tree obtained, once we fix an element of $ \mY( \tau_{\kappa},
\tau_{\kappa})$, can be treated via Lemma~\ref{lem_unif_bound_contraction_perm}.
In particular, it turns out that for contributing trees $ \tau_{ \kappa}$ and $ \gamma \in\mY( \tau_{\kappa},
\tau_{\kappa}) $ the cycle extraction map satisfies $ \Pi_{\tau} (\gamma) =
\mathrm{Id} $.
As a consequence, we can determine precisely the limiting behavior of such 
paired trees through the last statement of
Lemma~\ref{lem_unif_bound_contraction_perm}, which is the key
ingredient in the proof of Lemma~\ref{lem:characContr}. This is the content of the
following lemma.

\begin{lemma}\label{lem:contrIdent}
Let $ \tau \in \mT_{3}$,  $ \red{ \tau} = \mathscr{T} ( \tau)$ be the
trimmed tree, as in \eqref{trimming}, and $ \uppi \in S_{2i(\tau)}$. Then for every  $\gamma
\in \Pi_{ \tau}^{-1}(\uppi )$ and all $ (t,x) \in (0, \infty) \times \RR^{2}$
\begin{equation*}
\begin{aligned}
\lim_{\ve\to 0} \,\,(\log \tfrac{1}{\ve}) \cdot [ \tau, \tau]_{\gamma, \ve} (t,x) =
\begin{cases}
\hat{\lambda}^2 \, \left\{ \frac{1}{ \red{ \tau}!}  \left(
\frac{ \hat{\lambda}^{2}}{2 \pi}
\right)^{|\red{ \tau}|}\right\}^2 \, p_{2t}^{( \mathfrak{m})} (0) \quad & \text{if } \uppi = \mathrm{Id}\,, \\
0\quad & \text{otherwise} \,.
\end{cases}
\end{aligned}
\end{equation*}
If the above limit vanishes, we say $ \gamma$ is a \emph{non-contributing
pairing}, and call it a \emph{contributing pairing} otherwise.
\end{lemma}

The proof of Lemma~\ref{lem:contrIdent} uses the following
identity.
\begin{lemma}\label{lem:onecycleint}
Let $ \tau \in \mT_{3}$ and  $ \red{ \tau} = \mathscr{T} ( \tau)$ be the
trimmed tree as in \eqref{trimming}, then
\begin{equation}\label{e:defRedTauIntegral}
\begin{aligned}
 \int_{\mathfrak{D}_{ [\tau]}(t)} \prod_{v \in \mI ( \tau) }
\frac{\lambda_\varepsilon^2 e^{ 2{\mathfrak{m}} \, s_{v}}}{4 \pi(s_{v}+ \ve^2)}
  \ud s_{\mI(\tau)}
=
 \frac{1}{\red{\tau}!}
 \left(\frac{\lambda_\varepsilon^2}{4\pi}
\int_{0}^{t} \frac{e^{2 \mathfrak{m} \, s}}{s+ \ve^{2}}  \ud s
    \right)^{|\red{\tau}|}=: \red{[ \tau]_{\ve}}(t) \; ,
\end{aligned}
\end{equation}
with the tree-time-simplex $\mathfrak{D}_{ [\tau]}(t)$ introduced in \eqref{timesimplex}.
\end{lemma}
This lemma is a consequence of Lemma~\ref{symmetricintegral}, since the integrand on
the left-hand side in \eqref{e:defRedTauIntegral} is a symmetric function
over the variables $s_{\mI(\tau)}= s_{\mV(\red{\tau})}$.

\begin{proof}[Proof of Lemma~\ref{lem:contrIdent}]
Consider $ \tau \in \mT_{3}$, $ \uppi \in S_{2i(\tau)} $ and $ \gamma \in
\Pi_{\tau}^{-1}(\uppi)$. Then for all $(t,x) \in (0,\infty) \times \RR^2$,
Lemma~\ref{lem_unif_bound_contraction_perm} implies
\begin{align*}
   (\log \tfrac{1}{\ve}) \cdot  [\tau,\tau]_{\gamma,\ve}(t,x)
    \leqslant
    \hat{\lambda}^2\,\left\{
    \int_{\mathfrak{D}_{ [\tau,\tau]}(t)}
    \Psi_{\uppi,\varepsilon}(s) \ud s_{\mI} \right\}
	e^{2 \mathfrak{m} \, t}
    p_{2(t+\varepsilon^2)}(0)
\,,
\end{align*}
where we remind that $\mI:=\mI([\tau,\tau])\setminus \mf{o}$.
Extending the domain of integration from $\mathfrak{D}_{ [\tau,\tau]}(t)$ to
$[0,t]^{\mI}$, the right-hand side can be factorised
\begin{equation*}
\begin{aligned}
(\log \tfrac{1}{\ve}) \cdot  [\tau,\tau]_{\gamma,\ve}(t,x)
     \leqslant
	 \hat{\lambda}^{2}
\left\{
\prod_{i= 1}^{K ( \tau, \gamma)}
    \int_{[0,t]^{ \mI_{\mC_i}}}
     \<2cycle>^{\otimes | \mI_{\mC_i}|}(s_{v}; v\in \mI_{\mC_i})  \ud
s_{ \mI_{\mC_i}} \right\} e^{2 \mathfrak{m} \, t}
 p_{2(t+\varepsilon^2)}(0)	\,,
\end{aligned}
\end{equation*}
where $({\mC}_i )_{i=1}^{K ( \tau, \gamma)}$ denotes the
sequence of ${\rm v}$-cycles  constructed from $[\tau,\tau]_\gamma$ via the cycle extraction map (Definition
\ref{def:cycle-extraction}).
For each of the integrals we have
\begin{equation*}
\begin{aligned}
\int_{[0,t]^{ \mI_{{\mC}_i}}}
     \<2cycle>^{\otimes | \mI_{{\mC}_i}|}(s_{v};
v\in  \mI_{{\mC}_i})  \ud s_{ \mI_{{\mC}_i}}
\,\,\begin{cases}
\,\,\,=  \,\,\, \frac{\lambda_{\ve}^{2}}{4 \pi }  \int_{0}^{t}
\frac{ e^{2 \mf{m} \, s}}{s + \ve^{2}} \ud s  \quad &
 \,\, \text{if } \,\,\, | \mI_{{\mC}_i}|=1\,,\\
& \\
\,\,\, \leqslant \,\,\, \,\frac{(\lambda_\varepsilon e^{
\overline{\mathfrak{m}}\, t} )^{2| \mI_{{\mC}_i}|}}{2^{| \mI_{{\mC}_i}|}\, \pi}
   \log \big(
   1+\tfrac{t}{\varepsilon^2}
   \big)  &  \,\, \text{if } \,\,\, | \mI_{{\mC}_i}| \geqslant 2\,,
\end{cases}
\end{aligned}
\end{equation*}
where for the case $| \mI_{\mC_i}| \geqslant 2$ we used
Lemma~\ref{lem_v_cycle_order_bnd_new}. Note that the right-hand side in the
second case vanishes in the limit $ \ve \to 0 $, since $ \lambda_{\ve} \sim
(\log{\tfrac{1}{\ve}})^{- 1/2}$.
In particular, if $ \uppi \neq \mathrm{ Id}$ then at least one ${\rm v}$-cycle ${\mC}_i$ must
satisfy $ |\mI_{\mC_i}| \geqslant 2$, which yields that $[ \tau, \tau]_{
\gamma, \ve}(t,x)$ vanishes as $ \ve \to 0 $.

On the other hand, if $ \uppi = \mathrm{Id}$ all the ${\rm v}$-cycles $ \mC_i$
are 1-cycles and we can replace all inequalities with identities to obtain
\begin{equation*}
\begin{aligned}
 (\log \tfrac{1}{\ve}) \cdot  [\tau,\tau]_{\gamma,\ve}(t,x)
    &=
    \hat{\lambda}^2\, \left\{ \int_{\mathfrak{D}_{ [ \tau, \tau]}(t)}
	\prod_{ v \in \mI }  \<2cycle>^{\otimes 1}(s_{v}) \ud s_{\mI} \right\}
	e^{2 \mathfrak{m} \, t}
    p_{2(t+\varepsilon^2)}(0) \\
    &= \hat{\lambda}^2\, \left\{ \int_{\mathfrak{D}_{ [\tau]}(t)} \prod_{v \in \mI ( \tau) }
\frac{\lambda_\ve^2 e^{2 \mathfrak{m} \, s_{v}} }{4 \pi(s_{v}+ \ve^2)}  \ud s_{\mI(\tau)} \right\}^2
e^{2 \mathfrak{m} \, t}
p_{2(t+\varepsilon^2)}(0) \;.
\end{aligned}
\end{equation*}
Hence, we deduce from Lemma~\ref{lem:onecycleint} and Lemma~\ref{l_app_expInt} that
\begin{equation*}
\begin{aligned}
 \lim_{\ve\to 0} \,\, (\log \tfrac{1}{\ve}) \cdot  [\tau,\tau]_{\gamma,\ve}(t,x)
= \hat{\lambda}^2\, \left\{ \frac{1}{ \red{ \tau}!}
 \left( \frac{ \hat{\lambda}^{2}}{2 \pi}  \right)^{|\red{ \tau}|}\right\}^2
 p_{2t}^{(\mathfrak{m})} (0) \,,
\end{aligned}
\end{equation*}
which concludes the proof.
\end{proof}

Finally, we are ready to show that contributing contractions can be identified
as those that have a single uncontracted leaf and allow for an iterative removal of
1-cycles.

\begin{proof}[Proof of Lemma~\ref{lem:characContr}]
Let us start by recalling from \eqref{e_cov_contr_tree} that for any
tree $ \tau \in \mT_{3}$ and any contraction $ \kappa \in \mK(\tau) $
\begin{equation*}
\begin{aligned}
  (\log \tfrac{1}{\ve}) \cdot
\,\EE\left[[\tau]_{\kappa,\ve}^2 (t, x)\right]
= \sum_{\gamma \in \mY (\tau_{\kappa}, \tau_{\kappa})}
  (\log \tfrac{1}{\ve}) \cdot
 [ \tau,\tau]_{\gamma, \ve} (t, x) \,.
\end{aligned}
\end{equation*}
Therefore, to prove the first statement of the lemma, it suffices to prove that
$$ \lim_{\ve \to 0}(\log\tfrac{1}{\ve}) \cdot
 [ \tau,\tau]_{\gamma, \ve}(t, x) >0\,,$$ for all $ (t, x) \in (0, \infty) \times
\RR^{2} $ if and only if $ \Pi_{\tau} (\gamma) =
\mathrm{Id} $, which is implied by Lemma~\ref{lem:contrIdent}.

For the second statement we instead proceed by induction over the number
of inner nodes $ i( \tau)$.
We can check that the statement is true for $ i( \tau)=0$, i.e.  $
\tau= \<0>$, since $| \mY ( \tau, \tau)| = 1$ and
\begin{equation*}
\begin{aligned}
\limsup_{\varepsilon\to 0} \
  (\log \tfrac{1}{\ve})  \cdot
\,\EE\big[\big\vert \<10>_{\, \ve} \big\vert^2 (t, x) \big]
>0\,.
\end{aligned}
\end{equation*}
Now, let $m \in \NN $ and assume that the statement holds for all $ \tau' \in \mT_{3} $ satisfying
 $ i ( \tau') \leqslant m$.
Choose $ \tau \in \mT_{ 3}$ with $ i ( \tau) = m+1$ and let $ \kappa \in C(
\tau) $. By the first point of the present Lemma~\ref{lem:characContr}, which we have
just proven, we know there exists a $ \gamma \in \mY
( \tau_{\kappa}, \tau_{\kappa}) \cap \Pi_{ \tau}^{-1}( \mathrm{Id}) $.
In particular, let $ \mC_{1} $ be the first 1-cycle that is extracted by the
permutation-extraction map applied to the pairing $ \gamma $, and write
$ (\sigma_{2})_{ \gamma_{2}}= ([\tau,
\tau]\setminus \mC_1)_{ \gamma_{2}} = [ \hat{\tau}, \tau]_{\gamma_{2}}$, with
$ \hat{\tau}:= \tau \setminus \mC_{1}$, assuming without loss of generality
that we have removed the cycle from the left tree.

Then via \eqref{e:lem_extract_supp1}, we deduce that
\begin{equation}\label{e:lem_extract_supp1_1}
\begin{aligned}
[\tau,\tau]_{\gamma, \varepsilon}(t,x)
    &\leqslant 
\int_{D_{t}^{\mV\setminus \mC_{1}}}
   	K_{([\tau, \tau]\setminus \mC_1)_{\gamma_{2}}}^{t,x} (s_{\mV \setminus \mC_1}, y_{\mV  \setminus \mC_1})\\
    	& \int_{0}^{t}
    \<2cycle>^{\otimes 1}(s_{v}; v\in \mI_{\mC_1})
	\left\{
	\prod_{ v \in \mI_{\mC_1}} \mathds{1}_{  s_{\mf{d}_\gamma (v)} \leqslant  s_v \leqslant s_{ \mathfrak{p}(v)}}
	\right\}	\ud s_{\mI_{\mC_1}} \ud s_{\mV \setminus \mC_1} \ud
y_{\mV \setminus \mC_1}\,,
\end{aligned}
\end{equation}
where we again used $\mV = \mV([ \tau, \tau]) \setminus \mf{o}$. Note that the
product in the expression above only consists of a single term, because $
\mC_{1}$ is a 1-cycle.
Dropping the time-constraint encoded by $ \mathds{1}_{ \mf{d}_{ \gamma}( v)
\leqslant  s_v \leqslant s_{ \mathfrak{p}(v)}}
$, we therefore obtain
\begin{equation}\label{e:lem_extract_supp1_2}
\begin{aligned}
[\tau,\tau]_{\gamma, \varepsilon}(t,x)
    &\leqslant [ \hat{\tau}, \tau]_{ \gamma_{2}, \ve} (t,x)
\left\{
	 \int_0^t
    \<2cycle>^{\otimes 1}(s_{v}; v\in \mI_{\mC_1})
	\ud s_{\mI_{\mC_1}} \right\}\,.
\end{aligned}
\end{equation}
Now, taking the $\limsup$ over $ \ve \to 0 $ after multiplying both sides with
$ (\log{ \tfrac{1}{\ve}})$ yields
\begin{equation}\label{e:lem_extract_supp1_3}
\begin{aligned}
0<
\limsup_{ \ve \to 0}
 \  (\log \tfrac{1}{\ve})
\,[\tau,\tau]_{\gamma, \varepsilon}(t,x)
\leqslant \frac{ \hat{\lambda}^{2}}{2 \pi}
\limsup_{ \ve \to 0}
 \  (\log \tfrac{1}{\ve})
\, [ \hat{\tau}, \tau]_{\gamma_{2}, \ve}(t, x)  \,.
\end{aligned}
\end{equation}
Here, the first inequality holds since $ \kappa \in C(\tau)$, while the second inequality is a
consequence of \eqref{e:lem_extract_supp1_2} and
Lemma~\ref{l_app_expInt}.
In particular, the limit on the right-hand side must be positive.
Next, by Definition~\ref{def:PairingTocontraction} there exists $(
\kappa_{1}( \gamma_{2}), \kappa_{2}( \gamma_{2}))$ such that $ \gamma_{2} \in \mY(
\hat{\tau}_{\kappa_{1}( \gamma_{2})}, \tau_{\kappa_{2}(\gamma_{2})})$, with
$\kappa_2 ( \gamma_{2}) = \kappa$, so via an application of the
Cauchy--Schwartz inequality we obtain
\begin{equation}\label{e:CSfromPair2Cont}
\begin{aligned}
 [ \hat{\tau}, \tau]_{\gamma_{2}, \ve}
\leqslant 
\EE \left[ [\hat{\tau}]_{\kappa_{1}(\gamma_{2}),\ve}\, [\tau]_{\kappa,
\ve}\right]
\leqslant
\left( \EE \left[ [\hat{\tau}]_{\kappa_{1}(\gamma_{2}),\ve}^2\right]
 \right)^{ \frac{1}{2}}
\left( \EE \left[  [\tau]_{\kappa,\ve}^2\right]
 \right)^{ \frac{1}{2}} \,,
\end{aligned}
\end{equation}
which together with \eqref{e:lem_extract_supp1_3} implies $ \kappa_{1}(
\gamma_{2}) \in C(
\hat{\tau})$. By the induction assumption, $
\hat{\tau}_{\kappa_1( \gamma_2)}$ (note that $
i(\hat{\tau})=m$) has a single
uncontracted leaf, which implies that also $ \tau_{\kappa} $ has a single
uncontracted leaf. This concludes the proof.
\end{proof}

Note that Lemma~\ref{lem:characContr}, together with the identity from
Lemma~\ref{lem:contrIdent}, implies that if $ \kappa \in C (\tau) $ then $[\tau]_{\ve, \kappa}$ is (and
converges after rescaling to) a mean-zero
Gaussian, with limiting fluctuations
\begin{equation*}
\begin{aligned}
\lim_{\ve \to 0}\
\sqrt{\log{ \tfrac{1}{\ve}}} \cdot
\big\|[ \tau]_{\ve, \kappa}(t,x)\big\|_{L^{2}(\PP)}
=
 \frac{ \hat{\lambda}}{ \red{ \tau}!}  \left(
\frac{ \hat{\lambda}^{2}}{2 \pi}
\right)^{|\red{ \tau}|}
\sqrt{p^{(\mf{m})}_{2 t} (0) }
\,.
\end{aligned}
\end{equation*}
In the next subsection, we will see that a stronger statement holds true,
as we will be able to identify $[ \tau]_{\kappa, \ve }$ with
$\<10>_{\,\ve}$ up to a multiplicative factor.

\subsubsection{Determining contributions}

In the previous section, we identified
contributing pairings (and contractions) to be the ones mapped by $ \Pi_{
\tau}$ to the identity permutation, i.e.\
the algorithm defined in Definition~\ref{def:cycle-extraction} only extracts cycles of length one.
Precisely this fact will turn out to be useful, when determining the following
identity for contributing contracted trees.

\begin{lemma}\label{lem:contr}
For every $ \tau \in \mT_{3}, \kappa \in {C}( \tau) $, $ \ve \in
(0,\tfrac{1}{2})$ and $(t,x) \in
(0, \infty) \times \RR^{2}$ we have that
\begin{equation*}
\begin{aligned}
[ \tau]_{  \kappa, \ve}(t,x)=  \frac{1}{\red{\tau}!} \left(\frac{\lambda_\varepsilon^2}{4\pi}
\int_{0}^{t} \frac{e^{2 \mathfrak{m} \, s}}{s+ \ve^{2}} \ud s
    \right)^{|\red{\tau}|} \, \<10>_{ \, \ve}(t,x)
=
\red{[ \tau]_{\ve}}(t)  \, \<10>_{\,  \ve}(t,x)
\,.
\end{aligned}
\end{equation*}
\end{lemma}

Note that the right-hand side of the identity in the lemma above does not
depend on the particular contraction $ \kappa \in C(\tau)$.

\begin{proof}
Fix $ \tau \in \mT_{3}$ and let $ \kappa \in {C}( \tau) $.
Choose $ \gamma \in \Pi_{ \tau}^{-1}(\mathrm{Id}) \cap \mY (
\tau_{ \kappa}, \tau_{ \kappa} )$, which exists by
Lemma~\ref{lem:characContr}(i). Since $ \gamma \in \Pi_{\tau}^{-1}
(\mathrm{Id}) $, the cycle extraction map, see Definition~\ref{def:cycle-extraction},
associates to $ \gamma $ a sequence of 1-cycles $(\mC_i)_{i=1}^{2 i (\tau)}$.
In order to distinguish between the two trees generating $[ \tau, \tau]$, let
us write $[ \tau_{1}, \tau_{2}]:= [\tau, \tau]$.
Now, let $( \mC_{i} ')_{i=1}^{ i ( \tau)}$ be the
subset of cycles whose bases belong to the tree $ \tau_{1} $. In other words, $( \mC_{i}
')_{i=1}^{ i ( \tau)}$ contains all cycles in  $(\mC_i)_{i=1}^{2 i
(\tau)}$ such that $\mI_{\mC_{i}}\subset \mI ( \tau_{1})$.
This yields an iterative rule of removing 1-cycles from $ [\tau_{1}]_{
\kappa}=[ \tau]_{\kappa}$.
In particular, each $ {\mC}_{i}$ corresponds to a unique inner
node $ v_i$ of $ \tau$.

Recall that for $v \in \mI (\tau)$ we write $ \mathfrak{p}(v) \in \mI ([\tau])$ for the
parent of $v$ and recall also the representation of the stochastic integrals in \eqref{e:tk}, which allows
us to write
\begin{equation*}
\begin{aligned}
[ \tau]_{ \kappa, \ve}(t,x)
&=
  \int_{D_{t}^{\mV(\tau)}}
K_{[\tau_{\kappa}]}^{t, x} (
s_{\mV},y_{\mV})\;  \ud s_{\mV ( \tau)} \, \ud y_{\mV(\tau)\setminus
\ell}\  \eta_\ve( \ud
y_{\ell})  \;,
\end{aligned}
\end{equation*}
where we denote by $(s_{ \ell} ,y_{ \ell})$ the space-time point associated to
the single uncontracted leaf in $ \tau_{ \kappa} $, cf. Lemma~\ref{lem:characContr}(ii).
We apply Lemma~\ref{lem:cncl1cycl} with respect to the 1-cycle $ \mC^{\prime}_{1} $,
which yields
\begin{equation*}
\begin{aligned}
[ \tau]_{ \kappa , \ve } (t,x)
&=
 \int_{D_{t}^{\mV(\tau) \setminus \mC_{1}'}}
\left\{
 \int_{0}^{t} \<2cycle>^{\otimes 1}(s_{v_{1}}) \mathds{1}_{\{ 
s_{\mf{d}_\kappa (v_{1})} \leqslant  s_{v_1} \leqslant s_{
\mathfrak{p}(v_1)} \}}  \ud s_{v_{1}} \right\}
\\
&\qquad\qquad
K_{[ \widetilde{\tau}_{ \widetilde{\kappa}}]} (s_{\mV ( \tau)
\setminus \mC_1' },y_{\mV( \tau) \setminus \mC_1'})    \ud s_{\mV ( \tau) \setminus \mC_{1}'} \ud y_{\mV(\tau)\setminus (
\ell \cup \mC_{1}')}\ \eta_\ve( \ud
y_{\ell}) \;.
  \end{aligned}
\end{equation*}
Now, by applying  Lemma~\ref{lem:cncl1cycl} successively another $ i (\tau)-1 $
times with respect to each of the
1-cycles $( \mC_{i} ')_{i=2}^{ i ( \tau)}$, we obtain
\begin{equation*}
\begin{aligned}
[ \tau]_{ \kappa, \ve} (t,x)
&=
\left\{
 \int_{[0,t]^{\mI ( \tau)} } \prod_{v \in \mI ( \tau)} \<2cycle>^{\otimes
1}(s_{v}) \mathds{1}_{\{  s_{\mf{d}_{\kappa} (v)} \leqslant  s_{v} \leqslant s_{
\mathfrak{p}(v)} \}}  \ud
s_{\mI( \tau)}
\right\}
\int_{\RR^{2}} p_t^{(\mf{m})}(y_{\ell}-x)\  \eta_\ve( \ud
y_{\ell}) \;.
\end{aligned}
\end{equation*}
The stochastic integral on the right-hand side equals $\<10>_{\, \ve}$, whereas
the time integral in the brackets can be rewritten as
\begin{equation*}
\begin{aligned}
 \int_{[0,t]^{\mI( \tau)} } \prod_{ v \in \mI ( \tau)}
\<2cycle>^{\otimes 1}(s_{v}) \mathds{1}_{\{  s_{\mf{d}_\kappa (v)} \leqslant  s_{v} \leqslant s_{
\mathfrak{p}(v)} \}}  \ud
s_{\mI( \tau)}
=
\int_{\mathfrak{D}_{ [\tau]}(t)}  \prod_{v \in \mI ( \tau)} \frac{
\lambda^{2}_{ \ve}\, e^{2 \mf{m}\, s_{v}}}{4 \pi(s_{v}+ \ve^{2}) }   \ud s_{
\mI ( \tau)}\;,
\end{aligned}
\end{equation*}
where we used that
\begin{equation*}
\begin{aligned}
\prod_{v \in \mI(\tau)} \mathds{1}_{\{  s_{\mf{d}_\kappa (v)} \leqslant  s_{v} \leqslant s_{
\mathfrak{p}(v)} \}} = \prod_{v \in \mI(\tau)} \mathds{1}_{\{ s_{v} \leqslant s_{
\mathfrak{p}(v)} \}}\;.
\end{aligned}
\end{equation*}
Together with Lemma~\ref{lem:onecycleint}, this concludes the
proof.
\end{proof}

\subsubsection{Counting contributing contractions}

In the previous section, we saw that the limit of a contributing contracted
tree is independent of the precise structure of the contraction. Thus, in order
to conclude the limit of $ X^{ \tau}$, it is only left to determine the size of
$C( \tau)$.

\begin{lemma}\label{lem:count_contr}
Let $ \tau \in \mT_{3}$, then $|{C}( \tau)| = 3^{ i( \tau)}$.
\end{lemma}

In order to prove Lemma~\ref{lem:count_contr}, we first need the following
result.

\begin{lemma}\label{lem:toptridents_contractions}
Let $ \tau \in \mT_{3}$
and $ \kappa \in C( \tau)$, then every trident in $ \tau_\kappa$ has an internal contraction.
More precisely, for every $ v \in \mV_{\<3s>}( \tau) $, with
\begin{equation*}
\begin{aligned}
 \mV_{\<3s>}( \tau) := \{ v \in \mI ( \tau) \,:\, \text{there exist exactly three $ u_{1}, u_{2}, u_{3} \in \mL ( \tau)$ such
that $\mathfrak{p}(u_{i})=v$} \}\,,
\end{aligned}
\end{equation*}
we have $\{u_{i},
u_{j}\} \in \kappa$ for two distinct $i,j \in \{ 1,2,3\}$.
\end{lemma}
\begin{proof}
Let $ \tau \in \mT_{3}$, $ \kappa \in C( \tau)$ and consider  $[ \tau, \tau]_{\gamma} $ for the unique $
\gamma \in \mY ( \tau_{\kappa}, \tau_{\kappa})$, see
Lemma~\ref{lem:characContr}(ii).
For any $v \in  \mV_{\<3s>}( \tau)$, we write $ u_{1}(v), u_{2}(v),
u_{3}(v) \in \mL ( \tau)$ for the three leaves it is connected to (indexing
them with $1$ to $3$ from left to right).

Now, assume there exists a $v \in  \mV_{\<3s>}( \tau)$
without an internal contraction, i.e. $\{u_{i}(v), u_{j}(v)\} \notin \kappa$
for all $1 \leqslant  i,j \leqslant 3$. Graphically, this can be represented as
follows:
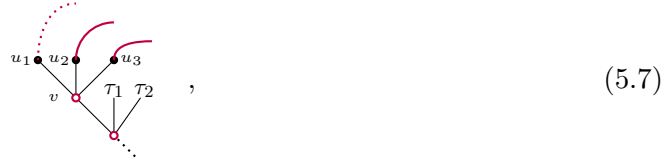
\begin{equation}\label{e:topTridentsupp1}
\begin{aligned}
	\vcenter{\hbox{
    \begin{tikzpicture}[scale=0.5]
        \draw (0,0) node[dot] {} -- (1,-1);
        \draw (1,0) node[dot] {} -- (1,-1);
        \draw (2,0) node[dot] {} -- (1,-1);
        \draw[thick, dotted] (2.1,-2.1) -- (2.6,-2.6);
         \draw (1,-1) -- (2,-2);  \draw (2,-2) -- (2,-1);  \node at (2,.-0.8) {\scalebox{0.8}{$\tau_1$}};
         \draw (2,-2) -- (2.7,-1);  \node at (2.8,.-0.8) {\scalebox{0.8}{$\tau_2$}};
        \draw[purple, thick] (1,0) to[out=90,in=180] (2,1);
        \draw[purple, thick] (2,0) to[out=90,in=180] (3,0.5);
        \draw[purple, thick, dotted] (0,0) to[out=90,in=180] (1,1.5);
	\node (c) [label = left:{\tiny $v$}] at (1.1, -1) {};
	\node (c) [label = left:{\tiny $u_{1}$}] at (0.4, 0) {};
	\node (c) [label = left:{\tiny $u_{2}$}] at (1.4, 0) {};
	\node (c) [label = left:{\tiny $u_{3}$}] at (3.3, 0) {};
	\node[idot] at (2,-2) {};
	\node[idot] at (1,-1) {};
        \end{tikzpicture}  }}\,,
\end{aligned}
\end{equation}
for some $ \tau_{1}, \tau_{2} \in \mT_{3}$ (note that possibly one of the leaves
could be uncontracted, which is indicated in the example above by the dotted
red line).
Then it is immediate to see that $ \Pi_{\tau}( \gamma) \neq
\mathrm{Id}$, since otherwise $\{u_{i}(v), u_{j}(v)\} \in  \kappa \subset
\gamma$ for some $1 \leqslant i,j \leqslant 3$.
Thus, contradicting the assumption $ \kappa \in C( \tau)$ by
Lemma~\ref{lem:characContr}(i).
This concludes the proof.
\end{proof}

\begin{proof}[Proof of Lemma~\ref{lem:count_contr}]
We prove the statement by induction over the number of inner nodes $ i (
\tau)$, starting with $ i( \tau) =0$, i.e. $ \tau = \<0>$. In this case,
$|C( \tau) | =|\mK ( \tau)|= |\mY ( \tau, \tau)|=1$ and the claim holds.
Now assume the statement holds true for any $ \widehat{\tau} \in
\mT_{3}$ satisfying $ i ( \widehat{\tau}) \leqslant m$.

Let $ \tau \in \mT_{3}$ with $ i( \tau) =m+1$ and fix any $ v \in \mV_{\<3s>}(
\tau)$. We denote its neighbouring leaves by  $ u_{1}(v), u_{2}(v),
u_{3}(v) \in \mL ( \tau)$ (indexing
them with $1$ to $3$ from left to right):
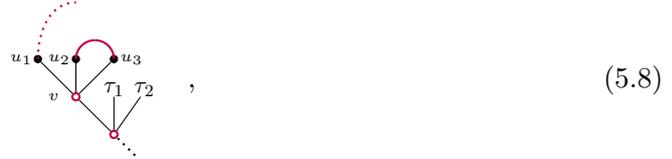
\begin{equation}\label{e:Countsupp1}
\begin{aligned}
	\vcenter{\hbox{
    \begin{tikzpicture}[scale=0.5]
        \draw (0,0) node[dot] {} -- (1,-1);
        \draw (1,0) node[dot] {} -- (1,-1);
        \draw (2,0) node[dot] {} -- (1,-1);
        \draw[thick, dotted] (2.1,-2.1) -- (2.6,-2.6);
         \draw (1,-1) -- (2,-2);  \draw (2,-2) -- (2,-1);  \node at (2,.-0.8) {\scalebox{0.8}{$\tau_1$}};
         \draw (2,-2) -- (2.7,-1);  \node at (2.8,.-0.8) {\scalebox{0.8}{$\tau_2$}};
        \draw[purple, thick] (1,0) to[out=90,in=180] (1.5,0.5);
        \draw[purple, thick] (1.5,0.5) to[out=0,in=90] (2,0);
        \draw[purple, thick, dotted] (0,0) to[out=90,in=180] (1,1.5);
	\node (c) [label = left:{\tiny $v$}] at (1.1, -1) {};
	\node (c) [label = left:{\tiny $u_{1}$}] at (0.4, 0) {};
	\node (c) [label = left:{\tiny $u_{2}$}] at (1.4, 0) {};
	\node (c) [label = left:{\tiny $u_{3}$}] at (3.3, 0) {};
	\node[idot] at (2,-2) {};
	\node[idot] at (1,-1) {};
        \end{tikzpicture}  }}\,,
\end{aligned}
\end{equation}
for some $ \tau_{1}, \tau_{2} \in \mT_{3}$. Again note that possibly one of the leaves
could be uncontracted.
Moreover, using Lemma~\ref{lem:toptridents_contractions}, we can partition $C( \tau)$ into
three sets, $C_{1,2}( \tau), C_{1,3}( \tau), C_{2,3}( \tau)$, with
\begin{equation*}
\begin{aligned}
C_{i,j}( \tau):= \left\{ \kappa \in C ( \tau) \,:\, \{u_{i}(v), u_{j}(v) \} \in
\kappa\right\}\,.
\end{aligned}
\end{equation*}
For any contraction $ \kappa \in C( \tau)$ we define the tree resulting from
$ \tau_{\kappa}$ after removing the 1-cycle $\mC$ with $\mI_{\mC}= \{v\}$:
\begin{equation}\label{e:countContr_supp1}
\begin{aligned}
\widehat{\tau}_{ \tilde{\kappa}}:= (\tau \setminus
\mC)_{\tilde{ \kappa}}\,,
\end{aligned}
\end{equation}
using the cycle removal from Definition~\ref{def:cycl-rmvl}.
Expression \eqref{e:countContr_supp1} defines a map $\mathfrak{K}_{v} : C( \tau) \to \mK( \widehat{
\tau})$ with $\mathfrak{K}_{v} ( \kappa):= \Tilde{ \kappa}$.
Moreover, we have that $ \widetilde{\kappa} $ is contributing for $
\hat{\tau} $ (namely $\Tilde{ \kappa}=\mathfrak{K}_{v} ( \kappa) \in C (
\hat{\tau})$).

In fact, for any choice $1 \leqslant i< j \leqslant  3 $, the map
$\mathfrak{K}_{v}|_{C_{i,j}( \tau)}$ maps onto $C( \widehat{\tau} )$ and defines a bijection.
To see this, consider an arbitrary contraction $ \widehat{ \kappa} \in C (
\widehat{ \tau})$ (with the labeling of $ \widehat{ \tau}$ induced by $
\tau$) and define $ \kappa:= \widehat{ \kappa} \cup \{u_{i}(v), u_{j}(v)\}$.
For example we have the following reconstruction of a contraction
in $C_{2,3}( \tau)$ using the inverse $(\mathfrak{K}_{v}|_{C_{2,3}(
\tau)})^{-1}$:
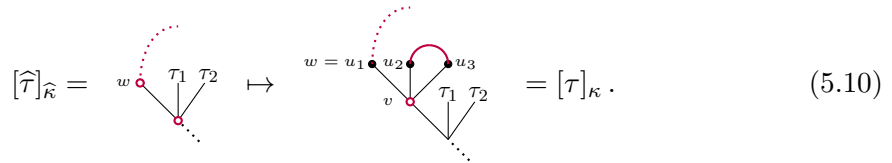
\begin{equation}\label{e:Countsupp2}
\begin{aligned}
[ \widehat{ \tau}]_{\widehat{ \kappa}}
=
\vcenter{\hbox{
    \begin{tikzpicture}[scale=0.5]
        \draw[thick, dotted] (2.1,-2.1) -- (2.6,-2.6);
         \draw (1,-1) node[dot] {} -- (2,-2);  \draw (2,-2) -- (2,-1);  \node at (2,.-0.8) {\scalebox{0.8}{$\tau_1$}};
         \draw (2,-2) -- (2.7,-1);  \node at (2.8,.-0.8) {\scalebox{0.8}{$\tau_2$}};
        \draw[purple, thick, dotted] (1,-1) to[out=90,in=180] (2,0.5);
	\node (c) [label = left:{\tiny $w$}] at (1.3, -1) {};
	\node[idot] at (2,-2) {};
	\node[idot] at (1,-1) {};
        \end{tikzpicture}  }}
\mapsto
	\vcenter{\hbox{
    \begin{tikzpicture}[scale=0.5]
        \draw (0,0) node[dot] {} -- (1,-1);
        \draw (1,0) node[dot] {} -- (1,-1);
        \draw (2,0) node[dot] {} -- (1,-1);
        \draw[thick, dotted] (2.1,-2.1) -- (2.6,-2.6);
         \draw (1,-1) -- (2,-2);  \draw (2,-2) -- (2,-1);  \node at (2,.-0.8) {\scalebox{0.8}{$\tau_1$}};
         \draw (2,-2) -- (2.7,-1);  \node at (2.8,.-0.8) {\scalebox{0.8}{$\tau_2$}};
        \draw[purple, thick] (1,0) to[out=90,in=180] (1.5,0.5);
        \draw[purple, thick] (1.5,0.5) to[out=0,in=90] (2,0);
        \draw[purple, thick, dotted] (0,0) to[out=90,in=180] (1,1.5);
	\node (c) [label = left:{\tiny $v$}] at (1.1, -1) {};
	\node (c) [label = left:{\tiny $w=u_{1}$}] at (0.4, 0) {};
	\node (c) [label = left:{\tiny $u_{2}$}] at (1.4, 0) {};
	\node (c) [label = left:{\tiny $u_{3}$}] at (3.3, 0) {};
	\node[idot] at (1,-1) {};
        \end{tikzpicture}  }}
=
[ \tau]_{\kappa}
\,.
\end{aligned}
\end{equation}
On the other hand, for the same $ \widehat{ \kappa}$ we can also
reconstruct the following two contractions in $C_{1,2}( \tau)$ and $C_{1,3}(
\tau)$ , respectively:
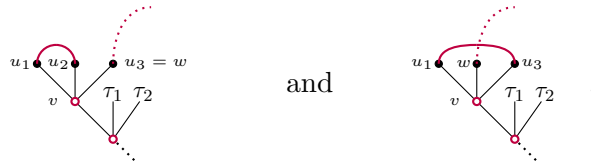
\begin{equation*}
\begin{aligned}
\vcenter{\hbox{
    \begin{tikzpicture}[scale=0.5]
        \draw (0,0) node[dot] {} -- (1,-1);
        \draw (1,0) node[dot] {} -- (1,-1);
        \draw (2,0) node[dot] {} -- (1,-1);
        \draw[thick, dotted] (2.1,-2.1) -- (2.6,-2.6);
         \draw (1,-1) -- (2,-2);  \draw (2,-2) -- (2,-1);  \node at (2,.-0.8) {\scalebox{0.8}{$\tau_1$}};
         \draw (2,-2) -- (2.7,-1);  \node at (2.8,.-0.8) {\scalebox{0.8}{$\tau_2$}};
        \draw[purple, thick] (0,0) to[out=90,in=180] (0.5,0.5);
        \draw[purple, thick] (0.5,0.5) to[out=0,in=90] (1,0);
        \draw[purple, thick, dotted] (2,0) to[out=90,in=180] (3,1.5);
	\node (c) [label = left:{\tiny $v$}] at (1.1, -1) {};
	\node (c) [label = left:{\tiny $u_{1}$}] at (0.4, 0) {};
	\node (c) [label = left:{\tiny $u_{2}$}] at (1.4, 0) {};
	\node (c) [label = left:{\tiny $u_{3}=w$}] at (4.5, 0) {};
	\node[idot] at (2,-2) {};
	\node[idot] at (1,-1) {};
        \end{tikzpicture}  }}
\qquad \text{and} \qquad
\vcenter{\hbox{
    \begin{tikzpicture}[scale=0.5]
        \draw (0,0) node[dot] {} -- (1,-1);
        \draw (1,0) node[dot] {} -- (1,-1);
        \draw (2,0) node[dot] {} -- (1,-1);
        \draw[thick, dotted] (2.1,-2.1) -- (2.6,-2.6);
         \draw (1,-1) -- (2,-2);  \draw (2,-2) -- (2,-1);  \node at (2,.-0.8) {\scalebox{0.8}{$\tau_1$}};
         \draw (2,-2) -- (2.7,-1);  \node at (2.8,.-0.8) {\scalebox{0.8}{$\tau_2$}};
        \draw[purple, thick] (0,0) to[out=90,in=180] (1,0.5);
        \draw[purple, thick] (1,0.5) to[out=0,in=90] (2,0);
        \draw[purple, thick, dotted] (1,0) to[out=90,in=180] (2,1.5);
	\node (c) [label = left:{\tiny $v$}] at (1.1, -1) {};
	\node (c) [label = left:{\tiny $u_{1}$}] at (0.4, 0) {};
	\node (c) [label = left:{\tiny $w$}] at (1.4, 0) {};
	\node (c) [label = left:{\tiny $u_{3}$}] at (3.3, 0) {};
	\node[idot] at (2,-2) {};
	\node[idot] at (1,-1) {};
        \end{tikzpicture}  }}\,.
\end{aligned}
\end{equation*}
In particular, for each set $C_{i,j}( \tau)$ there exists a unique $ \kappa \in C_{i,j}( \tau)$
such that $\mathfrak{K}_{v}(\kappa) = \widehat{\kappa}$.
As a consequence all three sets $C_{i,j}( \tau)$ have the same cardinality,
which agrees with $|C ( \widehat{ \tau})|$. Lastly, applying the induction
hypothesis to  $|C ( \widehat{ \tau})|$, yields
\begin{equation*}
\begin{aligned}
|C ( \tau )| = |C_{1,2}( \tau)|+ |C_{1,3}( \tau)|+ |C_{2,3}( \tau)| =  3 |C( \widehat{\tau})| = 3^{m+1}\,.
\end{aligned}
\end{equation*}
This concludes the proof.
\end{proof}

\subsection{Non-contributing trees}

Up to now, we have identified contributing pairings (and contractions)
to be the
ones that lie in the pre-image $
\Pi_{\tau}^{-1}(\mathrm{Id})$, when
considering a fixed tree $ \tau \in \mT_{3}$ . Moreover,
we determined their exact contribution.
Now, it is only left to control the overall contribution of the remaining
contractions, which we will prove to be negligible, in a strong
summable fashion.
We summarise the main findings of this section in the following lemma.

\begin{lemma}\label{lem_unif_estimate_all_noncontr_cont_new}
Let $T>0$, then
uniformly over any $ \ve \in (0, \tfrac{1}{T}
\wedge \tfrac{1}{2}) $,  $ \tau \in \mT_{3}^{N_{\ve}} $, for $ N_{\ve} = \lfloor
\log{\tfrac{1}{\ve}}\rfloor $, $ x \in \RR^{2}$ and uniformly for all
$ t \in [0,
 T] $, we have
\begin{align*}
    \Big\|
    \sum_{\kappa\notin C(\tau)}\sqrt{\log{\tfrac{1}{\ve}}}\cdot
    [\tau]_{\kappa,\ve}
(t, x)
    \Big\|_{L^2(\PP)}
    \leqslant
    \frac{1}{\sqrt{ 4\log \tfrac{1}{\ve}}}
    \frac{1}{ \red{\tau}!}
    \left(\frac{6e^{2 + 2\pi}\Hat{\lambda}^{2} e^{2 \overline{\mathfrak{m}}\, t} }{\pi}\right)^{|\red{\tau}|} \frac{ \hat{\lambda}
e^{ \mathfrak{m}\, t}}{\sqrt{4 (t + \ve^{2})}} \,,
\end{align*}
where $ \red{ \tau}$
denotes the trimmed tree $\mathscr{T} ( \tau)$ as in \eqref{trimming}.
In particular, for a fixed $\tau \in \mT_3$ the right-hand side vanishes in the small-$\varepsilon$ limit.
\end{lemma}

\begin{remark}\label{rem_correlation}
Our methods in Section~\ref{sec_rem_vcycle} (such as the cycle extraction and the
corresponding estimates in Lemma~\ref{lem_unif_bound_contraction_perm}) also
apply to covariances of the form 
\begin{equation*}
\begin{aligned}
\left(\log{\tfrac{1}{\ve}}\right) \cdot
 \sum_{\kappa , \kappa ' \notin C(\tau)}
\EE \big[
    [\tau]_{\kappa,\ve}(t,x) [\tau]_{\kappa,\ve}(t',x') 
\big]\,,
\end{aligned}
\end{equation*}
instead of just second moments as considered in
Lemma~\ref{lem_unif_estimate_all_noncontr_cont_new}. 
For this we don't identify the roots of the trees $ [ \tau]_{ \kappa}$ and
$ [ \tau]_{ \kappa'}$, when pairing the two trees, but keep them separate with
individual time-space points associated.
Thus, instead of stopping the cycle extraction
(Definition~\ref{def:cycle-extraction}) once we see
$\<cycle1>$, we terminate the algorithm once $\<10pair10>$ appears.
For this reason, we would see a $\sqrt{ \pi e^{2\mf{m}\, t} p_{2(t+
\ve^{2})}(x-x')}$
instead of 
\begin{equation*}
\begin{aligned}
\frac{ 
e^{ \mathfrak{m}\, t}}{\sqrt{4 (t + \ve^{2})}}
=
\sqrt{ \pi
e^{2\mf{m}\, t} p_{2(t+ \ve^{2})}(0)}\,,
\end{aligned}
\end{equation*}
on the right-hand side of Lemma~\ref{lem_unif_estimate_all_noncontr_cont_new}.
In particular, we expect this to allow for treatment of the statistics of the corresponding
field associated to \eqref{e:acn} in the H\"older space $\mC^{+1
-}(\RR^{2})$ for fixed $t>0$, which we leave for future work.
\end{remark}

For the proof of Lemma~\ref{lem_unif_estimate_all_noncontr_cont_new} we need
the following two lemmas. The first is an upper bound of a ``symmetrised''
integral over {\rm v}-cycles.

\begin{lemma}\label{lem:unifBoundOverPreimage}
Let $T>0$.
Then uniformly over any $ \ve \in (0, \tfrac{1}{T}
\wedge \tfrac{1}{2}) $, $ t \in [0,
 T] $ and  $ \tau \in \mT_{3}^{N_{\ve}} $, $ N_{\ve} = \lfloor
\log{\tfrac{1}{\ve}}\rfloor $, we have
\begin{align*}
    \sum_{\uppi\in S_{2 i ( \tau)}\setminus\{\mathrm{Id}\}}
    \Big(
    \int_{\mathfrak{D}_{ [\tau, \tau]}(t)}
    \Psi_{\uppi,\varepsilon}(s_{\mI}) \ud s_{ \mI}
    \Big)
    \leqslant
    \frac{1}{\log \tfrac{1}{\ve}}
    \frac{1}{(\red{\tau}!)^2}
    \frac{(\Hat{\lambda} e^{ \overline{\mathfrak{m}}\, t} )^{4 i ( \tau)}}{4 \pi^{2 i ( \tau)-1}}
     e^{2 (2+2 \pi)\, i ( \tau)}
    \,,
\end{align*}
where $\mI := \mI([ \tau, \tau]) \setminus \mf{o}$,
and  $ \red{ \tau} = \mathscr{T} ( \tau)$
denotes the trimmed tree \eqref{trimming}.
Recall \eqref{e:defPsi} for the definition of the function $
\Psi_{\uppi,\varepsilon}$.
\end{lemma}

The next result guarantees  that the number of pairings $ \gamma $ of a
tree $ [\tau, \tau] $, which correspond to
a permutation $ \uppi \in S_{2 i (\tau)}$,
grows at most exponentially in the number of inner nodes of a tree.

\begin{lemma}\label{lem:nonContrPreImage}
Let $ \tau \in \mT_{3}$ and $ \uppi \in S_{2 i ( \tau)}$, then
$\left| \Pi_{ \tau}^{-1} ( \uppi ) \right| \leqslant 6^{2 i (\tau)}$.
\end{lemma}

Having both Lemma~\ref{lem:unifBoundOverPreimage}
and~\ref{lem:nonContrPreImage} at hand, we can now prove
Lemma~\ref{lem_unif_estimate_all_noncontr_cont_new}.
The proofs of Lemma~\ref{lem:unifBoundOverPreimage} and
\ref{lem:nonContrPreImage} are  deferred to the end of this
section.

\begin{proof}[Proof of Lemma~\ref{lem_unif_estimate_all_noncontr_cont_new}]
Consider $T>0$, $ \ve \in (0, \tfrac{1}{T}
\wedge \tfrac{1}{2})$ and $(t,x) \in [0, T]\times \RR^{2}$.
By representing second moments of contracted trees in terms
of paired trees, cf. \eqref{e_cov_contr_tree}, we have 
\begin{equation*}
\begin{aligned}
\left\|
\sum_{\kappa\notin C(\tau)}\sqrt{\log{\tfrac{1}{\ve}}}
  \cdot  [\tau]_{\kappa,\ve}(t,x)\right\|^{2}_{L^2(\PP)}
&=
\log{\tfrac{1}{\ve}}
\sum_{\substack{\kappa, \kappa' \notin C(\tau)}}
\sum_{ \gamma \in
\mY( \tau_{ \kappa}, \tau_{ \kappa'})}
[\tau, \tau]_{ \gamma, \ve}(t,x)\\
&\leqslant 
 \log{\tfrac{1}{\ve}}
\sum_{\uppi\in S_{2 i( \tau)}\setminus\{\mathrm{Id}\}}
    \sum_{\gamma \in \Pi_\tau^{-1}(\uppi)}
    [\tau,\tau]_{\gamma,\varepsilon}(t,x)\,,
\end{aligned}
\end{equation*}
where we used additionally Lemma~\ref{lem:contrIdent} to identify
non-contributing pairings as
precisely the ones that do not map onto $\mathrm{Id}$ under $\Pi_{ \tau}
$\footnote{Because on the right-hand side we sum over all non-contributing
pairings, which may consist of one contributing contraction and one contraction
that doesn't contribute, we overestimate the left-hand side, which is the
reason for the inequality.},  and
the fact that
\begin{equation*}
\begin{aligned}
\exists  \gamma \in \mY (
\tau_{ \kappa}, \tau_{ \kappa'}) \quad
\limsup_{ \ve \to 0}\, \left( \log{ \tfrac{1}{\ve}} \right)
[ \tau , \tau]_{ \gamma, \ve}(t,x) >0  \quad \Leftrightarrow \quad
\kappa, \kappa' \in C ( \tau)\,,
\end{aligned}
\end{equation*}
which is a consequence of the Cauchy--Schwartz inequality, cf.
\eqref{e:CSfromPair2Cont}, and Lemma~\ref{lem:characContr}.
Thus, Lemmas~\ref{lem_unif_bound_contraction_perm} and
\ref{lem:nonContrPreImage} imply
\begin{equation*}
\begin{aligned}
\Bigg\|
\sum_{\kappa\notin C(\tau)}\sqrt{\log{\tfrac{1}{\ve}}}
  & \cdot  [\tau]_{\kappa,\ve}(t,x)\Bigg\|^{2}_{L^2(\PP)} \\
& \leqslant
\left(\log{\tfrac{1}{\ve}}\right)\,
 \lambda_\varepsilon^2\, 6^{2 i ( \tau)}
\sum_{\uppi\in S_{2 i ( \tau)}\setminus\{\mathrm{Id}\}}
   \Big(
    \int_{\mathfrak{D}_{ [ \tau,\tau]}(t)}
    \Psi_{\uppi,\varepsilon}(s_{\mI}) \ud s_{\mI} \Big)
\, e^{2 \mathfrak{m}\, t}  p_{2(t+\varepsilon^2)}(0)
\,.
\end{aligned}
\end{equation*}
Applying Lemma~\ref{lem:unifBoundOverPreimage}, this can be further upper bounded by
\begin{equation*}
\begin{aligned}
 \lambda_\varepsilon^2\, 6^{2 i ( \tau)}    \frac{1}{(\red{\tau}!)^2}
    \frac{(\Hat{\lambda} e^{ \overline{\mathfrak{m}}\, t} )^{4 i ( \tau)}}{4
\pi^{2 i ( \tau)-1}} & e^{ 2(2+2 \pi)\,  i ( \tau) }\, e^{2
\mathfrak{m}\, t}  p_{2(t+\varepsilon^2)}(0) \\
& = \frac{1}{4\, \log{ \tfrac{1}{\ve}}}  \frac{1}{(\red{\tau}!)^2}
\left( \frac{ 6 e^{2 + 2 \pi }
\hat{\lambda}^{2} e^{ 2\overline{\mathfrak{m}}\, t} }{ \pi} \right)^{2 i ( \tau)} \frac{ \hat{\lambda}^{2}e^{2 \mathfrak{m}\, t} }{4(t+ \ve^{2})}\,.
\end{aligned}
\end{equation*}
Since $i ( \tau)=  |\red{\tau}|$, the result follows
by taking the square root.
\end{proof}

Now we pass to the proof of Lemma~\ref{lem:unifBoundOverPreimage}. Note that
this is an improvement of Lemma~\ref{lem:contrIdent}.
Indeed, instead of extending the integration domain from $\mathfrak{D}_{ [\tau,
\tau]}(t)$ to the box
$[0,t]^{\mI([ \tau, \tau])\setminus \mf{o} }$, as it was done in the proof
of the latter lemma, we will make use of the fact that summation over all permutations
in $\uppi\in S_{n}\setminus\{\mathrm{Id}\}$ has a symmetrising effect that
allows for a more precise control of the integral.

\begin{lemma}\label{lem_phibar_symmetric}
Fix $\ve \in (0,\tfrac{1}{2})$ and $n \in \NN$. Consider
for any $ \uppi \in S_{n} $ the function $ \Psi_{\uppi,\ve}$
introduced in \eqref{e:defPsi}.
Then the function $\overline{\Psi}_\varepsilon : [0, \infty)^{n} \to \RR$ defined by
\begin{align*}
    \overline{\Psi}_\ve(s_{1}, \ldots s_{n})
    :=
    \sum_{\uppi\in S_{n}\setminus\{\mathrm{Id}\}}
    \Psi_{\uppi,\ve}(s_{1}, \ldots , s_{n})\,,
\end{align*}
is symmetric in the variables $ s_{1}, \cdots, s_{n} $.
\end{lemma}
\begin{proof}
It suffices to consider the function $ \overline{\varphi}_{ \ve}$ given by
\begin{align*}
     (s_1,...,s_n) \mapsto \sum_{\uppi \in S_{n}}
    \prod_{i=1}^{n}
        \frac{1}{s_i+s_{\uppi(i)}+2\varepsilon^2}\,,
\end{align*}
since the term corresponding to the identity partition is symmetric itself
and we have 
$$\overline{\Psi}_\varepsilon = \left( \frac{ \lambda_{ \ve}^{2}}{2 \pi}
\right)^{n} e^{2\mf{m} \sum_{i =1}^{n} s_{i} } \left( \overline{\varphi}_{
\ve} - \prod_{i=1}^{n} \frac{1}{2 s_i+2\varepsilon^2}\right) \;.$$
Now, for $ {\upsigma}\in S_{n}$, if we indicate $ s_{\upsigma} =
(s_{\upsigma (i) })_{i=1}^{n} $, we have
\begin{align*}
    \overline{\varphi}_{ \ve}(s_\upsigma)
    &=
    \sum_{\uppi \in S_{n}}
    \prod_{i=1}^{n}
    \frac{1}{s_{\upsigma(i)}+s_{\upsigma(\uppi(i))}+2 \ve^{2}}
    =
    \sum_{\uppi \in S_{n}}
    \prod_{i=1}^{n}
    \frac{1}{s_{i}+s_{\upsigma\uppi\upsigma^{-1}(i)}+2 \ve^{2}}\\
    &=
    \sum_{\uppi \in S_{n}}
    \prod_{i=1}^{n}
    \frac{1}{s_{i}+s_{\uppi(i)}+2 \ve^{2}}
    =
    \overline{\varphi}_{ \ve}(s)\,,
\end{align*}
since for every $ \sigma \in S_{n} $ we have $\{ \upsigma \uppi \upsigma^{-1}\, :\, \uppi \in S_{n}\}= S_{n}$.
This concludes the proof.
\end{proof}

Now we are ready to prove Lemma~\ref{lem:unifBoundOverPreimage}.

\begin{proof}[Proof of Lemma~\ref{lem:unifBoundOverPreimage}]
Fix $T>0$, $t \in [0, T]$ and define $n:= 2i ( \tau)$.
Using Lemma~\ref{symmetricintegral} and the definition of $
\overline{\Psi}_{\ve} $ in Lemma~\ref{lem_phibar_symmetric}, we find that
\begin{equation}\label{e:unifBoundNonContr_supp1}
\begin{aligned}
\int_{\mathfrak{D}_{[ \tau, \tau]}(t)}   \overline{\Psi}_\varepsilon(s_{ \mI}) \ud s_{ \mI}
=
\frac{1}{(\red{\tau}!)^2}
    \sum_{\uppi\in S_{n}\setminus\{\mathrm{Id}\}}
    \int_{[0,t]^{\mI}}
    \Psi_{\uppi,\varepsilon}(s_{\mI}) \ud s_{\mI}\,,
\end{aligned}
\end{equation}
since $ \overline{\Psi}_\varepsilon$ is symmetric by
Lemma~\ref{lem_phibar_symmetric}.
Recall from \eqref{e:defPsi} that  $\Psi_{\uppi,\varepsilon}$ is a product over
cycles in the permutation $ \uppi$, which allows us to factorise the integral
\begin{equation*}
\begin{aligned}
\int_{[0,t]^{\mI}}
    \Psi_{\uppi,\varepsilon}(s_{\mI}) \ud s_{\mI}
=
 \prod_{i=1}^{K ( \uppi)}
\int_{[0,t]^{\widehat{\mC}_{i}}} \<2cycle>^{\otimes |\widehat{\mC}_i|}(s_{v};
v\in \widehat{\mC}_i)  \ud s_{ \widehat{\mC}_i}\,,
\end{aligned}
\end{equation*}
with $K ( \uppi)$ denoting the number of cycles in the permutation
$\uppi$.
Applying Lemma~\ref{lem_v_cycle_order_bnd_new} to each term in
the product yields
\begin{equation*}
\begin{aligned}
\int_{[0,t]^{\mI}}
    \Psi_{\uppi,\varepsilon}(s_{\mI}) \ud s_{\mI}
& \leqslant
 \prod_{i=1}^{K ( \uppi)}
     \left( \frac{(\lambda_\varepsilon e^{ \overline{\mathfrak{m}}\, t})^{2|\widehat{\mC}_i|}}{2^{|\widehat{\mC}_i|}\, \pi}
   \log \big(
   1+\tfrac{t}{\varepsilon^2}
   \big)
\right)\\
& =
\frac{(\lambda_\varepsilon e^{ \overline{\mathfrak{m}}\, t})^{2n}}{2^{n}}
   \Big(
   \frac{1}{\pi} \log \big(
   1+\tfrac{t}{\varepsilon^2}
   \big)
   \Big)^{K ( \uppi)}
\leqslant
\frac{(\lambda_\varepsilon e^{ \overline{\mathfrak{m}}\, t})^{2n}}{2^{n}}
   M_{ \ve}(t)^{K ( \uppi)}
\,,
\end{aligned}
\end{equation*}
where we introduced
\begin{equation*}
\begin{aligned}
 M_\varepsilon(t):=
\left\lceil
\frac{1}{\pi} \log \big(
   1+\tfrac{t}{\varepsilon^2}
   \big)
 \right\rceil\,.
\end{aligned}
\end{equation*}
Therefore, we obtain the following upper bound to
\eqref{e:unifBoundNonContr_supp1}:
\begin{equation}\label{e:unifBoundNonContr_supp2}
\begin{aligned}
\int_{\mathfrak{D}_{[ \tau, \tau]}(t)}   \overline{\Psi}_\varepsilon(s_{\mI}) \ud s_{\mI}
\leqslant
\frac{n!}{(\red{\tau}!)^2}
\frac{(\lambda_\varepsilon e^{ \overline{\mathfrak{m}}\, t})^{2n}}{2^{n}}
\left(
 \mathrm{E}_{S_{n}}\left[
    M_\varepsilon(t)^{K(\uppi) }
   \right]
- \frac{M_{ \ve}(t)^{n}}{n!}
\right)\,,
\end{aligned}
\end{equation}
with the expectation taken with respect to
the uniform distribution on $S_n$, which has probability mass function $
\tfrac{1}{n!}$.
Here we used the identity
\begin{equation*}
\begin{aligned}
\mathrm{E}_{S_{n}}\left[
    M_\varepsilon(t)^{K(\uppi) }
   \right]
=
\frac{M_{ \ve}(t)^{n}}{n!}
+
    \sum_{\uppi\in S_{n}\setminus\{\mathrm{Id}\}}
\frac{ M_{ \ve}(t)^{K (\uppi )}}{n!} \,.
\end{aligned}
\end{equation*}
Hence, we have reduced the problem to studying the generating function of a
discrete random variable, namely the total number of cycles in a uniformly at
random chosen permutation.
Its distribution is a well studied object and we have the explicit
identity
\begin{equation*}
\begin{aligned}
  \mathrm{E}_{S_{n}}\Big[
   M_\varepsilon(t)^{K(\uppi)}
   \Big]=
   \binom{n+M_\varepsilon(t)-1}{n}
\end{aligned}
\end{equation*}
at hand, see e.g. \cite[Equation (5.14)]{ueltschi2022universal}.
Thus, we can rewrite \eqref{e:unifBoundNonContr_supp2} as
\begin{equation}\label{e:unifBoundNonContr_supp3}
\begin{aligned}
\int_{\mathfrak{D}_{[ \tau, \tau]}(t)}   \overline{\Psi}_\varepsilon(s_{\mI}) \ud s_{\mI}
\leqslant
\frac{1}{(\red{\tau}!)^2}
\frac{(\hat{\lambda} e^{ \overline{\mathfrak{m}}\, t} )^{2n}}{(2\, \log{ \tfrac{1}{\ve}})^{n}}
\left(
\frac{(n+M_{ \ve}(t)-1)!}{(M_{ \ve}(t)-1)!} -
M_{ \ve}(t)^{n}
\right)\,.
\end{aligned}
\end{equation}
Now, we expand the difference on the right-hand side to obtain
\begin{align}\label{eq_supp0_noncontr_xx}
    \frac{(n+M_\varepsilon(t)-1)!}{(M_\varepsilon-1)!}- M_\varepsilon(t)^{n}
    &=
    \sum_{j=1}^{n-1}
    j\, M_\varepsilon(t)^j \prod_{k=j+1}^{n-1} (M_\varepsilon(t)+k)\,.
\end{align}
Note that for every $k =0, \ldots, n-1$
\begin{equation*}
\begin{aligned}
 \frac{M_\varepsilon(t)+k}{2\, \log \tfrac{1}{\ve}}
    &\leqslant
    \frac{ \tfrac{1}{\pi} \log \big(
   1+\tfrac{t}{\varepsilon^2}
   \big)+k+1}{2\, \log \tfrac{1}{\ve}}
   \leqslant
   \frac{1}{\pi}\Big(1+\frac{|\log (t+\varepsilon^2)|+\pi(k+1)}{2\, \log
\tfrac{1}{\ve}} \Big)\\
   &\leqslant
   \frac{1}{\pi} \exp \left( \frac{|\log (t+\varepsilon^2)|+\pi(k+1)}{2\,\log
\tfrac{1}{\ve}} \right)\,,
\end{aligned}
\end{equation*}
thus, together with \eqref{e:unifBoundNonContr_supp3} and
\eqref{eq_supp0_noncontr_xx}
\begin{equation}\label{e:unifBoundNonContr_supp4}
\begin{aligned}
& \int_{\mathfrak{D}_{[ \tau, \tau]}(t)}    \overline{\Psi}_\varepsilon(s_{\mI})
\ud s_{\mI}\\
& \quad \leqslant
\frac{1}{(\red{\tau}!)^2}
\frac{(\hat{\lambda} e^{ \overline{\mathfrak{m}}\, t} )^{2n}}{2\, \log{ \tfrac{1}{\ve}}}
 \sum_{j=1}^{n-1}
    j\, \frac{ M_\varepsilon(t)^j}{ (2\log{\tfrac{1}{\ve}})^{j}}  \prod_{k=j+1}^{n-1}
\frac{M_\varepsilon(t)+k}{ 2 \log{ \tfrac{1}{\ve}}} \\
& \quad \leqslant
\frac{1}{\log{ \tfrac{1}{\ve}}}
\frac{1}{(\red{\tau}!)^2}
    \frac{(\Hat{\lambda} e^{ \overline{\mathfrak{m}}\, t} )^{2n}}{2\,\pi^{n-1}}
    \sum_{j=1}^{n-1}
    j\,  \exp \left( (n-1)\frac{|\log (t+\varepsilon^2)|}{2\, \log \tfrac{1}{\ve}}
    +\pi \sum_{k=j}^{n-1}  \frac{ k+1}{2\, \log \tfrac{1}{\ve}} \right)\,.
\end{aligned}
\end{equation}
The terms in the exponent can be estimated uniformly over $ \ve$ and $t
\in [0, T]$.
For the first summand, we make use of \eqref{e:def_const_c}, which yields 
\begin{equation*}
\begin{aligned}
 (n-1)\frac{|\log (t+\varepsilon^2)|}{2\, \log \tfrac{1}{\ve}}
\leqslant 2n \,.
\end{aligned}
\end{equation*}
Moreover, since  $ n \leq 2N_{\ve} = 2 \lfloor
\log{\tfrac{1}{\ve}}\rfloor $
\begin{equation*}
\begin{aligned}
\sum_{k=j}^{n-1} \frac{k+1}{2\, \log{ \tfrac{1}{\ve}}}
\leqslant
\sum_{k=1}^n  \frac{k}{2\, \log{ \tfrac{1}{\ve}}}
=
\frac{n(n+1)}{4 \, \log{ \tfrac{1}{\ve}} } \leqslant n\,.
\end{aligned}
\end{equation*}
Combining the two estimates with \eqref{e:unifBoundNonContr_supp4} yields
\begin{equation*}
\begin{aligned}
\int_{\mathfrak{D}_{[ \tau, \tau]}(t)}   \overline{\Psi}_\varepsilon(s_{\mI}) \ud s_{\mI}
&\leqslant
\frac{1}{\log{ \tfrac{1}{\ve}}}
\frac{1}{(\red{\tau}!)^2}
    \frac{(\Hat{\lambda} e^{ \overline{\mathfrak{m}}\, t})^{2n}}{2\,\pi^{n-1}}
    \frac{n(n-1)}{2}  e^{( 2 +\pi) n }\\
&\leqslant
\frac{1}{\log{ \tfrac{1}{\ve}}}
\frac{1}{(\red{\tau}!)^2}
    \frac{(\Hat{\lambda} e^{ \overline{\mathfrak{m}}\, t} )^{2n}}{4\,\pi^{n-1}}
     e^{( 2  +2\pi) n }\,,
\end{aligned}
\end{equation*}
where we used $ n(n-1) \leqslant e^{2n} \leqslant e^{\pi n}$ in the last
step. This concludes the proof.
\end{proof}

\begin{proof}[Proof of Lemma~\ref{lem:nonContrPreImage}]

For convenience let us write $n= 2i ( \tau)$.
Our aim is to
obtain an upper bound on the total number of parings $ \gamma $ that give rise to a
given permutation $ \uppi \in S_{n} $, via the map $ \Pi_{\tau} (\gamma) $ from
Definition~\ref{def:permCycleMap}. It will be convenient to represent $ \uppi
$ as a product of permutation cycles $ \uppi = \prod_{i =1}^{K(\uppi)}
\widehat{\mC}_{i} $, for some $ K( \uppi) \in \{ 1, \cdots , n \} $. 
Here we slightly abuse  the notation
$ \widehat{\mC}_{i} $, which is already used in Definition~\ref{def:permCycleMap}.
Indeed, while the decomposition of a permutation into a product of cycles is
unique, the order in which these cycles are chosen is arbitrary. 
On the other hand, not every order of permutation cycles is admissible 
 in Definition~\ref{def:permCycleMap}, as for example  cycle $
\widehat{\mC}_{1} $ is necessarily a cycle among bases of tridents or cherries
(because it is the first cycle we extract from a paired tree). In our setting,
since we start from an arbitrary permutation $
\uppi $, we assume nothing further on $ \widehat{\mC}_{i} $ other than that they
are cycles that decompose $ \uppi $.  We now provide the desired upper bound via the
steps that follow.

\begin{enumerate}
    \item By Lemma~\ref{vcycleexists} we know that any pairing $ \gamma $ in $\Pi^{-1}_\tau(\uppi)$ must contain a ${\rm v}$-cycle
that alternates between leaves and inner nodes of $[\tau,\tau]$. This is because the first
${\rm v}$-cycle we extract must be a ${\rm v}$-cycle in the paired tree $ [\tau,
\tau]_{\gamma} $ (later ones belong instead to trees that are derived from $ [\tau,
\tau]_{\gamma} $ by extracting ${\rm v}$-cycles).
    Therefore, there must exist a cycle $ \widehat{\mC}_{i_{1}} $ that only runs
through inner vertices in $V_{\<2s>}\cup  \mV_{\<3s>}=:W^{(1)}$.
Here $V_{\<2s>}$ and $  \mV_{\<3s>}$ denote the set of cherries and tridents in
$[ \tau, \tau]$,
respectively, see Appendix~\ref{sec_v_cycle_exist} for their definition.

    Indeed, if no such cycle exists, then the given permutation cannot arise from any
paired tree $[\tau,\tau]_\gamma$ using the extraction algorithm $ \Pi_\tau$ and
the pre-image is the empty set, so that our upper bound holds true. See the discussion below
Example~\ref{exampl_cycle_extr} for an example of this kind.

    \item Since
$ \widehat{\mC}_{i_{1}} \subset V_{\<2s>}\cup \mV_{\<3s>}$, to construct a ${\rm v}$-cycle
corresponding to $ \widehat{\mC}_{i_{1}} $ we must choose for every vertex $v$
with label in $
\widehat{\mC}_{i_{1}} $ an (outgoing) leaf that connects to the next vertex of the
${\rm v}$-cycle and an (incoming) leaf that connects to the previous vertex of the
${\rm v}$-cycle. Here, which leaf is outgoing and which leaf is incoming matters,
leading to at most $ 2 ! {3 \choose 2} =6 $ choices for every vertex (for nodes in
$\mV_{\<3s>}$ we have six choices, for nodes
in $V_{\<2s>}$ only two).
Thus, there are at most $6^{| \widehat{\mC}_{i_{1}}|}$ ways to construct a ${\rm v}$-cycle
through  inner nodes labeled by $ \widehat\mC_{i_{1}}$.
\item 
Now proceed iteratively. For $j>1$, we define the set $W^{(j)}
\subset \mI ( [ \tau, \tau]) \setminus \mf{o}$ as follows:
An inner node $v $ lies in $W^{(j)}$, if and only if
\begin{itemize}
\item there exist at least two distinct
paths from $v$ to leaves in $\mL ( [ \tau, \tau])$, which only run through
descendants of $v$ (away from the root), such that the descendants have been previously extracted,
i.e.\ they lie in $ \bigcup_{k=1}^{j-1}   \mI_{\widehat\mC_{i_{k}}}$,
\item and $v$ has not been extracted previously, i.e. $ v \not\in  \bigcup_{k=1}^{j-1}   \mI_{\widehat\mC_{i_{k}}}$.
\end{itemize}
The set $W^{(j)}$ describes the nodes that became
``admissible'' after extracting the cycles $ \{ \widehat{\mC}_i \}_{i \in \{
i_{1}, \cdots, i_{j-1} \}}$, meaning that the inner nodes of the next cycle that
we extract must belong to $ W^{(j)} $.

    \item Proceed by choosing a cycle $ \widehat{\mC}_{i_{j}}$ with nodes in $W^{(j)}$
and counting all possible choices to create a ${\rm v}$-cycle with the corresponding
nodes as inner nodes.

    \item We are done once all cycles have been removed, or we cannot find a
cycle that runs through vertices in $W^{(j)}$. In the latter case, we again
found a permutation that cannot be obtained as image of the map $ \Pi_\tau$
(so the pre-image is empty and the bound trivially true).
\end{enumerate}
In this way, we count all possible pairings that lead to $ \uppi $. Overall,
either the pre-image is empty and the stated bound is trivially true, or it is
bounded by
\begin{align*}
    6^{\sum_{i=1}^{n} |\widehat\mC_i|}
    =
    6^{n}\,,
\end{align*}
which is the desired bound and concludes the proof.
\end{proof}

\subsection{Proof of Proposition~\ref{prop_single_tree}}

\begin{proof}[Proof of Proposition~\ref{prop_single_tree}]
Fix any $T>0$, $ \ve \in (0,\tfrac{1}{T}
\wedge \tfrac{1}{2}) $ and $\tau\in \mT_3^{N_{\ve}}$.
Then by
Lemma~\ref{lem_expansion_factor} and  identity \eqref{e:homchaos}, we can
write
\begin{align*}
    X^\tau_\varepsilon
    =
    \frac{h^{( \red{ \tau})} (1)}{s(\red{\tau})}
    [\tau]_\varepsilon
    =
    \frac{h^{( \red{ \tau})} (1)}{s(\red{\tau})}
    \sum_{\kappa \in \mK(\tau)}
    [\tau]_{\kappa,\ve}
	=
    \frac{h^{( \red{ \tau})} (1)}{s(\red{\tau})}
     \left\{ \sum_{\kappa \in C(\tau)}
    [\tau]_{\kappa,\ve}
 + \sum_{\kappa \not\in C(\tau)}
    [\tau]_{\kappa,\ve}
\right\} \;.
\end{align*}
By the triangle inequality and the fact that $3^{| \red{\tau}|} =
\sum_{\kappa \in C ( \tau)} 1$, see Lemma~\ref{lem:count_contr}, we have that for every $(t,x) \in (0,
T] \times
\RR^2$
\begin{equation}\label{e:singleTree_supp0}
\begin{aligned}
\Bigg\| \sqrt{ \log{ \tfrac{1}{\ve}}}\cdot &X^\tau_\varepsilon(t,x) - \frac{h^{( \red{ \tau})} (1)}{\red{\tau}!\, s(\red{\tau})}
\left(\frac{3\Hat{\lambda}^2}{2\pi}\right)^{|\red{\tau}|} \hat{\lambda}
e^{\mathfrak{m}\, t} P_{t+
\ve^{2}}
\eta (x)  \Bigg\|_{L^{2}(\PP)}\\
 \leqslant  &
\frac{|h^{( \red{ \tau})} (1)|}{s(\red{\tau})}
 \sum_{\kappa \in C(\tau)}
\left\|      \sqrt{ \log{ \tfrac{1}{\ve}}}\cdot [\tau]_{\kappa,\ve}(t,x)
 - \frac{1}{ \red{ \tau}!}
\left(\frac{\Hat{\lambda}^2}{2\pi}\right)^{|\red{\tau}|} \hat{\lambda} e^{\mathfrak{m}\, t} P_{t+
\ve^{2}}
\eta (x)  \right\|_{L^{2}(\PP)}\\
&\qquad  +
 \frac{|h^{( \red{ \tau})} (1)|}{s(\red{\tau})}
\left\| \sum_{\kappa \not\in C(\tau)}
 \sqrt{ \log{ \tfrac{1}{\ve}}}\cdot
    [\tau]_{\kappa,\ve}(t,x)
 \right\|_{L^{2}(\PP)}\,.
\end{aligned}
\end{equation}
The second term on the right-hand side can be directly estimated using
Lemma~\ref{lem_unif_estimate_all_noncontr_cont_new} as follows:
\begin{equation}\label{e:singleTree_supp1}
\begin{aligned}
\left\| \sum_{\kappa \not\in C(\tau)}
 \sqrt{ \log{ \tfrac{1}{\ve}}}\cdot
    [\tau]_{\kappa,\ve}(t,x)
 \right\|_{L^{2}(\PP)}
\leqslant
\frac{1}{ \red{ \tau}!}
 \frac{1}{\sqrt{ 4\log \tfrac{1}{\ve}}}
    \left(\frac{6e^{2 + 2 \pi}\Hat{\lambda}^{2} e^{2 \overline{\mathfrak{m}}\, t} }{\pi}\right)^{|\red{\tau}|} \frac{ \hat{\lambda} e^{\mathfrak{m}\, t}
}{\sqrt{4 (t + \ve^{2})}}\,.
\end{aligned}
\end{equation}
For the first term, by applying Lemma~\ref{lem:contr} and
Lemma~\ref{lem:count_contr} together with the identity $\<10>_{\, \ve}(t,x)=
\hat{\lambda} ( \log{ \tfrac{1}{\ve}})^{- \frac{1}{ 2} } e^{\mathfrak{m}\, t} P_{t+ \ve^{2}} \eta(x)$, we obtain
\begin{align}
\sum_{\kappa \in C(\tau)}&
\left\|      \sqrt{ \log{ \tfrac{1}{\ve}}}\cdot  [\tau]_{\kappa,\ve}(t,x)
 - \frac{1}{ \red{ \tau}!}
\left(\frac{\Hat{\lambda}^2}{2\pi}\right)^{|\red{\tau}|} \hat{\lambda} e^{\mathfrak{m}\, t}  P_{t+
\ve^{2}}
\eta (x)  \right\|_{L^{2}(\PP)}\nonumber\\
=&
\frac{1}{ \red{ \tau}!}
\left(
\frac{ 3 \hat{\lambda}^{2}}{ 2 \pi} \right)^{|\red{ \tau}|}
\left\| \left( \left(
\frac{1}{2 \log{ \frac{1}{ \ve}}}
\int_{0}^{t} \frac{e^{2 \mathfrak{m} \, s}}{s+ \ve^{2}} \ud s
        \right)^{|\red{\tau}|}  -1 \right)
\hat{\lambda} 
 e^{\mathfrak{m}\, t} P_{t+ \ve^{2}}
\eta (x)  \right\|_{L^{2}(\PP)}\nonumber  \\
\leqslant &
\frac{
1}{ \red{ \tau}!}
\left(
\frac{
9 \hat{\lambda}^{2} e^{2 \overline{\mathfrak{m}} \, t+1}
}{2 \pi}
\right)^{| \red{ \tau}|}
\frac{ e^{2 |\mathfrak{m}|\, t} + | \log{(t + \ve^{2})}
| }{2 \log{ \frac{1}{ \ve}}}
\frac{ \hat{\lambda} e^{\mathfrak{m}\, t}}{ \sqrt{4 \pi (t+ \ve^{2})}}
 \,, \label{e:singleTree_supp2}
\end{align}
where in the last step, we used $\| P_{t+ \ve^{2}} \eta (x)\|_{L^{2}(\PP)}=
\sqrt{4 \pi (t+
\ve^{2})}^{- 1} $ and applied Corollary~\ref{c_1cycle_taylor}, which yields
\begin{equation*}
\begin{aligned}
 \left| \left(
\frac{1}{2 \log{ \frac{1}{ \ve}}}
\int_{0}^{t} \frac{e^{2 \mathfrak{m} \, s}}{s+ \ve^{2}} \ud s
        \right)^{|\red{\tau}|}
 -
1 \right|
& \leqslant
| \red{ \tau}| \,  \left(
3 e^{2 \overline{\mathfrak{m}} \, t}
\right)^{| \red{ \tau}|-1}
\frac{ e^{2 |\mathfrak{m}|\, t} + | \log{(t + \ve^{2})}
| }{2 \log{ \frac{1}{ \ve}}}
\\
& \leqslant
  \left( 3
e^{2 \overline{\mathfrak{m}} \, t+1}
\right)^{| \red{ \tau}|}
\frac{ e^{ 2|\mathfrak{m}|\, t} + | \log{(t + \ve^{2})}
| }{2 \log{ \frac{1}{ \ve}}}
\,,
\end{aligned}
\end{equation*}
where we used additionally \eqref{e:def_const_c} in the first
inequality,
as well as the bound $ | \red{ \tau}| \leqslant 3 e^{ | \red{ \tau}| }$ in
the second inequality.

Now, combining  \eqref{e:singleTree_supp0}, \eqref{e:singleTree_supp1} and \eqref{e:singleTree_supp2} yields
\begin{equation*}
\begin{aligned}
\bigg\| \sqrt{ \log{ \tfrac{1}{\ve}}}\cdot X^\tau_\varepsilon(t,x) & - \frac{h^{( \red{ \tau})} (1)}{\red{\tau}!\, s(\red{\tau})}
\left(\frac{3\Hat{\lambda}^2}{2\pi}\right)^{|\red{\tau}|} \hat{\lambda} e^{\mathfrak{m}\, t} P_{t+
\ve^{2}}
\eta (x)  \bigg\|_{L^{2}(\PP)}\\
&\leqslant
 \frac{|h^{( \red{ \tau})} (1)|}{\red{\tau}!\, s(\red{\tau})} ( \hat{\lambda}
e^{ \overline{\mathfrak{m}} \,t})^{2 | \red{ \tau}|}
\Bigg(
\left(
\frac{
9 \, e
}{2 \pi}
\right)^{| \red{ \tau}|}
\frac{ e^{2 |\mathfrak{m}|\, t} + | \log{(t + \ve^{2})}
| }{2 \log{ \frac{1}{ \ve}}}
\frac{1}{\sqrt{ \pi}}\\
& \qquad \qquad \qquad \qquad \qquad +
 \frac{1}{2\sqrt{ \log \frac{1}{ \ve}}}
    \left(\frac{6e^{2 + 2\pi}}{\pi}\right)^{|\red{\tau}|}
\Bigg)
 \frac{ \hat{\lambda} e^{\mathfrak{m}\, t}
}{\sqrt{4 (t + \ve^{2})}}
\\
&\leqslant
 \frac{|h^{( \red{ \tau})} (1)|}{\red{\tau}!\, s(\red{\tau})}
\left(C\, \hat{\lambda}^{2} e^{2 \overline{\mathfrak{m}}\, t} \right)^{ |\red{ \tau}|}
\frac{ e^{2 |\mathfrak{m}|\, t}+|\log{(t+
\ve^{2})}| + \sqrt{\log \tfrac{1}{\ve}} }{2 \log{ \tfrac{1}{\ve}}}
 \frac{ \hat{\lambda} e^{\mathfrak{m}\, t}
}{\sqrt{4 (t + \ve^{2})}}\,,
\end{aligned}
\end{equation*}
with $C$ being the constant defined in \eqref{e_def_CT}, satisfying
\begin{equation*}
\begin{aligned}
C
=
\frac{6e^{ 2 + 2 \pi}}{\pi}
\geqslant
\max \left\{
\frac{
9\,  e
}{2 \pi}
,
\frac{6e^{ 2 + 2 \pi}}{\pi}
\right\}
\,.
\end{aligned}
\end{equation*}
This concludes the proof.
\end{proof}


\section{Link to a McKean-Vlasov SPDE}\label{sec:MV}

This section is dedicated to the proof of Proposition~\ref{p_homog_ac}.
Before we do so, we must clarify the meaning of solution to mean-field SPDEs of
the form
\begin{equation} \label{e:mfr}
\begin{aligned}
\partial_t v = \frac{1}{2} \Delta v
+\mf{m}\, v - \alpha 
\EE \left[ v^{2} \right] 
v \,, \qquad v (0,x) = v_{0} (x)\,, \qquad \forall (t, x) \in (0, \infty)
\times \RR^{2} \;,
\end{aligned}
\end{equation}
for some $ \mf{m} \in \RR, \alpha > 0 $. To simplify the notation in the next
definition, let us write $ \mE $ for the set of functions $ f \colon
\RR^{2} \to \RR$ such that for some $ \lambda (f) > 0 $ we have $ \sup_{x
\in \RR^{2}} | f(x) | e^{- \lambda |x|} < \infty$.

\begin{definition}\label{def:solmf}
We say that $ v $ is a solution to
\eqref{e:mfr} if -- in addition to satisfying the equation -- it is smooth on $ (0,
\infty) \times \RR^{2} $, and if $ x \mapsto \sup_{0 \leqslant
s \leqslant T} | v (s, x) | \in \mE $ and  similarly $ x \mapsto \sup_{0
\leqslant s \leqslant T} \EE\left[ v^{2} (s, x)
\right] \in \mE $ for all $ T \in (0, \infty)$, $\PP$--almost surely.
\end{definition}

In this setting, our first result is uniqueness of solutions to
\eqref{e:mfr}. We observe that well-posedness of Mc-Kean--Vlasov SPDEs of this
type \emph{on finite volume} follows for example through the same arguments as
in \cite{hao}. Yet the extension to infinite volume is not entirely trivial,
and as a matter of fact we only prove existence of solutions in a special,
Gaussian, case. Instead, our argument for uniqueness works in full generality.

\begin{proposition} \label{prop:uniqueness}
For any $ v_{0} \in \mE  $ such that 
$$ 
C_{0} : = \sup_{x \in \RR^{2}} \EE[ v_{0}^{2}(x)] < \infty\,,
$$
there exists at most one solution $v$ to \eqref{e:mfr} with initial data $ v_{0} $.
\end{proposition}
\begin{proof}
We start by proving the following a priori estimate:
\begin{align}\label{apriori1}
\sup_{t \geqslant 0 \,,\, x\in \RR^2} \mathbb{E}\big[ v (t,x)^2\,\big]
\leqslant  C_{0} \;.
\end{align}
To establish the above, we start by observing that since $ v $ is smooth in
space and time, $v^2$ solves the initial value problem
\begin{align*}
\partial_t v^2 &= \frac{1}{2} \Delta v^2 + 2 \mf{m} v^{2} - |\nabla v|^2 -2 \alpha\mathbb{E}\big[
v^2 \big] \, v^2,  \qquad && (t,x)\in (0,\infty)\times \RR^2 \;, \\
\qquad v^2(0,x) & = v_{0}^2(x)\;,\qquad  && x\in \RR^2 \;.
\end{align*}
By the maximum principle or the Feynman-Kac formula one can therefore see that for
every $t \geqslant 0$ and $x\in \RR^2$ it holds that
\begin{align*}
v^2(t,x) \leqslant  \int_{\RR^2} v^2_{0}(x-y) p_t^{(\mf{m})}(y) \, \ud y \;, \qquad  & x\in \RR^2,
\end{align*}
and upon taking expectations on both sides we obtain
\begin{align*}
\mathbb{E}[ v^{2} (t,x) ] \leqslant  C_{0} \int_{\RR^2}  p_t^{(\mf{m})}(y) \ud  y =
C_{0}e^{ \mf{m} t} \;,
\end{align*}
which confirms \eqref{apriori1}.

Let us now proceed with the claim of uniqueness. Let $u,v$ be two solutions to
\eqref{e:mfr} with the same initial condition $ v_{0} $.
We then show that the difference $w (t,x):=u (t,x)-v (t,x)$ is
identically equal to zero. The see this, note that $ w $
satisfies the initial value problem
\begin{equation} \label{weq}
\begin{aligned}
\partial_t w &= \frac{1}{2}  \Delta w + \mf{m} w - \alpha \Big( \mathbb{E}\big[ u^2 \big] -
\mathbb{E}\big[ v^2 \big] \Big) \, u
  - \alpha \mathbb{E}\big[ v^2 \big] \, w, & & \forall (t, x) \in (0, \infty) \times
\RR^{2} \;, \\
w(0,x) & = 0 \;,  & & \forall  x\in \RR^2 \;. 
\end{aligned}
\end{equation}
Denote by $Z(t,x):= \mathbb{E}\big[ u(t,x)^2 \big] - \mathbb{E}\big[
v(t,x)^2 \big] = 
\mathbb{E}\big[ \big(u(t,x) - v(t,x) \big)  \big(u(t,x) + v (t,x) \big) \big]$.
From Cauchy--Schwarz, it follows that
\begin{align}\label{apriori2}
|Z (t,x)| 
& \leqslant  \mathbb{E}\big[ w (t,x)^2 \big]^{1/2} \, \mathbb{E}\big[ (u
+v )^2(t,x) \big]^{1/2} \leqslant   2\, \sqrt{C_{0}} \,\,\mathbb{E}\big[ w(t,x)^2 \big]^{1/2 } \;,
\end{align} 
where in the last we use the triangle inequality and the a priori estimate
\eqref{apriori1}. Then, by the Feynman--Kac formula and in view of the growth
assumption on solutions in Definition~\ref{def:solmf}, we can express $w$
from \eqref{weq} through
\begin{align*}
w (t,x) = \alpha \E_x \left[  \int_0^t  Z (t-s,\beta (s)) \, u (t-s, \beta (s)) \,\,
e^{\mf{m}s- \alpha \int_0^s \mathbb{E}\big[ v (s-r,\beta (r))^2 \big] \ud r } \ud s\right] \;,
\end{align*} 
where $ \mathbf{E}_{x} $ indicates the expectation over a Brownian motion $
\beta $ (independent from all other random variables) started in $ \beta(0) = x $.
From here, using \eqref{apriori2} and bounding the exponential term in the
Feynman--Kac representation, it follows that
\begin{align*}
| w(t,x) | \leqslant   2 \alpha \, \sqrt{C_{0}} e^{ \overline{\mf{m}} t}\, 
  \E_x\left[ \int_0^t \mathbb{E} \big[ w(t-s,\beta(s))^2 \big]^{1/2 } \, | u
(t-s, \beta(s)) | \, \ud s\right],
\end{align*}
with $ \overline{ \mf{m}} = \max \{ 0, \mf{m} \} $.
Next, by Cauchy--Schwarz, we further have the estimate
\begin{align*}
&| w(t,x) | 
\leqslant   2 \alpha \, \sqrt{C_{0}}\, e^{ \overline{\mf{m}} t}   \E_x
\left[ \int_0^t \mathbb{E} \big[
w(t-s,\beta(s))^2 \big]\, \ud s\right]^{1/2} 
        \,  \E_x \left[ \int_0^t u (t-s, \beta(s))^2 \, \ud s \right]^{1/2},
\end{align*}
and then
\begin{align*}
\mathbb{E}\big[ | w(t,x) |^2\big]  
& \leqslant   4 \alpha^{2}  C_{0}   \, e^{ 2\overline{ \mf{m}} t} \E_x \left[
\int_0^t \mathbb{E} \big[ w(t-s,\beta(s))^2
\big]\, \ud s\right] \,
        \, \E_x \left[ \int_0^t \mathbb{E} [u (t-s, \beta(s))^2] \, \ud s \right] \\
& \leqslant  4 \alpha^{2} t C_{0}^2  e^{ 2\overline{ \mf{m}} t} \, \E_x \left[
\int_0^t \mathbb{E} \big[ w(t-s,\beta(s))^2
\big]\, \ud s \right] \\
& = 4 \alpha^{2}t C_{0}^2 e^{ 2\overline{ \mf{m}} t} \int_0^t  \int_{\RR^2} \mathbb{E} \big[ w(t-s,y)^2
\big]\, p_s(y-x) \, \ud y \, \ud s \;,
\end{align*}
where in the second inequality we used, again, the a priori bound \eqref{apriori1}. 
Taking the supremum over the $ x $ variable on both sides we conclude
\begin{equation*}
\begin{aligned}
\sup_{x \in \RR^{2}} \EE[ | w (t, x) |^{2} ] \leqslant 4 \alpha^{2} t C_{0}^{2}
e^{ 2 \overline{\mf{m}} t}
\int_{0}^{t} \sup_{x \in \RR^{2}} \EE[| w (t-s, x) |^{2}] \ud s \;.
\end{aligned}
\end{equation*}
Now, setting $F(t):= \sup_{x \in \RR^{2}}  \mathbb{E}  | w(t,x) |^2 $, we
conclude that for every $T>0$ and $0<t \leqslant T$
\begin{align*}
F(t) \leqslant  \overline{C}_T \int_0^t F(s) \, \ud s,\qquad \text{where} \quad
\overline{C}_T:=4 \alpha^{2} T C_{0}^2 e^{2 \overline{\mf{m}} t} \;.
\end{align*}
Since $F(0)=0$, it follows that $F(t)=0$ for every $0<t \leqslant  T$ and
$T>0$.
\end{proof}

We are now ready for the main result of this section.

\begin{proof}[Proof of Proposition~\ref{p_homog_ac}]

In the particular case of $ v_{0} = p_{\ve^{2}} \star \eta $ and $ \alpha =
\frac{ 3 }{\log{ \frac{1}{\ve}}} $, we can check that the McKean--Vlasov equation
\eqref{e:mfr} admits the solution
\begin{equation*}
\begin{aligned}
v_{ \hat{\lambda}, \ve}(t,x) := \hat{\lambda}\, \sigma_{\hat{\lambda}, \ve} (t) P_{t}^{(\mf{m})}
(p_{\ve^{2}} \star \eta)(x)
 \;,
\end{aligned}
\end{equation*}
where $ \sigma_{\hat{\lambda}, \ve} $ is the solution to the ODE
\begin{equation}\label{e_ode_mckean}
\begin{aligned}
\partial_{t} \sigma_{ \hat{\lambda}, \ve} = - \frac{3 \hat{\lambda}^{2}}{
\log{ \tfrac{1}{\ve}}}\frac{  e^{2 \mf{m} t}}{4 \pi
(t + \ve^{2})}  \sigma_{ \hat{\lambda}, \ve}^{3} \;,
\quad \sigma_{\hat{\lambda}, \ve} (0) = 1 \;.
\end{aligned}
\end{equation}
For this, let us substitute $v_{ \hat{\lambda}, \ve}$ into the mild formulation of
\eqref{e:mfr}, which yields 
\begin{equation*}
\begin{aligned}
& \hat{\lambda}\,  \sigma_{\hat{\lambda}, \ve} (t) P_{t}^{(\mf{m})}
(p_{\ve^{2}} \star \eta)(x)\\
&=
 \hat{\lambda}\,
P_{t}^{(\mf{m})}
(p_{\ve^{2}} \star \eta)(x)\\
&\qquad -
\frac{ 3 \hat{\lambda}^{3} }{\log{ \tfrac{1}{\ve}}} 
\int_{0}^{t}
 \sigma_{\hat{\lambda}, \ve} (s)^{3}
\int_{\RR^{2}} 
p_{t-s}^{(\mf{m})}(y-x)
\EE \big[
| P_{s}^{(\mf{m})}
(p_{\ve^{2}} \star \eta)(y)
 |^{2}
\big]
P_{s}^{(\mf{m})}
(p_{\ve^{2}} \star \eta)(y)
 \ud y
 \ud s\\
&=
 \hat{\lambda}\,
P_{t}^{(\mf{m})}
(p_{\ve^{2}} \star \eta)(x)
-
\frac{ 3 \hat{\lambda}^{3} }{\log{ \tfrac{1}{\ve}}} 
\left\{
\int_{0}^{t}
 \sigma_{\hat{\lambda}, \ve} (s)^{3}
\frac{ e^{2\mf{m} \, s}}{4 \pi(s+ \ve^{2})} 
 \ud s\right\}
\;
P_{t}^{(\mf{m})}
(p_{\ve^{2}} \star \eta)(x)\,,
\end{aligned}
\end{equation*}
where we used that $ \EE \big[
| P_{s}^{(\mf{m})}
(p_{\ve^{2}} \star \eta)(y)
 |^{2}
\big]
 = e^{2 \mf{m}\, s} (4 \pi(s + \ve^{2}))^{-1}$ and integrated out the spatial
variable using Chapman--Kolmogorov. Now, dividing by $ \hat{\lambda}\,P_{t}^{(\mf{m})}
(p_{\ve^{2}} \star \eta)(x)$ on both sides and taking the temporal derivative
yields \eqref{e_ode_mckean}. To verify that $ v_{ \hat{\lambda}, \ve} $ is a
solution in the sense of Definition~\ref{def:solmf}, it now suffices to check
that $ \sup_{0 \leqslant s \leqslant T} | p_{s+\ve^{2}} \star \eta| \in \mE $
for any $ \ve, T > 0 $. This follows for
example because $ \EE \big[ \sup_{| x |_{\infty} \leqslant 1} \sup_{0 \leqslant
s \leqslant T}| p_{s+\ve^{2}} \star
\eta (x)| \big] < \infty  $ (where $ | x |_{\infty} = \max_{i=1,2} | x_{i} | $), so that
\[ \EE \left[ \sum_{z \in \ZZ^{2}} \frac{\sup_{| x - z |_{\infty} \leqslant 1}
\sup_{0 \leqslant s \leqslant T} | p_{s+\ve^{2}} \star
\eta (x) |}{1 + |z|^3} \right] < \infty \;,\] 
which yields the desired result (cf. \cite[Lemma 1.1]{HL}).

Now, the differential equation \eqref{e_ode_mckean} admits the explicit solution 
\begin{equation}\label{e:explicit}
\begin{aligned}
{\sigma}_{ \hat{\lambda}, \ve}(t)
=
\frac{1}{\sqrt{1 + \frac{3 \hat{\lambda}^{2}}{ \pi} e^{- 2 \mf{m}\, 
\ve^{2}} c_{ \ve}(t)
}} \,,
\qquad \text{with}
\qquad
c_{ \ve}(t):=
\frac{1}{ 2 \log{ \frac{1}{ \ve}}} 
\int_{0 }^{t } \frac{ e^{2 \mf{m} s}}{s + \ve^{2}}   \ud s \,.
\end{aligned}
\end{equation}
Moreover, the solution $v_{ \hat{\lambda}, \ve}$  to \eqref{e:mfr} is
unique, which follows from Proposition~\ref{prop:uniqueness}.
Finally, by Lemma~\ref{l_app_expInt}, we have that $c_{ \ve}(t) \to 1
 $ as $ \ve \to 0$, for every $ t>0$. Hence, 
\begin{equation*}
\begin{aligned}
\lim_{ \ve \to 0}
\sigma_{ \hat{\lambda},  \ve}(t) 
=
\frac{1}{\sqrt{1 + \frac{3 \hat{\lambda}^{2}}{ \pi}
}}=  \sigma_{ \hat{\lambda}} \qquad \forall t >0 \,.
\end{aligned}
\end{equation*} 
This completes the proof.
\end{proof}

\appendix
\section{Appendix}

\subsection{On the exponential integral}

Throughout the paper, a crucial ingredient is to understand the small-$ \ve $
behaviour of the integral
\begin{equation*}
\begin{aligned}
 \int_{0}^{t} \<cycle1>_{ \ve} (s) \ud s
=
\frac{\lambda_\varepsilon^2}{4\pi}
\int_{0}^{t} \frac{e^{2 \mathfrak{m} \, s}}{s+ \ve^{2}}  \ud s  \,, \qquad t
\in (0, \infty) \,, \ \ve \in ( 0, \tfrac{1}{2})\,.
\end{aligned}
\end{equation*}
In this appendix we will prove some useful estimates concerning this integral.

\begin{lemma}\label{l_app_expInt}
Let $\mathfrak{m} \in \RR$ and $t \in (0, \infty) $, then
\begin{equation*}
\begin{aligned}
\left|
\frac{1}{2 \log{ \frac{1}{ \ve}}}
\int_{0}^{t} \frac{e^{2 \mathfrak{m} \, s}}{s+ \ve^{2}} \ud s   -1
        \right|
\leqslant
\frac{ e^{ 2|\mathfrak{m}|\, t} + | \log{(t + \ve^{2})}
| }{2 \log{ \frac{1}{ \ve}}}
\,,
\end{aligned}
\end{equation*}
for every $ \ve \in (0,\tfrac{1}{2})$, with $ \overline{\mathfrak{m}} =
 \max \{ \mf{m}\;, 0 \}$. In particular, for every $ t \in (0,
\infty)$
\begin{equation}\begin{aligned}
\lim_{ \ve \to 0}
\int_{0}^{t} \<cycle1>_{ \ve} (s) \ud s
= \frac{ \hat{\lambda}^{2}}{2 \pi} \,.
\end{aligned}
\end{equation}\end{lemma}

\begin{proof}
First expanding the exponential, we write
\begin{equation*}
\begin{aligned}
 \int_{0}^{t} \frac{e^{2\mathfrak{m}\,
s}-1}{s+ \ve^{2}} \ud s
&=
\int_{0}^{t}
\frac{ s}{s + \ve^{2}} \sum_{k=1}^{ \infty}
\frac{(2\mathfrak{m})^{k}\, s^{k-1}}{k!}
\ud s\,.
\end{aligned}
\end{equation*}
Thus, we obtain
\begin{equation*}
\begin{aligned}
\left|
 \int_{0}^{t} \frac{e^{2 \mathfrak{m}\,
s}-1}{s+ \ve^{2}} \ud s
\right|
\leqslant
\sum_{k=1}^{ \infty}
\int_{0}^{t}
 \frac{|2 \mathfrak{m}|^{k}\, s^{k-1}}{k!}
\ud s \leqslant e^{2 | \mathfrak{m}|\, t} \,.
\end{aligned}
\end{equation*}
In addition, we have 
\begin{equation*}
\begin{aligned}
\left|
\frac{1}{2 \log{ \frac{1}{ \ve}}}
\int_{0}^{t} \frac{1}{s+ \ve^{2}} \ud s   -1
\right|
=
\frac{| \log{( t+ \ve^{2} )} |}{ 2 \log{ \frac{1}{ \ve}}}\,,
\end{aligned}
\end{equation*}
so that the statement follows from the triangle inequality.
The second part of the statement is now an immediate consequence, since 
\begin{equation*}
\begin{aligned}
 \int_{0}^{t} \<cycle1>_{ \ve} (s) \ud s
= \frac{\hat\lambda^{2}}{2 \pi } \frac{ 1}{ 2 \log \frac{1}{\ve}}  \int_{0}^{t} \frac{ e^{2\mf{m} \, s}}{s+ \ve^{2}} \ud s\,.
\end{aligned}
\end{equation*}
This completes the proof.
\end{proof}

\begin{corollary}\label{c_1cycle_taylor}
Let $\mathfrak{m} \in \RR$, $T>0$ and $k \in \NN $, then
\begin{equation*}
\begin{aligned}
 \left| \left(
\frac{1}{2 \log{ \frac{1}{ \ve}}}
\int_{0}^{t} \frac{e^{2 \mathfrak{m} \, s}}{s+ \ve^{2}} \ud s
        \right)^{k}
 -
1 \right|
\leqslant
k \, \left(
3e^{ 2\overline{\mathfrak{m}}\, t}
\right)^{k-1}
\frac{ e^{ 2|\mathfrak{m}|\, t} + | \log{(t + \ve^{2})}
| }{2 \log{ \frac{1}{ \ve}}}
\,,
\end{aligned}
\end{equation*}
for every $t \in [0,T]$ and $ \ve \in (0,\tfrac{1}{T}\wedge \tfrac{1}{2})$.
\end{corollary}

\begin{proof}
Using a first order Taylor expansion of the monomial of order $k$ around $1$, we have
\begin{equation} \label{e:supapp}
\begin{aligned}
& \left| \left(
\frac{1}{2 \log{ \frac{1}{ \ve}}}
\int_{0}^{t} \frac{e^{2 \mathfrak{m} \, s}}{s+ \ve^{2}} \ud s
        \right)^{k}
 -
1 \right|\\
&\leqslant
k \,\left\{\sup_{x \in [1,b_{\ve}(t) \vee 1]} x^{k-1} \right\}
\left|
\frac{1}{2 \log{ \frac{1}{ \ve}}}
\int_{0}^{t} \frac{e^{2 \mathfrak{m} \, s}}{s+ \ve^{2}} \ud s   -1
        \right|\,,
\end{aligned}
\end{equation}
with
\begin{equation*}
\begin{aligned}
b_{ \ve}(t)
:=
\frac{1}{2 \log{ \frac{1}{ \ve}}}
\int_{0}^{t} \frac{e^{2 \mathfrak{m} \, s}}{s+ \ve^{2}} \ud s
\leqslant
e^{2 \overline{\mathfrak{m}} \, t}
\left( 
1+
\frac{| \log{(t+ \ve^{2})} |}{ 2 \log{
\frac{1}{\ve}}
}
\right)
\,.
\end{aligned}
\end{equation*}
In particular, we find that
\begin{equation*}
\begin{aligned}
\sup_{x \in [1,b_{\ve}(t) \vee 1]} x^{k-1}
\leqslant
e^{2(k-1) \overline{\mathfrak{m}} \, t} \left(
1+
\frac{| \log{(t+ \ve^{2})} |}{ 2 \log{
\frac{1}{\ve}}
}
\right)^{k-1}
\leqslant
\big( 3 
e^{2 \overline{\mathfrak{m}} \, t} \big)^{k-1}
\,,
\end{aligned}
\end{equation*}
where we used \eqref{e:def_const_c} in the last step.
The statement follows now by upper bounding the last term in \eqref{e:supapp}
via Lemma~\ref{l_app_expInt}.
\end{proof}

\subsection{Symmetric functions and trees}

\begin{lemma}\label{symmetricintegral}
For a tree $\tau \in \mT $ of the form $ \tau= [\tau_1\ \cdots\ \tau_n]$ such that  $\mI
:= \mI ( \tau) \setminus \mf{o}_{ \tau}\neq \emptyset$  and any symmetric
function $\Psi\colon \RR^{\mI}\to \RR$, we have that
\begin{align*}
\int_{\mathfrak{D}_{\tau}(t)} \Psi( s_{\mI}) \ud s_{\mI} =
\frac{1}{\red{\tau_{1}} !\cdots \red{ \tau_{n}}!} \int_{[0,t]^{\mI}} \Psi(s_{\mI }) \ud
s_{\mI }\,,
\end{align*}
with the domain $\mathfrak{D}_{\tau}(t)$ defined in \eqref{timesimplex} and $ \red{ \tau} =
\mathscr{T} ( \tau)$ denoting the trimmed tree defined in \eqref{trimming}.
\end{lemma}
\begin{proof}
We prove the statement by induction over $ i ( \tau) \geqslant 2$. Note that the
case $i ( \tau)=1$ is excluded as the single inner node must necessarily be
the root.
If $i (\tau) = 2$, then $ \tau$ must be of the form $
\tau = [  \begin{tikzpicture}[baseline=2,scale=.8]
			\draw (0,0) -- (-.5,0.4) node[dot] {};
			\draw (0,0) -- (-.2,0.4) node[dot] {} ;
			\draw (0,0) -- (0,0.3) node  at (0.2,.4) {\tiny$\cdots$} ;
			\draw (0,0) -- (0.2,0.3) ;
			\draw (0,0) node[idot] {} --  (0.5,0.4) node[dot] {}  ;
\end{tikzpicture}
\  \<0>\  \cdots\  \<0>]$ and $\mI =\{v\}$ for some node $ v $. Then
\begin{equation*}
\begin{aligned}
\int_{\mathfrak{D}_{\tau}(t)} \Psi( s_{\mI}) \ud s_{\mI}
= \int_{0}^{t}  \Psi( s_{v})  \ud s_{v} \,,
\end{aligned}
\end{equation*}
which is the desired statement since $ \mathscr{T} ( \begin{tikzpicture}[baseline=2,scale=.8]
			\draw (0,0) -- (-.5,0.4) node[dot] {};
			\draw (0,0) -- (-.2,0.4) node[dot] {} ;
			\draw (0,0) -- (0,0.3) node  at (0.2,.4) {\tiny$\cdots$} ;
			\draw (0,0) -- (0.2,0.3) ;
			\draw (0,0) node[idot] {} --  (0.5,0.4) node[dot] {}  ;
\end{tikzpicture}
) ! =1 $.
Now assume that the statement holds true for any tree $ \tau'= [
\tau'_{1} \ \cdots \tau'_{n'}]$ such that $i (
\tau') \leqslant N$, for some $N \geqslant 2$, and write $ \mI ' :=
\mI ( \tau') \setminus \mf{o}'$ with $\mf{o}' := \mf{o}_{ \tau'}$.
Let  $ \tau \in \mT$ be of the form $ \tau = [ \tau'\  \<0> \ \cdots\
\<0>]$, so that $ \red{ \tau} = [ \red{ \tau'}]$.
Then
\begin{equation*}
\begin{aligned}
\int_{\mathfrak{D}_{\tau}(t)} \Psi( s_{\mI}) \ud s_{\mI}
&=
\int_{0}^{t} \int_{\mathfrak{D}_{\tau'}(s_{\mf{o}'})}
\Psi(s_{\mf{o}'}, s_{\mI'})
\ud s_{\mI'}
\ud s_{\mf{o}'}\\
&=
\frac{1}{ \red{ \tau '_{1}}! \cdots \red{ \tau '_{n'}}!  }
\int_{0}^{t}
\int_{[0, s_{\mf{o}'}]^{\mI'}}
\Psi(s_{\mf{o}'}, s_{\mI'})
\ud s_{\mI'}
\ud s_{\mf{o}'}\,,
\end{aligned}
\end{equation*}
where we used the induction hypothesis and the fact
that $\Psi(s_{\mf{o}'}, \cdot) \colon \RR^{\mI^{\prime}} \to \RR $ is a symmetric function. Using the symmetry of the function $\Psi$, we further see that the
identification of variable $s_{\mf{o}'}$ as the maximum
variable is irrelevant with regards to the integration, and the assignment of any of the
variables $s_{\mI}$ as being the maximum would result in the same value.
Thus, the above is equal to
\begin{equation*}
\begin{aligned}
\int_{\mathfrak{D}_{\tau}(t)} \Psi( s_{\mI}) \ud s_{\mI}
=
\frac{1}{|\mI|\cdot   \red{ \tau'_{1}}! \cdots  \red{ \tau'_{n'}}! }
\int_{[0, t]^{\mI }}
\Psi(s_{\mI})
\ud s_{\mI}
\end{aligned}\,,
\end{equation*}
and the statement follows now for $ \tau$ because $ \red{ \tau'}!= |\red{ \tau'}|\cdot  \red{
\tau '_{1}}! \cdots \red{ \tau'_{n'}}!  =|\mI|\cdot  \red{
\tau '_{1}}! \cdots \red{ \tau'_{n'}}!$.
Furthermore, we notice that
\begin{equation*}
\begin{aligned}
\int_{\mathfrak{D}_{\tau}(t)} \Psi( s_{\mI}) \ud s_{\mI}
=
\int_{\mathfrak{D}_{[\tau']}(t)} \Psi( s_{\mI( \tau')}) \ud s_{\mI( \tau')} \;,
\end{aligned}
\end{equation*}
meaning that the additional occurrences of $\<0>$ in the grafting of $ \tau$ do not
affect the integral.
Finally, consider the case $ \tau \in \mT$ such that $i ( \tau) = N+1$, with $ \tau = [
\tau_{1}\  \cdots \  \tau_{k}\  \<0>\ \cdots \ \<0>]$, $k \geqslant 2 $ and $
\tau_{i}\neq \<0>$.
Again the extra occurrences of $\<0>$ in the grafting of $ \tau$ do not affect
the value of the integral. Moreover, the integration domain
$\mathfrak{D}_{\tau}(t)$ can be written as an union of sub-tree-simplices:
\begin{equation*}
\begin{aligned}
\int_{\mathfrak{D}_{\tau}(t)} \Psi( s_{\mI}) \ud s_{\mI}
=
\int_{\mathfrak{D}_{[\tau_{1}]}(t)}
\cdots
\int_{\mathfrak{D}_{[\tau_{k}]}(t)}
 \Psi( s_{\mI}) \ud s_{\mI(
\tau_{k})} \cdots \ud s_{\mI(
\tau_{1})}\,.
\end{aligned}
\end{equation*}
The statement follows from the induction hypothesis, since the restriction
of $\Psi$ to a subset of variables remains a symmetric
function.
\end{proof}

\subsection{Proof of Lemma \ref{vcycleexists}}\label{sec_v_cycle_exist}

Here we prove Lemma \ref{vcycleexists}, which guarantees the existence of a
${\rm v}$-cycle in the gluing of two ternary trees.
First, let us introduce some notation. For $ \tau \in \mT_{ \leqslant 3} $,
we partition the subset of inner nodes neighbouring leaves in $\mL ( \tau)$
as follows. Let
\begin{itemize}
\item $\mV_{\<3s>}( \tau)$ be the subset of inner nodes $v \in \mI ( \tau)$ that
is a \emph{basis of a
trident}, i.e.\ there exist exactly three $ u_{1}, u_{2}, u_{3} \in \mL ( \tau)$ such
that $\mathfrak{p}(u_{i})=v$.
\item $\mV_{\<2s>}( \tau)$ be the subset of inner nodes $v \in \mI ( \tau)$ that
is a \emph{basis of a
cherry}, i.e.\ there exist exactly two $ u_{1}, u_{2} \in \mL ( \tau)$ such
that $\mathfrak{p}(u_{i})=v$.
\item $\mV_{\<10s>}( \tau)$ be the subset of inner nodes $v \in \mI ( \tau)$ that
is a \emph{basis of a
lollipop}, i.e.\ there exist exactly one $ u \in \mL ( \tau)$ such
that $\mathfrak{p}(u)=v$.
In the following, we will call elements in $\mV_{\<10s>}( \tau)$
\textbf{dead-ends}.
\end{itemize}

\begin{proof}
Let us start by noting that if $[\tau_1,\tau_2]$ does not contain any dead-ends, then the paired tree $[\tau_1,\tau_2]_\gamma$
contains a ${\rm v}$-cycle. To see this, first notice that every leaf  of $[\tau_1,\tau_2]$ belongs either to a cherry or to a trident.
Now,
consider an arbitrary leaf, call it $v_0$, and let $v_1$ be the unique leaf in $[\tau_1,\tau_2]$, which is  connected to $v_0$ via $\gamma$.
If $v_1$ is inside the same cherry or trident component as  $v_0$, then we have already identified a ${\rm v}$-cycle, which in this case is of length $1$.
If not, then let $v_2$ be a different leaf inside the same cherry or trident
component as that of $v_1$ and denote by $v_3$ the leaf which is
connected to $v_2$ via $\gamma$. Again, if $v_2$ and $v_3$ fall inside the same component (which in this case would necessarily be a trident), then
a ${\rm v}$-cycle comprising of leaves $v_2,v_3$ and the corresponding base point of the trident is identified.
If not, then continue the procedure. Since there is only a finite number
of leaves, we will either encounter somewhere in the process a ${\rm v}$-cycle of
length $1$, or the path will return to a component previously visited
during
this process,
thus identifying a ${\rm v}$-cycle. Diagrammatically, we have the following representation:
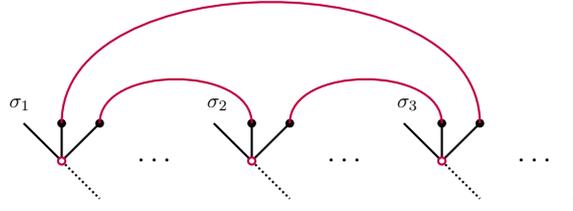
\begin{align*}
\begin{tikzpicture}[scale=0.5]
\draw  [fill] (1, 1)  circle [radius=0.1];
\draw[-, thick] (0,0)--(1, 1); \draw[-, thick] (0,0)--(-1, 1);
\draw[-, thick] (0,0) --(0,1);
\node at (-1.1,1.5) {\scalebox{0.7}{$\sigma_1$}};
\draw  [fill] (0, 1)  circle [radius=0.1];
\draw[densely dotted, thick] (0,0) node[idot] {}  -- (1,-1) ;
\node at (2.5,0) {$\cdots$};
\draw  [fill] (5, 1)  circle [radius=0.1];
 \draw[-, thick] (5,0)--(6, 1); \draw[-, thick] (5,0)--(4, 1);
\draw[-, thick] (5,0)--(5,1);
\node at (4.1,1.5) {\scalebox{0.7}{$\sigma_2$}};
\draw  [fill] (6, 1)  circle [radius=0.1];
\draw[densely dotted, thick] (5,0) node[idot] {}  -- (6,-1) ;
\node at (7.5,0) {$\cdots$};
\draw  [fill] (10, 1)  circle [radius=0.1];
 \draw[-, thick] (10,0)--(11, 1); \draw[-, thick] (10,0)--(9, 1);
\draw[-, thick] (10,0)--(10,1);
\node at (9.1,1.5) {\scalebox{0.7}{$\sigma_3$}};
\draw  [fill] (11, 1)  circle [radius=0.1];
\draw[densely dotted, thick] (10,0) node[idot] {}  -- (11,-1) ;
\node at (12.5,0) {$\cdots$};
\draw [thick, purple] (1,1)  to [out=90,in=90] (5,1); \draw [thick, purple]  (6,1) to [out=90,in=90]  (10,1);
\draw [thick, purple] (0,1)  to [out=90,in=90] (11,1);
\end{tikzpicture}\,,
\end{align*}
where sub-trees $\sigma_1, \sigma_1,\sigma_3$ may be identical to just a single
leaf, i.e. $\<0>$,
and even though we did not include them, there are $\gamma$ links emanating from the leaves of
these trees.

We will next reduce the case that $[\tau_1,\tau_2]$ contains also dead-ends to a situation of no dead-ends.
Dead-ends present a problem: When tracing contractions in $\gamma$, we may hit a dead-end
and thus are not able to continue to complete a ${\rm v}$-cycle. What we will show is that by eliminating paths that
start from a dead-end, the resulting sub-graph is one that consists of only cherries and tridents linked through
$\gamma$. Thus, a ${\rm v}$-cycle exists within this sub-graph by the previous argument.

Let us start by picking an arbitrary dead-end of $[\tau_1,\tau_2]$.
Call $v_0$ its associated leaf and suppose it connects via $ \gamma$ to another leaf
of $[\tau_1,\tau_2]$, which we call $v_1$. Now, remove this connection as
follows:
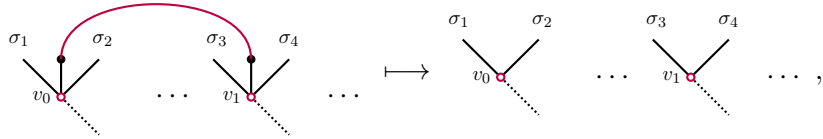
\begin{align*}
\begin{tikzpicture}[baseline={([yshift=-1.5ex]current bounding box.center)},vertex/.style={anchor=base,
    circle,fill=black!25,minimum size=18pt,inner sep=2pt}, scale=0.5]
\draw  [fill] (0, 1)  circle [radius=0.1];
\draw[-, thick] (0,0)--(1, 1); \draw[-, thick] (0,0)--(-1, 1);
\draw[-, thick] (0,0)--(0,1);
\node at (-1.1,1.5) {\scalebox{0.7}{$\sigma_1$}};
\node at (1.1,1.5) {\scalebox{0.7}{$\sigma_2$}};
\draw[densely dotted, thick] (0,0) node[idot] {}  -- (1,-1) ;
\node at (3,0) {$\cdots$};
\draw  [fill] (5, 1)  circle [radius=0.1];
 \draw[-, thick] (5,0)--(6, 1); \draw[-, thick] (5,0)--(4, 1);
\draw[-, thick] (5,0)--(5,1);
\node at (4.1,1.5) {\scalebox{0.7}{$\sigma_3$}};
\node at (6, 1.5)  {\scalebox{0.7}{$\sigma_4$}};
\draw[densely dotted, thick] (5,0) node[idot] {}  -- (6,-1) ;
\node at (7.5,0) {$\cdots$};
\draw [thick, purple] (0,1)  to [out=90,in=90] (5,1);
\node at (-.5,0) {\scalebox{0.7}{$v_{0}$}};
\node at (4.5,0) {\scalebox{0.7}{$v_{1}$}};
\end{tikzpicture}
\longmapsto
\begin{tikzpicture}[baseline={([yshift=-1.5ex]current bounding box.center)},vertex/.style={anchor=base,
    circle,fill=black!25,minimum size=18pt,inner sep=2pt}, scale=0.5]
\draw[-, thick] (0,0)--(1, 1); \draw[-, thick] (0,0)--(-1, 1);
\node at (-1.1,1.5) {\scalebox{0.7}{$\sigma_1$}};
\node at (1.1,1.5) {\scalebox{0.7}{$\sigma_2$}};
\draw[densely dotted, thick] (0,0) node[idot] {}  -- (1,-1) ;
\node at (3,0) {$\cdots$};
\draw[-, thick] (5,0)--(6, 1); \draw[-, thick] (5,0)--(4, 1);
\node at (4.1,1.5) {\scalebox{0.7}{$\sigma_3$}};
\node at (6, 1.5)  {\scalebox{0.7}{$\sigma_4$}};
\draw[densely dotted, thick] (5,0) node[idot] {}  -- (6,-1) ;
\node at (7.5,0) {$\cdots$};
\node at (-.5,0) {\scalebox{0.7}{$v_{0}$}};
\node at (4.5,0) {\scalebox{0.7}{$v_{1}$}};
\end{tikzpicture},
\end{align*}
where again we understand that, even not shown, there are $\gamma$-links emanating from trees
$\sigma_1,\sigma_2,\sigma_3,\sigma_4$. Moreover, we agree that neither
$\sigma_1$ nor $\sigma_2$ equals $\<0>$
(they might be $\emptyset$, though),
so that this part of the tree corresponds to a dead-end, while $\sigma_3,\sigma_4$
 might be comprising a $\<0>$.

In the resulting (contracted) tree on the right-hand side, we distinguish three cases:
\begin{itemize}
\item[1.] Neither of $\sigma_3$ and $\sigma_4$ are single leaves. In this case we have eliminated two dead-ends, while
not affecting the number of cherries and tridents.
\item[2.] Only one of the  $\sigma_3$ and $\sigma_4$ is a $\<0>$.  In this case, we have eliminated a dead-end in the left part
of the tree, while we have also created a new dead-end in the right part, corresponding to either $\sigma_3$ or $\sigma_4$, whichever
happens to be the $\<0>$. In this case, we have also reduced by one the number of cherries
 (by eliminating the cherry that was present in the right part of the sub-tree)
but, nevertheless, we have not reduced the number of tridents.
\item[3.] Both $\sigma_3$ and $\sigma_4$ are $\<0>$. In this case, in the right part of the tree we had, before the elimination,
 a trident. Thus, after the elimination we reduced both the number of dead-ends
and tridents by one, while the number
of cherries actually increased by one.
\end{itemize}
As we will prove in Lemma~\ref{lem_trid_vs_deadends} below, the total number of tridents  in $[\tau_1,\tau_2]$ is strictly larger than
the number of dead-ends. In all three cases above, the elimination procedure preserves this inequality.
Indeed, in Case 1.\ the number of dead-ends is reduced by $2$ while the number of tridents remains the same,
in Case 2.\ both the number of tridents and dead-ends remain the same (the number of cherries is reduced by one but this
has no effect) and in Case 3.\ both the number of tridents and dead-ends is reduced by $1$
(the number of cherries increases by $1$). Thus, continuing to eliminate dead-ends,
we necessarily end up with a sub-tree of $[\tau_1,\tau_2]_\gamma$, which will only contain cherries and tridents.
We can now return to the beginning of the proof and the situation of a (sub-ternary) tree that consists of only tridents and cherries,
which necessarily contains a ${\rm v}$-cycle.
\end{proof}
\begin{lemma}\label{lem_trid_vs_deadends}
Let $\tau\in \mT_3\setminus \{\<0>\}$, then $|\mV_{\<10s>} (\tau)|\leqslant
|\mV_{\<3s>}(\tau)|-1$.
\end{lemma}

\begin{proof}
Clearly the statement is true for $\tau=\<3>$.
Any larger tree in $\mT_3$ can be constructed from $\<3>$ by successively
gluing tridents $\<3>$ onto leaves. Now, there are three possibilities for a
trident to be glued onto an existing tree in $\tau \in \mT_{3}$, as described
in the table below.
\begin{table}[h]
\centering
\begin{tabular}{cccc}
 Pre-gluing & Post-gluing & $|\mV_{\<10s>}|$ & $|\mV_{\<3s>}|$ \vspace{0.5em} \\
 \toprule
\vspace{0.5em}
 $ \vcenter{\hbox{
\begin{tikzpicture}[scale=0.3]
\draw[dotted, thick]  (0,0) -- (0,-.75) ;
\draw[thick]  (0,0) -- (0,1) node[dot] {};
\draw[thick]  (-.9,1)  node[dot] {} -- (0,0) node[idot] {} -- (.9,1)  node[dot] {};
\end{tikzpicture}  }}$
&
 $\vcenter{\hbox{
\begin{tikzpicture}[scale=0.3]
\draw[dotted, thick]  (0,0) -- (0,-.75) ;
\draw[thick]  (0,0) -- (0,1) node[dot] {};
\draw[thick]  (-.9,1) -- (0,0) node[idot] {} -- (.9,1) node[dot] {};
\draw[thick]  (-1.6,2) node[dot] {} -- (-.9,1)  -- (-.2,2) node[dot] {};
\draw[thick] (-.9,1) node[idot] {} -- (-.9,2) node[dot] {};
\end{tikzpicture}  }}$
& $+0$ &   $+0$ \vspace{0.5em}\\
 $\vcenter{\hbox{
\begin{tikzpicture}[scale=0.3]
\draw[dotted, thick]  (0,0) -- (0,-.75) ;
\draw[thick]  (0,0) -- (0,1) node[dot] {};
\draw[thick]  (-.9,1)  node[dot] {} -- (0,0) node[idot] {} -- (.9,1)  node[] {};
\draw (1.2,1.2) node {\scriptsize{$ \sigma_1$}};
\end{tikzpicture}  }}$
&
 $\vcenter{\hbox{
\begin{tikzpicture}[scale=0.3]
\draw[dotted, thick]  (0,0) -- (0,-.75) ;
\draw[thick]  (0,0) -- (0,1) node[dot] {};
\draw[thick]  (-.9,1) -- (0,0) node[idot] {} -- (.9,1);
\draw[thick]  (-1.6,2) node[dot] {} -- (-.9,1)  -- (-.2,2) node[dot] {};
\draw[thick] (-.9,1) node[idot] {} -- (-.9,2) node[dot] {};
\draw (1.2,1.2) node {\scriptsize{$ \sigma_1$}};
\end{tikzpicture}  }}$
  & $+1$   & $+1$ \vspace{1.5em}\\
$\vcenter{\hbox{
\begin{tikzpicture}[scale=0.3]
\draw[dotted, thick]  (0,0) -- (0,-.75) ;
\draw[thick]  (0,0) -- (0,1);
\draw[thick]  (-.9,1)  node[dot] {} -- (0,0) node[idot] {} -- (.9,1);
\draw (0,1.2) node {\scriptsize{$ \sigma_1$}};
\draw (1.2,1.2) node {\scriptsize{$ \sigma_2$}};
\end{tikzpicture}  }}$
 &
 $\vcenter{\hbox{
\begin{tikzpicture}[scale=0.3]
\draw[dotted, thick]  (0,0) -- (0,-.75) ;
\draw[thick]  (0,0) -- (0,1);
\draw[thick]  (-.9,1) -- (0,0) node[idot] {} -- (.9,1);
\draw[thick]  (-1.6,2) node[dot] {} -- (-.9,1)  -- (-.2,2) node[dot] {};
\draw[thick] (-.9,1) node[idot] {} -- (-.9,2) node[dot] {};
\draw (0,1.2) node {\scriptsize{$ \sigma_1$}};
\draw (1.2,1.2) node {\scriptsize{$ \sigma_2$}};
\end{tikzpicture}  }}$
 & $-1$&  $+1$ \vspace{0.5em} \\
\bottomrule
\end{tabular}
\end{table}
Here, $ \sigma_{1}, \sigma_{2} \in \mT_{3} \setminus \{ \<0>\}$ are placeholders for corresponding
sub-trees. 
In all three cases the claimed inequality remains true as we can only create a new dead-end
by creating a trident at the same time.
\end{proof}

\subsection{Extracting the tree simplex}

In the proof of Lemma~\ref{lem_unif_bound_contraction_perm}, we eventually
integrate time over indicator functions which we collected from the estimate
\eqref{removal_estimate}. 
The following lemma states that the restrictions imposed by these indicator
functions agree with the corresponding tree-simplex \eqref{timesimplex}.  
In particular, the cycle removal estimate in
Lemma~\ref{lem_unif_bound_contraction_perm} is independent of the chosen
pairing.

\begin{lemma}\label{lem_2timesimplex}
Let $ t>0$, $ \tau \in \mT_{3}$ and write $ \mI := \mI ( [ \tau, \tau]) \setminus \mf{o}$.
Then for every pairing $ \gamma \in \mY ( \tau, \tau)$, we have
\begin{equation}\label{e_2timesimplex}
\begin{aligned}
[0,t]^{\mI}\cap
\bigcap_{i=1}^{K( \tau, \gamma)}
\bigcap_{v \in \mI_{\mC_{i}}} \{  s_{\mf{d}_{ \sigma_{i}, \gamma} (v)}
\leqslant  s_v \leqslant s_{ \mathfrak{p}_{ \sigma_{i}}(v)}  \}
=
\mathfrak{D}_{ [ \tau, \tau]}(t) \,,
\end{aligned}
\end{equation}where $ ( \mC_{k} )_{k =1}^{K (\tau, \gamma)} $ denotes  the sequence of ${\rm v}$-cycles
and $
(\sigma_{i})_{i=1}^{K ( \tau, \gamma)}$ the sequence of reduced trees,
 constructed from $[\tau,\tau]_\gamma$ via the cycle extraction map (Definition
\ref{def:cycle-extraction}).
The set $\mathfrak{D}_{[ \tau, \tau]}(t)$ was defined in \eqref{timesimplex}.
\end{lemma}

\begin{proof}
Let $v \in \mI$ be arbitrary. Then there exists an $i =1,\ldots, K ( \tau, \gamma
)$ such that $ v \in \mI_{\mC_{i}}$ and we define $u_{1}:= \mf{d}_{ \sigma_{i},
\gamma}(v)$, $ w_{1}:=
\mf{p}_{ \sigma_{i}}(v) $.
Notably, there  exists a unique path  $(u_{1},\ldots,  u_{m}, v, w_{m'} ,\ldots,
w_{1})$ in the tree $[ \tau, \tau]$ with $u_{j},w_{j} \in \mI$, $j \geqslant 2$.

If $s_{\mI} \in \mathfrak{D}_{ [ \tau, \tau]}(t) $, then
\begin{equation*}
\begin{aligned}
0 \leqslant s_{ \mf{d}_{ \sigma_{i}, \gamma}(v)}=s_{u_{1}} \leqslant \cdots \leqslant
s_{u_{m}} \leqslant s_{v} \leqslant s_{w_{m'}} \leqslant \cdots \leqslant
s_{w_{1}}= s_{\mf{p}_{ \sigma_{i}}(v)} \leqslant t \,,
\end{aligned}
\end{equation*}
which implies in particular that $ s_{ \mf{d}_{ \sigma_{i}, \gamma}(v)}
\leqslant s_{v} \leqslant s_{\mf{p}_{ \sigma_{i}}(v)} $.

On the other hand, the node $ \mf{p} (v)= \mf{p}_{[ \tau, \tau]}(v) = w_{m'}$ lies in $\mI_{\mC_{i'}}$ for some $ i'=1,\ldots, K ( \tau , \gamma)$.
Thus, if $s_{\mI}$ lies in the left-hand side of
\eqref{e_2timesimplex}, then
\begin{itemize}
\item either the parent of $v$ has not been removed by the cycle extraction map in an
earlier iteration, i.e. $i' \geqslant i$, in which case
\begin{equation*}
\begin{aligned}
s_{v} \leqslant s_{\mf{p}_{ \sigma_{i}} (v)}= s_{w_{m'}} = s_{\mf{p} (v)} \,,
\end{aligned}
\end{equation*}

\item or the parent of $v$ has been removed in a previous iteration, i.e.\ $
i' < i$, then
\begin{equation*}
\begin{aligned}
s_{v}= s_{\mf{d}_{ \sigma_{i'}, \gamma}(w_{m'})} \leqslant s_{w_{m'}} =
s_{\mf{p} (v)}\,.
\end{aligned}
\end{equation*}
\end{itemize}
As the choice of $v$ was arbitrary, this concludes the proof.
\end{proof}

\bibliography{cites}
\bibliographystyle{alpha}

\end{document}